\pdfoutput=1
\RequirePackage{ifpdf}
\ifpdf 
\documentclass[pdftex]{sigma}
\else
\documentclass{sigma}
\fi

\numberwithin{equation}{section}

\newtheorem{Theorem}{Theorem}[section]
\newtheorem{Corollary}[Theorem]{Corollary}
\newtheorem{Lemma}[Theorem]{Lemma}
\newtheorem{Proposition}[Theorem]{Proposition}
\newtheorem{Conjecture}[Theorem]{Conjecture}
 { \theoremstyle{definition}
\newtheorem{Definition}[Theorem]{Definition}

 }

\begin{document}
\allowdisplaybreaks

\newcommand{\arXivNumber}{2108.13883}

\renewcommand{\PaperNumber}{072}

\FirstPageHeading

\ShortArticleName{Quadratic Relations of the Deformed $W$-Algebra for the Twisted Affine Lie Algebra}

\ArticleName{Quadratic Relations of the Deformed $\boldsymbol{W}$-Algebra\\ for the Twisted Affine Lie Algebra of Type $\boldsymbol{A_{2N}^{(2)}}$}

\Author{Takeo KOJIMA}

\AuthorNameForHeading{T.~Kojima}

\Address{Department of Mathematics and Physics, Faculty of Engineering, Yamagata University,\\
Jonan 4-3-16, Yonezawa 992-8510, Japan}
\Email{\href{mailto:kojima@yz.yamagata-u.ac.jp}{kojima@yz.yamagata-u.ac.jp}}
\URLaddress{\url{http://bt.yz.yamagata-u.ac.jp/mathematics/kojima/kojima.html}}

\ArticleDates{Received December 15, 2021, in final form September 09, 2022; Published online October 04, 2022}

\Abstract{We revisit the free field construction of the deformed $W$-algebra by Frenkel and Reshetikhin [{\it Comm. Math. Phys.} {\bf 197} (1998), 1--32], where the basic $W$-current has been identified. Herein, we establish a free field construction of higher $W$-currents of the deformed $W$-algebra associated with the twisted affine Lie algebra $A_{2N}^{(2)}$. We obtain a closed set of quadrat\-ic relations and duality, which allows us to define deformed $W$-algebra ${\mathcal W}_{x,r}\big(A_{2N}^{(2)}\big)$ using generators and relations.}


\Classification{81R10; 81R12; 81R50; 81T40; 81U15}

\begin{flushright}
\begin{minipage}{85mm}
\it This paper is dedicated to Professor Michio Jimbo\\
on the occasion of his 70th anniversary
\end{minipage}
\end{flushright}

\section{Introduction}

The deformed $W$-algebra ${\mathcal W}_{x,r}({\mathfrak g})$ is a two-parameter
deformation of the classical $W$-al\-ge\-bra~${\mathcal W}({\mathfrak g})$.
The deformation theory of the $W$-algebra has been studied in
{papers}
\cite{
Awata-Kubo-Odake-Shiraishi,Brazhnikov-Lukyanov,Ding-Feigin,
Feigin-Frenkel,Feigin-Jimbo-Mukhin-Vilkoviskiy,
Frenkel-Reshetikhin1,
Harada-Matsuo-Noshita-Watanabe,
Kojima2,Kojima1,
Odake,Sevostyanov,
Shiraishi-Kubo-Awata-Odake}.
For instance, free field constructions of the basic $W$-current $T_1(z)$ of~${\mathcal W}_{x,r}({\mathfrak g})$ were
{suggested in the case when the underlying Lie algebra is of classical type.}
However, in comparison with the conformal case,
the deformation theory of $W$-algebras
is still not fully developed and understood.
Moreover, finding quadratic relations of the deformed
$W$-algebra~${\mathcal W}_{x,r}({\mathfrak g})$ is still an unresolved problem.

In this {paper}, we generalize the study for
${\mathcal W}_{x,r}\bigl(A_2^{(2)}\bigr)$\footnote{We use two types of symbols,
${\mathcal W}_{x,r}({\mathfrak g})$ and
${\mathcal W}_{x,r}(X_n^{(r)})$,
for the deformed $W$-algebra associated with the affine Lie algebra
${\mathfrak g}$ of type $X_n^{(r)}$.}
by Brazhnikov and Lukyanov \cite{Brazhnikov-Lukyanov}.
They obtained
a quadratic relation for
the $W$-current $T_1(z)$ of
the deformed $W$-algebra ${\mathcal W}_{x,r}\bigl(A_2^{(2)}\bigr)$
\begin{gather*}
f\biggl(\frac{z_2}{z_1}\biggr)T_1(z_1)T_1(z_2)-
f\biggl(\frac{z_1}{z_2}\biggr)T_1(z_2)T_1(z_1)
\\ \qquad
{}=\delta\biggl(\frac{x^{-2}z_2}{z_1}\biggr)T_1(x^{-1}z_2)
-\delta\biggl(\frac{x^{2}z_2}{z_1}\biggr)T_1(xz_2)+c \biggl(
\delta\biggl(\frac{x^{-3}z_2}{z_1}\biggr)
-\delta\biggl(\frac{x^{3}z_2}{z_1}\biggr)\biggr)
\end{gather*}
with an appropriate constant $c$ and a function $f(z)$.
This study aims to generalize the result for the cases
$A_2^{(2)}$ to $A_{2N}^{(2)}$.
We introduce higher
$W$-currents $T_i(z)${, $1 \leq i \leq 2N$,}
by fusion of
the free field construction of the basic $W$-current $T_1(z)$
of ${\mathcal W}_{x,r}\bigl(A_{2N}^{(2)}\bigr)$ \cite{Frenkel-Reshetikhin1}
(see formula (\ref{def:Ti(z)})).
We~obtain a closed set of quadratic relations for the $W$-currents $T_i(z)$,
which is completely different from those
{in the case of deformed $W$-algebras associated with affine Lie algebras of types} $A_N^{(1)}$ and $A(M,N)^{(1)}$
(see formula (\ref{thm:quadratic})).
{We refer the reader to references \cite{Leur1, Leur2}
for the affine Lie superalgebra notation.}
We obtain the duality $T_{2N+1-i}(z)=c_i T_{i}(z)$
with $1 \leq i \leq N$, which
is a new {phenomenon} that does not occur in
the case of deformed $W$-algebras associated with
affine Lie algebras of types
$A_2^{(2)}$, $A_N^{(1)}$, and $A(M,N)^{(1)}$
(see formula (\ref{prop:duality})).
This allows us to define ${\mathcal W}_{x,r}\bigl(A_{2N}^{(2)}\bigr)$ using
generators and relations.
We believe
that this paper presents a key step toward extending our construction for general affine Lie algebras
${\mathfrak g}$, because the structures of
the free field construction of the basic $W$-current
$T_1(z)$ for the affine algebras other than {that of type}
$A_N^{(1)}$
are quite similar to those of {type} $A_{2N}^{(2)}$, not $A_N^{(1)}$.
We have checked that there are similar quadratic relations as
those {for type $A_{2N}^{(2)}$} in the case
of {type} $B_N^{(1)}$ with small rank~$N$.

The remainder of this paper is organized as follows.
In Section~\ref{section:2}, we review
the free field construction of the basic $W$-current $T_1(z)$ of
the deformed $W$-algebra ${\mathcal W}_{x,r}\bigl(A_{2N}^{(2)}\bigr)$~\cite{Frenkel-Reshetikhin1}.
In Section~\ref{section:3},
we introduce higher $W$-currents $T_i(z)$ and present a closed set of quadratic relations and duality.
We also obtain the $q$-Poisson algebra in the classical limit.
In Section~\ref{section:4}, we establish
proofs of Proposition~\ref{prop:3-1} and Theorem~\ref{thm:3-2}.
Section~\ref{section:5} is devoted to discussion.
In Appendices~\ref{appendix:normal} and~\ref{appendix:exchange},
we summarize normal ordering rules.

\section{Free field construction}
\label{section:2}

In this section, we define notation and review the free field construction
of the basic $W$-current $T_1(z)$ of ${\mathcal W}_{x,r}\bigl(A_{2N}^{(2)}\bigr)$.
Throughout this paper, we fix a natural number $N=1,2,3,\dots$,
a real number $r>1$, and a complex number $x$ with $0<|x|<1$.

\subsection{Notation}

In this section, we use complex numbers $a$, $w$, $q$, and $p$
with $w\neq 0$, $q \neq 0,\pm1$, and $|p|<1$. For any integer $n$, we define $q$-integers
\begin{gather*}
[n]_q=\frac{q^n-q^{-n}}{q-q^{-1}}.
\end{gather*}
We {use} symbols for infinite products,
\begin{gather*}
(a;p)_\infty=\prod_{k=0}^\infty (1-a p^k),\qquad
(a_1,a_2, \dots, a_N; p)_\infty=\prod_{i=1}^N (a_i; p)_\infty
\end{gather*}
for complex numbers $a_1, a_2, \dots, a_N$.
The following standard formulas are used,
\begin{gather*}
\exp\Biggl({-}\sum_{m=1}^\infty \frac{1}{m}a^m \Biggr)=1-a,\qquad
\exp\Biggl({-}\sum_{m=1}^\infty \frac{1}{m}\frac{a^m}{1-p^m}\Biggr)=(a;p)_\infty.
\end{gather*}
We use the elliptic theta function $\Theta_p(w)$ and the compact notation $\Theta_p(w_1,w_2, \dots, w_N)$,
\begin{gather*}
\Theta_p(w)=\big(p, w, p w^{-1};p\big)_\infty,\qquad
\Theta_p (w_1, w_2, \dots, w_N)=\prod_{i=1}^N \Theta_p(w_i)
\end{gather*}
for complex numbers $w_1, w_2, \dots, w_N \neq 0$. Define
$\delta(z)$ by the formal series
\begin{gather*}
\delta(z)=\sum_{m \in \mathbb{Z}}z^m.
\end{gather*}

\subsection[Twisted affine Lie algebra of type \protect{$A\_\{2N\}\textasciicircum{}\{(2)\}}$]
{Twisted affine Lie algebra of type $\boldsymbol{A_{2N}^{(2)}}$}

In this section we recall the definition of the twisted affine Lie algebra of type
$A_{2N}^{(2)}$, $N=1,2,3, \dots$, in~\cite{Kac}.
The Dynkin diagram of type $A_{2N}^{(2)}$
is given by
\begin{figure}[h!]
\centering\includegraphics[scale=1.0]{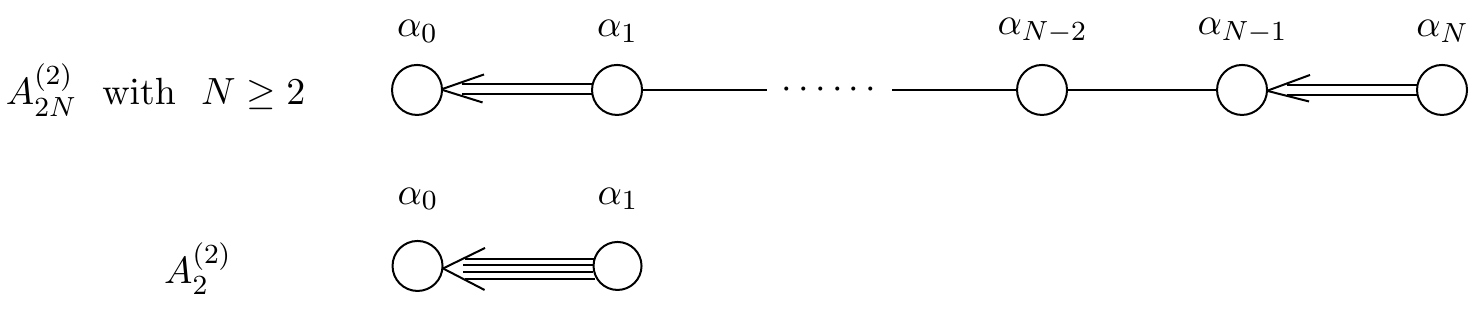}
\end{figure}%

\noindent
The corresponding Cartan matrix
$A=(A_{i,j})_{i, j=0}^N$ of type $A_{2N}^{(2)}$
is given by
\begin{gather*}
A=\begin{pmatrix}
\hphantom{-}2&-2&\hphantom{-}0&\cdots&\cdots&\cdots&\hphantom{-}0\\
-1&\hphantom{-}2&-1&\ddots&&&\vdots\\
\hphantom{-}0&-1&\hphantom{-}2&\ddots&\ddots&&\vdots\\
\hphantom{-}\vdots&\ddots&\ddots&\ddots&\ddots&\ddots&\vdots\\
\hphantom{-}\vdots&&\ddots&\ddots&\hphantom{-}2&-1&\hphantom{-}0\\
\hphantom{-}\vdots&&&\ddots&-1&\hphantom{-}2&-2\\
\hphantom{-}0&\cdots&\cdots&\cdots&\hphantom{-}0&-1&\hphantom{-}2
\end{pmatrix}
\end{gather*}
with $N \geq 2$, and
\begin{gather*}
A=\begin{pmatrix}
\hphantom{-}2&-4\\
-1&\hphantom{-}2
\end{pmatrix}
\end{gather*}
with $N=1$. We set the labels $a_i=2$, $0\leq i \leq N-1$, $a_N=1$, and the co-labels $a_0^{\vee}=1$,
$a_i^{\vee}=2$, $1\leq i \leq N$. We set $D=\operatorname{diag}\big(a_0 {a_0^{\vee}}^{-1}, a_1 {a_1^{\vee}}^{-1},\dots, a_N {a_N^{\vee}}^{-1}\big)$.
We obtain $A=DB$, where~$B$ is a~symmetric matrix. Thus, the Cartan matrix~$A$ is symmetrizable.
Let ${\mathfrak h}$ be an $(N+2)$-dimensional vector space over~${\mathbf C}$.
Let $\{h_0, h_1,\dots, h_N, d\}$ be a basis of ${\mathfrak h}$, and
$\{\alpha_0, \alpha_1,\dots, \alpha_N, \Lambda_0\}$ a basis of
${\mathfrak h}^*={\rm Hom}_{{\mathbf C}}({\mathfrak h}, {\mathbf C})$
such that we have with respect to pairing
$\langle \cdot, \cdot \rangle\colon{\mathfrak h}\times{\mathfrak h}^* \rightarrow {\mathbf C}$
\begin{gather*}
\langle h_i, \alpha_j\rangle=A_{i,j}, \qquad
0 \leq i, j \leq N,\qquad
\langle d, \alpha_i \rangle=\delta_{0,i},
\\
\langle h_i, \Lambda_0\rangle=\delta_{i,0}, \qquad \,
0\leq i \leq N, \qquad\quad
\langle d,\Lambda_0\rangle=0.
\end{gather*}

Let ${\mathfrak g}(A)$
be the affine Lie algebra associated with the Cartan matrix $A$.
Since $A$ is symmetrizable, it is defined as the Lie algebra generated by
$e_i$, $f_i$, $0\leq i \leq N$, and ${\mathfrak h}$ with the following relations:
\begin{gather*}
[e_i, f_j]=\delta_{i,j} h_i, \quad 0\leq i,j \leq N,\qquad
[h,h']=0, \quad h,h' \in {\mathfrak h},\qquad
[h, e_i]=\langle h, \alpha_i\rangle e_i,
\\
[h, f_i]=-\langle h, \alpha_i\rangle f_i,\quad h \in {\mathfrak h}, \quad 0\leq i \leq N,
\\
(\operatorname{ad}e_i)^{-A_{i,j}+1}e_j=0,\qquad
(\operatorname{ad}f_i)^{-A_{i,j}+1}f_j=0, \quad
0\leq i, j \leq N, \quad i \neq j.
\end{gather*}
Here we used the adjoint action $(\operatorname{ad}x) y=[x,y]$.

\subsection{Free field construction}

In this section, we recall the free field construction of
the basic $W$-current $T_1(z)$
and {of} the screening operators $S_i$
of the deformed $W$-algebra ${\mathcal W}_{x,r}\big(A_{2N}^{(2)}\big)$ introduced
by Frenkel and Reshe\-ti\-khin~\cite{Frenkel-Reshetikhin1}.

First, we define the $N \times N$ symmetric matrix
$B(m)=(B_{i,j}(m))_{i,j=1}^N$, $m \in \mathbb{Z}$,
associated with~$A_{2N}^{(2)}$, $N=1,2,3, \dots$,
as follows:
\begin{gather*}
B_{i,j}(m)=
\begin{cases}
\dfrac{[2m]_x}{[m]_x},& 1\leq i, j \leq N-1,\quad i=j,
\\[2ex]
\dfrac{[2m]_x-[m]_x}{[m]_x}, & i=j=N,
\\
-1,& |i-j|=1,
\\
\hphantom{-}0,& |i-j|\geq 2,
\end{cases}
\\
B_{i,j}(0)=\begin{cases}
\hphantom{-}2,& 1\leq i, j \leq N-1,\quad  i=j,\\
\hphantom{-}1,& i=j=N,\\
-1,& {|i-j|=1,}\\
\hphantom{-}0,& |i-j|\geq 2.
\end{cases}
\end{gather*}
We introduce the Heisenberg algebra ${\mathcal H}_{x,r}$
with generators $a_i(m)$, $Q_i$, $m \in \mathbb{Z}$, $1\leq i \leq N$, satisfying
\begin{gather*}
[a_i(m),a_j(n)]=\frac{1}{m}[rm]_x[(r-1)m]_x
B_{i,j}(m)\big(x-x^{-1}\big)^2 \delta_{m+n,0},\quad\
m, n\neq 0, \quad 1\leq i, j \leq N,
\\
[a_i(0),Q_j]=B_{i,j}(0),\qquad 1\leq i,j \leq N.
\end{gather*}
The remaining commutators vanish. The generators $a_i(m)$, $Q_i$ are ``root'' type generators of
${\mathcal H}_{x,r}$. There is a unique set of ``fundamental weight'' type generators
$y_i(m)$, $Q_i^y$, $m \in \mathbb{Z}$, $1\leq i \leq N$, which
satisfy the following relations
\begin{gather*}
[y_i(m), a_j(n)]=\frac{1}{m}[rm]_x[(r-1)m]_x\big(x-x^{-1}\big)^2 \delta_{i,j} \delta_{m+n,0},\qquad
m, n \neq 0, \quad 1\leq i, j \leq N,
\\
[y_i(0),Q_j]=\delta_{i,j},\qquad
[a_i(0), Q_j^y]=\delta_{i,j},\qquad
[y_i(0),a_j(m)]=0,\qquad
m \in \mathbb{Z},\quad 1\leq i, j \leq N.
\end{gather*}
The explicit formulas for $y_i(m)$ and $Q_j^y$ are given in (\ref{eqn:y(m)}).
We use the normal ordering ${:}~{:}$ {on~${\mathcal H}_{x,r}$} that satisfies
\begin{gather*}
{:}a_i(m)a_j(n){:}= \begin{cases}
a_i(m)a_j(n),& m<0,
\\
a_j(n)a_i(m),& m \geq 0,
\end{cases} \qquad
m,n \in \mathbb{Z}, \quad 1\leq i, j \leq N.
\end{gather*}

Let $|0\rangle \neq 0$ be the Fock vacuum of the Fock space
of ${\mathcal H}_{x,r}$ such that $a_i(m)|0\rangle=0$, $m \geq 0$, $1\leq i \leq N$.
Let $\pi_\lambda$ be the Fock space of ${\mathcal H}_{x,r}$
generated by $|\lambda \rangle={\rm e}^{\lambda}|0\rangle$,
$\lambda=\sum_{j=1}^N\lambda_j Q_j^y$. We obtain
\begin{gather}
a_i(0)|\lambda\rangle=\lambda_i|\lambda\rangle,\qquad
a_i(m)|\lambda \rangle=0, \qquad
m>0,\quad 1\leq i \leq N.
\label{def:Fock}
\end{gather}
We work in the Fock space $\pi_\lambda$ of the Heisenberg algebra ${\mathcal H}_{x,r}$.
Let the vertex operators $A_i(z)$, $Y_i(z)$, and $S_i(z)$, $1\leq i \leq N$, be
\begin{gather}
A_i(z)=x^{r a_i(0)}{:}\exp\Biggl(\sum_{m \neq 0}a_i(m)z^{-m}\Biggr){:},
\label{def:A}
\\
Y_i(z)=x^{r y_i(0)}{:}\exp\Biggl(\sum_{m \neq 0}y_i(m)z^{-m}\Biggr){:},
\label{def:Y}
\\
S_i(z)=z^{\frac{r-1}{2r}B_{i,i}(0)}{\rm e}^{-\sqrt{\frac{r-1}{r}}Q_i}z^{-\sqrt{\frac{r-1}{r}}a_i(0)}
{:}\exp\Biggl(\sum_{m \neq 0}\frac{a_i(m)}{x^{rm}-x^{-rm}}z^{-m}\Biggr){:}.
\label{def:S}
\end{gather}
The main parts of (\ref{def:A}), (\ref{def:Y}), and (\ref{def:S})
are the same as those of~\cite{Frenkel-Reshetikhin1}.
We corrected the misprints in the formulas for $A_i(z)$, $Y_i(z)$, and $S_i(z)$ in~\cite{Frenkel-Reshetikhin1} by multiplying (\ref{def:A}) and (\ref{def:Y}) by constants and multiplying (\ref{def:S}) by $z^{\frac{r-1}{2r}B_{i,i}(0)}$.
With our fine-tuning, both (\ref{prop:duality}) and (\ref{eqn:screening}) hold.

Let $J_N=\{1,2,\dots,N,0,\overline{N},\dots,\overline{2},\overline{1}\}$.
Here, the indices are ordered as
\begin{gather*}
1 \prec 2 \prec \cdots \prec N \prec 0
\prec \overline{N} \prec \cdots \prec \overline{2} \prec \overline{1}.
\end{gather*}
Let $\overline{\overline{k}}=k$, $k=1,2,\dots, N$, and $\overline{0}=0$.
The indices $i, j \in J_N$ satisfy $i \prec j$ if and only if $\overline{j} \prec \overline{i}$.
We~define $\overline{I}=\{\overline{i_1}, \overline{i_2}, \dots, \overline{i_k}\}$ for a subset
$I \subset J_N$, $I=\{i_1, i_2, \dots, i_k\}$.
Let $T_1(z)$ be the generating series with operator valued coefficients acting on
the Fock space $\pi_\lambda$,
\begin{gather*}
T_1(z)=\sum_{i \in J_N}\Lambda_i(z),
\end{gather*}
where
\begin{gather}
\Lambda_1(z)=Y_1(z),\qquad
{\Lambda}_k(z)={:}\Lambda_{k-1}(z)A_{k-1}\big(x^{-k+1}z\big)^{-1}{:},\quad
2 \leq k \leq N,\notag
\\
\Lambda_0(z)=\frac{[r-\frac{1}{2}]_x}{[\frac{1}{2}]_x}{:}\Lambda_N(z)A_N\big(x^{-N}z\big)^{-1}{:},\notag
\\
\Lambda_{\overline{N}}(z)=\frac{[\frac{1}{2}]_x}{[r-\frac{1}{2}]_x}
{:}\Lambda_0(z)A_N\big(x^{-N-1}z\big)^{-1}{:},\notag
\\
\Lambda_{\overline{k}}(z)={:}\Lambda_{\overline{k+1}}(z)A_k\big(x^{-2N+k-1}z\big)^{-1}{:},\qquad
1\leq k \leq N-1.
\label{def:Lambda}
\end{gather}
We call $T_1(z)$ the basic $W$-current of the deformed $W$-algebra
${\mathcal W}_{x,r}\big(A_{2N}^{(2)}\big)$.

Let $\pi_\mu$ be the Fock space of ${\mathcal H}_{x,r}$ generated by $|\mu \rangle={\rm e}^\mu|0\rangle$ with
$\mu=\sum_{i=1}^N \mu_i Q_i^y$, where we choose
$\mu_i \in\frac{1}{2}\sqrt{\frac{r-1}{r}}B_{i,i}(0)+\sqrt{\frac{r}{r-1}}\mathbb{Z}$, $1\leq i \leq N$.
From~(\ref{def:Fock}) and~(\ref{def:S}) the power of $w$ in $S_i(w)$, $w^{\frac{r-1}{2r}B_{i,i}(0)}w^{-\sqrt{\frac{r-1}{r}}a_i(0)}$, takes values in integers on $\pi_\mu$.
Hence, $S_i$ is well-defined on $\pi_\mu$.
We define the screening operators $S_i$, $1\leq i \leq N$, acting on the Fock space $\pi_\mu$ as
\begin{gather}
S_i=\oint \frac{{\rm d}w}{{2\pi \sqrt{-1}w}} S_i(w).
\label{def:screening}
\end{gather}
The integral in formula (\ref{def:screening}) means the residue at zero.

\section{Quadratic relations}
\label{section:3}

In this section, we introduce the higher $W$-currents $T_i(z)$ and present
a set of quadratic relations {between}
$T_i(z)$ for the deformed $W$-algebra
${\mathcal W}_{x,r}\big(A_{2N}^{(2)}\big)$.

\subsection{Quadratic relations}

We define the formal series
$\Delta(z)\in {\mathbf C}[[z]]$ and the constant $c(x,r)$ as
\begin{gather*}
\Delta(z)=\frac{\big(1-x^{2r-1}z\big)\big(1-x^{-2r+1}z\big)}{(1-xz)\big(1-x^{-1}z\big)},\qquad
c(x,r)=[r]_x[r-1]_x\big(x-x^{-1}\big).
\end{gather*}
The formal series $\Delta(z)$ satisfies
\begin{gather*}
\Delta(z)-\Delta\big(z^{-1}\big)=c(x,r)\big(\delta\big(x^{-1}z\big)-\delta(xz)\big),
\\
\Delta(z)\Delta(x^sz)-\Delta\big(z^{-1}\big)\Delta\big(x^{-s}z^{-1}\big)
\\ \qquad\!
{}=c(x,r)\big\{\Delta\big(x^{s+1}\big)\big(\delta\big(x^{-1}z\big)-\delta(x^{s+1}z)\big)
+\Delta\big(x^{s-1}\big)(\delta\big(x^{s-1}z\big)-\delta(xz))\big\},\quad
s \neq 0, \pm2.
\end{gather*}
We define the structure functions $f_{i,j}(z)$, $i, j=0,1,2,\dots$, as
\begin{align}
f_{i,j}(z)={} &\exp\Biggl({-}\sum_{m=1}^\infty
\frac{1}{m}[(r-1)m]_x[rm]_x\big(x-x^{-1}\big)^2\notag
\\
&
{} \times\frac{[\operatorname{Min}(i,j)m]_x
\big([(N+1-\operatorname{Max}(i,j))m]_x-[(N-\operatorname{Max}(i,j))m]_x\big)
}{[m]_x\big([(N+1)m]_x-[Nm]_x\big)}z^m\Biggr).
\label{def:fij}
\end{align}
The ratio of the structure functions $f_{1,1}(z)$ is
\begin{gather*}
\frac{f_{1,1}(z^{-1})}{f_{1,1}(z)}=-z
\frac{\Theta_{x^{4N+2}}\big(x^2z, x^{2N-1}z, x^{4N+2-2r}z, x^{4N+2r}z, x^{2N+1+2r}z,x^{2N-2r+3}z\big)}
{\Theta_{x^{4N+2}}\big(x^2/z, x^{2N-1}/z, x^{4N+2-2r}/z, x^{4N+2r}/z, x^{2N+1+2r}/z, x^{2N-2r+3}/z\big)}.
\end{gather*}

We introduce higher $W$-currents $T_i(z)$ as follows:
\begin{gather}
T_0(z)=1,\qquad
T_1(z)=\sum_{i \in J_N}\Lambda_i(z),\notag
\\
T_i(z)=\sum\limits_{\substack{\Omega_i \subset J_N\\|\Omega_i|=i}}
d_{\Omega_i}(x,r) \overrightarrow{\Lambda}_{\Omega_i}(z),\qquad
2 \leq i \leq 2N+1.
\label{def:Ti(z)}
\end{gather}
Here, for a subset $\Omega_i=\{s_1, s_2, \dots, s_i\}
\subset J_N$ with $s_1\prec s_2 \prec \cdots \prec s_i$, we set
\begin{gather*}
d_{\Omega_i}(x,r)=
\prod\limits_{\substack{1\leq p<q \leq i \\s_q=\overline{s}_p}}\Delta\big(x^{2(q-p+s_p-N-1)}\big),\qquad
d_{\varnothing}(x,r)=1,
\\
\overrightarrow{\Lambda}_{\Omega_i}(z)={:}\Lambda_{s_1}\big(x^{-i+1}z\big)\Lambda_{s_2}\big(x^{-i+3}z\big)
\cdots \Lambda_{s_i}\big(x^{i-1}z\big){:},\qquad
\overrightarrow{\Lambda}_{\varnothing}(z)=1.
\end{gather*}

\begin{Proposition}
\label{prop:3-1}
The $W$-currents $T_i(z)$ satisfy the duality
\begin{gather}
T_{2N+1-i}(z)=\frac{\big[r-\frac{1}{2}\big]_x}{\big[\frac{1}{2}\big]_x}\prod_{k=1}^{N-i}
\Delta\big(x^{2k}\big)T_i(z),\qquad
0\leq i \leq N.
\label{prop:duality}
\end{gather}
\end{Proposition}

\begin{Theorem}\label{thm:3-2}
The $W$-currents $T_i(z)$ satisfy the set of quadratic relations
\begin{gather}
f_{i,j}\bigg(\frac{z_2}{z_1}\bigg)T_i(z_1)T_j(z_2)-
f_{j,i}\bigg(\frac{z_1}{z_2}\bigg)T_j(z_2)T_i(z_1)\notag
\\ \qquad
{}=c(x,r)\sum_{k=1}^i \prod_{l=1}^{k-1}\Delta\big(x^{2l+1}\big)
\bigg(\delta\bigg(\frac{x^{-j+i-2k}z_2}{z_1}\bigg)
f_{i-k, j+k}\big(x^{j-i}\big)T_{i-k}\big(x^{k}z_1\big)T_{j+k}\big(x^{-k}z_2\big)\notag
\\ \qquad\hphantom{=}
{}-\delta\bigg(\frac{x^{j-i+2k}z_2}{z_1}\bigg)
f_{i-k, j+k}\big(x^{-j+i}\big)T_{i-k}\big(x^{-k}z_1\big)T_{j+k}\big(x^kz_2\big)\bigg)\notag
\\ \qquad\hphantom{=}
{}+ c(x,r)\prod_{l=1}^{i-1}\Delta\big(x^{2l+1}\big)
\prod_{l=N+1-j}^{N+i-j}\Delta\big(x^{2l}\big)\notag
\bigg(\delta\bigg(\frac{x^{-2N+j-i-1}z_2}{z_1}\bigg)T_{j-i}\big(x^{-i}z_2\big)
\\ \qquad\hphantom{=+ c(x,r)}
{}-\delta\bigg(\frac{x^{2N-j+i+1}z_2}{z_1}\bigg)T_{j-i}\big(x^{i}z_2\big)\bigg),\qquad
1\leq i \leq j \leq N
\label{thm:quadratic}.
\end{gather}
Here, we use $f_{i,j}(z)$ {introduced} in \eqref{def:fij}.
\end{Theorem}

In view of Proposition \ref{prop:3-1}	
and Theorem \ref{thm:3-2}, we obtain the following definition.

\begin{Definition}
\label{def:3-3}
Let $W$ be the free complex associative algebra generated by elements
$\overline{T}_i[m]$, $m\in \mathbb{Z}$, $1\leq i \leq 2N$, $I_K$ the left
ideal generated by elements $\overline{T}_i[m]$, $m \geq K \in \mathbb{N}$,
$1\leq i \leq 2N$, and
\begin{gather*}
\widehat{W}=\lim_{\leftarrow}W/I_K.
\end{gather*}
The deformed $W$-algebra
${\mathcal W}_{x,r}\bigl(A_{2N}^{(2)}\bigr)$
is the quotient of $\widehat{W}$ by the two-sided ideal generated by
the coefficients of the generating series which are the differences
of the right hand sides and of the left hand sides
of the relations (\ref{prop:duality}) and (\ref{thm:quadratic}),
where the generating series $T_i(z)$
are replaced with
$\overline{T}_i(z)=\sum_{m \in \mathbb{Z}}\overline{T}_i[m]z^{-m}$, $1\leq i \leq 2N$,
and $\overline{T}_0(z)=1$.
\end{Definition}

The justification of this definition is presented later. We compare
this definition of the deformed $W$-algebra with other definitions in Section \ref{section:5}.
\begin{Lemma}\label{lem:3-4}
The $W$-currents $T_i(z)$ commute with the screening operators $S_j$,
\begin{gather}
[T_i(z), S_j]=0,\qquad
1 \leq i \leq 2N,\quad 1\leq j \leq N.
\label{eqn:screening}
\end{gather}
\end{Lemma}
We present the proofs of Proposition \ref{prop:3-1},
Theorem \ref{thm:3-2}, and Lemma \ref{lem:3-4}
in Section \ref{section:4}.

\subsection{Classical limit}

The deformed $W$-algebra
${\mathcal W}_{x, r}\bigl({\mathfrak g}\bigr)$
yields a $q$-Poisson $W$-algebra
\cite{Frenkel-Reshetikhin2,Frenkel-Reshetikhin1,Frenkel-Reshetikhin-Semenov,Semenov-Sevostyanov}
in the classical limit.
As an application of the quadratic relations (\ref{thm:quadratic}),
we obtain a $q$-Poisson $W$-algebra of type $A_{2N}^{(2)}$.
We set parameters $q=x^{2r}$ and $\beta=(r-1)/r$.
We define the $q$-Poisson bracket $\{\cdot, \cdot\}$
by taking the classical limit $\beta \to 0$ with $q$ fixed as
\begin{gather*}
\bigl\{T_i^{{\rm PB}}[m], T_j^{{\rm PB}}[n]\bigr\}
=\lim_{\beta \to 0}\frac{1}{2 \beta \log q}[T_i[m],T_j[n]].
\end{gather*}
Here, we introduce $T_i^{\rm PB}[m]$ by
\[
T_i(z)=\sum_{m \in \mathbb{Z}}T_i[m]z^{-m} \longrightarrow T_i^{\rm PB}(z)=\sum_{m \in \mathbb{Z}}T_i^{\rm PB}[m]z^{-m},\qquad
\beta \to 0,\quad
q~\text{fixed}.
\]

The $\beta$-expansions of the structure functions are given as
\begin{gather*}
f_{i,j}(z)=1-2 \beta \log q \big(q-q^{-1}\big)\sum_{m=1}^\infty
[\operatorname{Min}(i,j)m ]_q
\\ \hphantom{f_{i,j}(z)=}
{}\times\frac{[(N+1-\operatorname{Max}(i,j))m]_q-[(N-\operatorname{Max}(i,j))m]_q}
{[(N+1)m]_q-[Nm]_q}z^m+O\big(\beta^2\big),\qquad i, j \geq 1,
\\
c(x,r)=2 \beta \log q+O\big(\beta^2\big).
\end{gather*}
As corollaries of Proposition \ref{prop:3-1} and Theorem \ref{thm:3-2}
we obtain the following.

\begin{Corollary}
For the $q$-Poisson $W$-algebra associated with affine Lie algebra
of type $A_{2N}^{(2)}$, the currents $T_i^{\rm PB}(z)$ satisfy
\begin{gather}
\big\{T_i^{\rm PB}(z_1),T_j^{\rm PB}(z_2)\big\}\notag
\\ \qquad
{}=\big(q-q^{-1}\big)C_{i,j}\bigg(\frac{z_2}{z_1}\bigg)T_i^{\rm PB}(z_1)T_j^{\rm PB}(z_2)\notag
+\!\sum_{k=1}^i\!\bigg(\delta\bigg(\frac{q^{-j+i-2k}z_2}{z_1}\bigg)
T_{i-k}^{\rm PB}\big(q^{k}z_1\big)T_{j+k}^{\rm PB}\big(q^{-k}z_2\big)
\\ \qquad\hphantom{=}
{}-\delta\bigg(\frac{q^{j-i+2k}z_2}{z_1}\bigg)
T_{i-k}^{\rm PB}\big(q^{-k}z_1\big)T_{j+k}^{\rm PB}\big(q^{k}z_2\big)\bigg)\notag
+\delta\bigg(\frac{q^{-2N+j-i-1}z_2}{z_1}\bigg)
T_{j-i}^{\rm PB}\big(q^{-i}z_2\big)
\\ \qquad\hphantom{=}
{}-\delta\bigg(\frac{q^{2N-j+i+1}z_2}{z_1}\bigg) T_{j-i}^{\rm PB}\big(q^{i}z_2\big),\qquad
1\leq i \leq j \leq N.
\label{def:q-Poisson}
\end{gather}
Here, the structure functions $C_{i,j}(z)$ are given by
\begin{gather*}
C_{i,j}(z)=\sum_{m \in \mathbb{Z}}\frac{[\rm Min(i,j)m]_q
\bigl([(N+1-\operatorname{Max}(i,j))m]_q-[(N-\operatorname{Max}(i,j))m]_q\bigr)}
{[(N+1)m]_q-[Nm]_q}z^m,
\\
1\leq i,j \leq N.
\end{gather*}
\end{Corollary}

\begin{Corollary}
The currents $T_i^{\rm PB}(z)$ satisfy the duality relations
\begin{gather}
T_{2N+1-i}^{\rm PB}(z)=T_i^{\rm PB}(z),\qquad 0\leq i \leq N.
\label{def:q-Poisson2}
\end{gather}
\end{Corollary}

\section{Proof of Theorem \ref{thm:3-2}}
\label{section:4}

In this section, we prove
Proposition \ref{prop:3-1},
Theorem \ref{thm:3-2},
and Lemma \ref{lem:3-4}.

\subsection{Proof of Proposition \ref{prop:3-1}}

\begin{Lemma}
The $\Lambda_i(z)$, $i \in J_N$, satisfy
\begin{gather}
f_{1,1}\biggl(\frac{z_2}{z_1}\biggr)\Lambda_i(z_1)\Lambda_j(z_2)=
\Delta\biggl(\frac{x^{-1}z_2}{z_1}\biggr)
{:}\Lambda_i(z_1)\Lambda_j(z_2){:},\qquad
i,j \in J_N,\quad i \prec j,\quad j \neq \bar{i},\notag
\\
f_{1,1}\biggl(\frac{z_2}{z_1}\biggr)\Lambda_j(z_1)\Lambda_i(z_2)=
\Delta\biggl(\frac{x z_2}{z_1}\biggr){:}\Lambda_j(z_1)\Lambda_i(z_2){:},\qquad
{i,j \in J_N,\quad i \prec j,\quad j \neq \bar{i}},\notag
\\
f_{1,1}\biggl(\frac{z_2}{z_1}\biggr)\Lambda_0(z_1)\Lambda_0(z_2)=
\Delta\biggl(\frac{z_2}{z_1}\biggr){:}\Lambda_0(z_1)\Lambda_0(z_2){:},\notag
\\
f_{1,1}\biggl(\frac{z_2}{z_1}\biggr)\Lambda_i(z_1)\Lambda_i(z_2)=
{:}\Lambda_i(z_1)\Lambda_i(z_2){:},\qquad
i \in J_N\setminus \{0\},
\label{eqn:Lambda}
\\
f_{1,1}\biggl(\frac{z_2}{z_1}\biggr)\Lambda_k(z_1)\Lambda_{\bar{k}}(z_2)=
\Delta\biggl(\frac{x^{-1}z_2}{z_1}\biggr)
\Delta\biggl(\frac{x^{-2N-2+2k}z_2}{z_1}\biggr)
{:}\Lambda_k(z_1)\Lambda_{\bar{k}}(z_2){:},\qquad
1\leq k \leq N,
\notag
\\
f_{1,1}\biggl(\frac{z_2}{z_1}\biggr)\Lambda_{\bar{k}}(z_1)\Lambda_{k}(z_2)=
\Delta\biggl(\frac{xz_2}{z_1}\biggr)
\Delta\biggl(\frac{x^{2N+2-2k}z_2}{z_1}\biggr)
{:}\Lambda_{\bar{k}}(z_1)\Lambda_{k}(z_2){:},\qquad
1\leq k \leq N.
\notag
\end{gather}
\end{Lemma}

\begin{proof}
Using (\ref{appendix:A}) and (\ref{appendix:Y}),
we obtain the normal ordering rules (\ref{eqn:Lambda}).
\end{proof}

\begin{Lemma}
The $\Lambda_i(z)$, $i \in J_N$, satisfy
\begin{gather}
{:}\Lambda_0(z)\Lambda_0(xz){:}=\Delta(1){:}\Lambda_N(z)\Lambda_{\bar{N}}(xz){:},
\label{eqn:fusion-Lambda1}
\\
{:}\Lambda_1(z)\Lambda_{\bar{1}}\big(x^{2N+1}z\big){:}=1,
\label{eqn:fusion-Lambda2}
\\
{:}\Lambda_k(z)\Lambda_{\bar{k}}(x^{2N-2k+3}z){:}=
{:}\Lambda_{k-1}(z)\Lambda_{\overline{k-1}}\big(x^{2N-2k+3}z\big){:},\qquad
2\leq k \leq N.
\label{eqn:fusion-Lambda3}
\end{gather}
\end{Lemma}

\begin{proof}
From (\ref{def:Lambda}), we obtain (\ref{eqn:fusion-Lambda1}) and (\ref{eqn:fusion-Lambda3}).
From (\ref{def:A}), (\ref{def:Y}) and (\ref{def:Lambda}), we obtain (\ref{eqn:fusion-Lambda2}).
\end{proof}

\begin{Lemma}
The $\Delta(z)$ and $f_{i,j}(z)$ satisfy the following fusion relations:
\begin{gather}
f_{i,j}(z)=f_{j,i}(z)=\prod_{k=1}^i f_{1,j}\big(z^{-i-1+2k}z\big),\qquad
1\leq i \leq j,
\label{eqn:fusion-f1}
\\
f_{1,i}(z)=\Bigg(\prod_{k=1}^{i-1}\Delta\big(x^{-i+2k}z\big)\Bigg)^{-1}
\prod_{k=1}^i f_{1,1}\big(x^{-i-1+2k}z\big),\qquad
i \geq 2,
\label{eqn:fusion-f2}
\\
f_{i,2N+1}(z)=\prod_{k=1}^{i}\Delta\big(x^{-i-1+2k}z\big),\qquad i \geq 1,
\label{eqn:fusion-f3}
\\
f_{i, j}(z)=f_{i, 2N+1-j}(z)=
f_{2N+1-j, i}(z)=f_{j,i}(z),\qquad
i \geq 1, 1\leq j \leq N,
\label{eqn:fusion-f3(2)}
\\
f_{1,j}(z)f_{1,j}\big(x^{2N+1}z\big)=
\Delta\big(x^j z\big)\Delta\big(x^{2N+1-j}z\big),\qquad j \geq 1,
\label{eqn:fusion-f3(3)}
\\
f_{1,i}(z)f_{j,i}\big(x^{\pm (j+1)}z\big)=\begin{cases}
f_{j+1,i}\big(x^{\pm j}z\big)\Delta\big(x^{\pm i}z\big),& 1\leq i \leq j,
\\
f_{j+1,i}\big(x^{\pm j}z\big),& 1\leq j<i,
\end{cases}
\label{eqn:fusion-f4}
\\
f_{1,i}(z)f_{1,j}\big(x^{\pm (i+j)}z\big)=
f_{1,i+j}\big(x^{\pm j}z\big)\Delta\big(x^{\pm i}z\big),\qquad
i,j \geq 1,
\label{eqn:fusion-f5}
\\
f_{1,i}(z)f_{1,j}\big(x^{\pm (i-j-2k)}z\big)=
f_{1,i-k}\big(x^{\mp k}z\big)f_{1,j+k}\big(x^{\pm (i-j-k)}z\big),\qquad
i,j,i-k,j+k \geq 1.
\label{eqn:fusion-f6}
\end{gather}
\end{Lemma}

\begin{proof}
We show (\ref{eqn:fusion-f2}) here. From the definitions, we have
\begin{gather*}
\Bigg(\prod_{k=1}^{i-1}\Delta_1\big(x^{-i+2k}z\big)\Bigg)^{-1}
\prod_{k=1}^i f_{1,1}\big(x^{-i-1+2k}z\big)
\\ \qquad
{}=\exp\Biggl(-\sum_{m=1}^\infty
\frac{1}{m}\frac{[rm]_x[(r-1)m]_x}{[(N+1)m]_x-[Nm]_x}\big(x-x^{-1}\big)^2
\Biggl\{([Nm]_x-[(N-1)m]_x)
\\ \qquad\phantom{=}
{}\times\sum_{k=1}^i x^{(-i+2k-1)m}-([(N+1)m]_x-[Nm]_x)
\sum_{k=1}^{i-1}x^{(-i+2k)m}\Biggr\}z^m\Biggr).
\end{gather*}
Using the relation
\begin{gather*}
[(a-1)m]_x\sum_{k=1}^i x^{(-i+2k-1)m}-[a m]_x\sum_{k=1}^{i-1}x^{(-i+2k)m}=[(a-i)m]_x,\qquad a=N, N+1,
\end{gather*}
we obtain $f_{1,i}(z)$ in the right hand side of the previous formula.
We obtain (\ref{eqn:fusion-f1}), (\ref{eqn:fusion-f3}), (\ref{eqn:fusion-f3(2)}), and (\ref{eqn:fusion-f3(3)})
by straightforward calculation from the definitions.
Using (\ref{eqn:fusion-f1}) and (\ref{eqn:fusion-f2}),
we obtain the relations (\ref{eqn:fusion-f4}), (\ref{eqn:fusion-f5}), and (\ref{eqn:fusion-f6}).
\end{proof}

\begin{Lemma}
The following relation holds for $A \subset J_N$:
\begin{gather}
\overrightarrow{\Lambda}_{\overline{J_N \setminus A}}(z)=
\overrightarrow{\Lambda}_{A}(z)\times
\begin{cases}
\dfrac{\big[r-\frac{1}{2}\big]_x}{\big[\frac{1}{2}\big]_x},& 0 \notin A,
\\[2.5ex]
\dfrac{\big[\frac{1}{2}\big]_x}{\big[r-\frac{1}{2}\big]_x},& 0 \in A.
\end{cases}
\label{eqn:fusion-Lambda4}
\end{gather}
\end{Lemma}
\begin{proof}
First, we {consider} the case $A=\varnothing$ and $\overline{J_N \setminus A}=J_N$.
In this case, (\ref{eqn:fusion-Lambda4}) {can be} rewritten as
\begin{gather}
{:}\Lambda_1\big(x^{-2N}z\big)\cdots \Lambda_{N}\big(x^{-2}z\big)
\Lambda_0(z)\Lambda_{\overline{N}}\big(x^2z\big)\cdots
\Lambda_{\overline{1}}\big(x^{2N}z\big){:}=\frac{\big[r-\frac{1}{2}\big]_x}{\big[\frac{1}{2}\big]_x}.
\label{eqn:fusion-Lambda5}
\end{gather}
Using (\ref{def:A}), (\ref{def:Y}), and (\ref{def:Lambda}),
the left side of (\ref{eqn:fusion-Lambda5}) can be written as
\begin{gather*}
\frac{\big[r-\frac{1}{2}\big]_x}{\big[\frac{1}{2}\big]_x}{:}\exp\Bigg(\sum_{m \neq 0}
\Bigg(\frac{[(2N+1)m]_x}{[m]_x}y_1(m)-\sum_{j=1}^N
\frac{[(2N+1-j)m]_x+[jm]_x}{[m]_x}a_j(m)\Bigg)z^{-m}\Bigg){:}.
\end{gather*}
Using the relation ${\frac{[(2N+1-j)m]_x+[jm]_x}{[(2N+1)m]_x}=
\frac{[(N+1-j)m]_x-[(N-j)m]_x}{[(N+1)m]_x-[Nm]_x}}$,
the generators $y_1(m)$ in (\ref{eqn:y(m)}) are rewritten as
${y_1(m)=\sum_{j=1}^N
\frac{[(2N+1-j)m]_x+[jm]_x}{[(2N+1)m]_x}a_j(m)}$.
Hence, we obtain (\ref{eqn:fusion-Lambda5}).

Next, we show (\ref{eqn:fusion-Lambda4}) for $A \subset J_N$.
Cases $(i)$, $0 \in A$ and~$(ii)$, $0 \notin A$ are proved separately.
First, we study case $(i)$, $0 \in A$. Let
\begin{gather*}
A=\{k_1, \dots,k_K, 0, \overline{l_L}, \dots, \overline{l_1} \mid
k_1\prec \cdots \prec k_K \prec 0 \prec\overline{l_L} \prec \cdots \prec \overline{l_1},\,
1\leq K, L \leq N\}.
\end{gather*}
Multiplying (\ref{eqn:fusion-Lambda5}) by $\overrightarrow{\Lambda_A}\big(x^{L-K+1}z\big)$
on the left, and using (\ref{eqn:Lambda}) and (\ref{eqn:fusion-f3}) yields
\begin{gather}
{:}\overrightarrow{\Lambda}_{J_N}(z)
\overrightarrow{\Lambda}_A\big(x^{L-K+1}z\big){:}=\frac{\big[r-\frac{1}{2}\big]_x}{\big[\frac{1}{2}\big]_x}
\overrightarrow{\Lambda}_A\big(x^{L-K+1}z\big).
\label{eqn:fusion-Lambda6}
\end{gather}
Using (\ref{eqn:fusion-Lambda1}), (\ref{eqn:fusion-Lambda2}) and
(\ref{eqn:fusion-Lambda3}) yields
\begin{gather*}
{:}\overrightarrow{\Lambda}_{J_N}(z)\Lambda_0(xz){:}
=\Delta(1)\overrightarrow{\Lambda}_{J_N{\setminus} \{0\}}(xz),
\\
{:}\overrightarrow{\Lambda}_{J_N{\setminus} \{0\}}(z)
\Lambda_{\overline{l}_L}\big(x^2z\big){:}
=\overrightarrow{\Lambda}_{J_N{\setminus} \{l_L, 0\}}(xz),
\\
{:}\overrightarrow{\Lambda}_{J_N{\setminus} \{l_{L-s+1},\dots, l_{L}, 0\}}(z)
\Lambda_{\overline{l}_{L-s}}\big(x^{2+s}z\big){:}
=\overrightarrow{\Lambda}_{J_N{\setminus} \{l_{L-s},\dots, l_L, 0\}}(xz),\qquad
1\leq s \leq L-1,
\\
{:}\overrightarrow{\Lambda}_{J_N{\setminus} \{l_1, \dots, l_L, 0\}}(z)
\Lambda_{{k}_K}\big(x^{-L-2}z\big){:}
=\overrightarrow{\Lambda}_{J_N{\setminus} \{l_1,\dots, l_L, 0, \overline{k}_K\}}\big(x^{-1}z\big),
\\
{:}\overrightarrow{\Lambda}_{J_N{\setminus} \{l_1,\dots, l_{L}, 0, \overline{k}_K, \dots, \overline{k}_{K-s+1}\}}(z)
\Lambda_{{k}_{K-s}}\big(x^{-L-2-s}z\big){:}
=\overrightarrow{\Lambda}_{J_N{\setminus} \{l_{1},\dots, l_L, 0, \overline{k}_K,\dots, \overline{k}_{K-s}\}}\big(x^{-1}z\big),
\\ \qquad
1\leq s \leq K-1.
\end{gather*}
Using the above five relations yields
\begin{gather*}
{:}\overrightarrow{\Lambda}_{J_N}(z)\overrightarrow{\Lambda}_A\big(x^{L-K+1}z\big){:}=
\Delta(1)\overrightarrow{\Lambda}_{\overline{J_N \setminus A}}\big(x^{L-K+1}z\big).
\end{gather*}
From (\ref{eqn:fusion-Lambda6})
we obtain (\ref{eqn:fusion-Lambda4}) for $0 \in A$.

Next, we study case $(ii)$, $0 \notin A$. The proof for this case is similar to that of case $(i)$.
Let
\begin{gather*}
A=\big\{k_1, \dots,k_K, \overline{l_L}, \dots, \overline{l_1} \mid
k_1\prec \cdots \prec k_K \prec\overline{l_L} \prec \cdots \prec \overline{l_1},\,
1\leq K, L \leq N\big\}.
\end{gather*}
Multiplying (\ref{eqn:fusion-Lambda5}) by
$\overrightarrow{\Lambda_A}\big(x^{L-K}z\big)$
on the left, and using (\ref{eqn:Lambda}) and (\ref{eqn:fusion-f3}) yields
\begin{gather}
{:}\overrightarrow{\Lambda}_{J_N}(z)
\overrightarrow{\Lambda}_A\big(x^{L-K}z\big){:}=\frac{\big[r-\frac{1}{2}\big]_x}{\big[\frac{1}{2}\big]_x}
\overrightarrow{\Lambda}_A\big(x^{L-K}z\big).
\label{eqn:fusion-Lambda7}
\end{gather}
Using (\ref{eqn:fusion-Lambda1}), (\ref{eqn:fusion-Lambda2}), and
(\ref{eqn:fusion-Lambda3}) yields
\begin{gather*}
{:}\overrightarrow{\Lambda}_{J_N}(z)\Lambda_{\overline{l}_L}(xz){:}
=\overrightarrow{\Lambda}_{J_N{\setminus} \{l_L\}}(xz),
\\
{:}\overrightarrow{\Lambda}_{J_N{\setminus}
\{l_{L-s+1},\dots, l_{L}\}}(z)\Lambda_{\overline{l}_{L-s}}\big(x^{1+s}z\big){:}
=\overrightarrow{\Lambda}_{J_N{\setminus}
\{l_{L-s},\dots, l_L\}}(xz),\qquad
1\leq s \leq L-1,
\\
{:}\overrightarrow{\Lambda}_{J_N{\setminus} \{l_1, \dots, l_L\}}(z)
\Lambda_{{k}_K}\big(x^{-L-1}z\big){:}
=\overrightarrow{\Lambda}_{J_N{\setminus} \{l_1,\dots, l_L, \overline{k}_K\}}\big(x^{-1}z\big),
\\
{:}\overrightarrow{\Lambda}_{J_N{\setminus} \{l_1,\dots, l_{L},
\overline{k}_K, \dots, \overline{k}_{K-s+1}\}}(z)
\Lambda_{{k}_{K-s}}(x^{-L-1-s}z){:}
=\overrightarrow{\Lambda}_{J_N{\setminus} \{l_1,\dots, l_L, \overline{k}_K,\dots, \overline{k}_{K-s}\}}\big(x^{-1}z\big),
\\
\qquad
1\leq s \leq K-1.
\end{gather*}
Using the above five relations yields
\begin{gather*}
{:}\overrightarrow{\Lambda}_{J_N}(z)\overrightarrow{\Lambda}_A\big(x^{L-K}z\big){:}=
\overrightarrow{\Lambda}_{\overline{J_N \setminus A}}\big(x^{L-K}z\big).
\end{gather*}
From (\ref{eqn:fusion-Lambda7}) we obtain (\ref{eqn:fusion-Lambda4}) for $0 \notin A$.
\end{proof}

\begin{Lemma}
The following relation holds for $A \subset J_N$ with $|A|\leq N$:
\begin{gather}
\frac{d_{{J_N \setminus A}}(x,r)}{d_{A}(x,r)}=\prod_{k=1}^{N-|A|}
\Delta\big(x^{2k}\big){\times}
\begin{cases}
\Delta(1),& 0 \in A,
\\
1,&0 \notin A.
\end{cases}
\label{eqn:d}
\end{gather}
\end{Lemma}

\begin{proof}
We define the map $\sigma \colon J_N \to J_{N+1}$ by
\begin{gather*}
\sigma(j)=\begin{cases}
k+1,&j=k,\quad 1\leq k \leq N,
\\
0,&j=0,
\\
\overline{k+1},&j=\overline{k},\quad 1\leq k \leq N.
\end{cases}
\end{gather*}
For $T\subset J_N$ with
$|T|\leq N$, relation (\ref{eqn:d}) is rewritten as
\begin{gather*}
\frac{d_{{\sigma(J_N \setminus T)}}(x,r)}{d_{\sigma(T)}(x,r)}=\prod_{k=1}^{N-|T|}
\Delta\big(x^{2k}\big)\times
\begin{cases}
\Delta(1),& 0 \in T,
\\
1,&0 \notin T.
\end{cases}
\end{gather*}
Hence, the relation
\begin{gather}
\frac{d_{{(J_N \setminus B)\cap \sigma(J_N)}}(x,r)}{d_{B \cap \sigma(J_N)}(x,r)}=
\prod_{k=1}^{N-|B \cap \sigma(J_N)|}
\Delta\big(x^{2k}\big)\times
\begin{cases}
\Delta(1),& 0 \in B,\\
1,& 0 \notin B,
\end{cases}
\label{eqn:d2}
\end{gather}
for $B \subset J_{N+1}$ with $|B \cap \sigma(J_N)|\leq N$ holds
if relation (\ref{eqn:d}) for $A \subset J_N$ with $|A|\leq N$ is assumed.
Here, we used $|B\cap \sigma(J_N)|=|\sigma^{-1}(B\cap \sigma(J_N))|$.

We prove (\ref{eqn:d}) by induction on $N$.
First, we establish the base $N=1$ using case-by-case analysis.
For $A=\varnothing$, we obtain $d_A(x,r)=1$ and $d_{J_N \setminus A}(x,r)=\Delta\big(x^2\big)$.
For $A=\{1\}$, we obtain $d_A(x,r)=1$ and $d_{J_N \setminus A}(x,r)=1$.
For $A=\{0\}$, we obtain $d_A(x,r)=1$ and $d_{J_N \setminus A}(x,r)=\Delta(1)$.
For $A=\{\overline{1}\}$, we obtain $d_A(x,r)=1$ and $d_{J_N \setminus A}(x,r)=1$.
This implies that (\ref{eqn:d}) holds for $N=1$.

Next, we assume that relation (\ref{eqn:d}) holds for some $N$, and show (\ref{eqn:d})
for $N$ replaced by $N+1$.
Let $A \subset J_{N+1}$.
From the definition of $d_A(x,r)$, we obtain
\begin{align}
\frac{d_{J_N \setminus A}(x,r)}{d_A(x,r)}={}&
\frac{d_{(J_N \setminus A)\cap \sigma(J_N)}(x,r)}{d_{A \cap \sigma(J_N)}(x,r)}\notag
\\
&\times\begin{cases}
1,& 1\in A,\quad \overline{1} \notin A\quad {\rm or}\quad 1\notin A,\quad \overline{1} \in A,
\\
\Delta\big(x^{2(N-|J_N \setminus A|+1)}\big)^{-1},& 1,\quad \overline{1} \in A,
\\
\Delta\big(x^{2(N-|A|+1)}\big),& 1,\quad \overline{1} \notin A.
\end{cases}
\label{eqn:d3}
\end{align}
Cases $(i)$, $1\in A$, $\overline{1} \notin A$ (or $1\notin A$, $\overline{1} \in A$),
$(ii)$, $1$, $\overline{1} \in A$, and $(iii)$, $1$, $\overline{1} \notin A$ are proved separately.

First, we study case $(i)$,
$1\in A$, $\overline{1} \notin A$ (or $1\notin A$, $\overline{1} \in A$).
In this case, we obtain $|A \cap \sigma(J_N)|=|A|-1 \leq N$.
Hence, (\ref{eqn:d2}) holds with $B=A$.
Using (\ref{eqn:d2}), (\ref{eqn:d3}) and $|A \cap \sigma(J_N)|=|A|-1$ yields
(\ref{eqn:d}) with {$N$ replaced by $N+1$}.

Next, we study case $(ii)$, $1$, $\overline{1} \in A$.
In this case, we obtain $|A \cap \sigma(J_N)|=|A|-2 \leq N-1$.
Hence, (\ref{eqn:d2}) holds with $B=A$.
Using (\ref{eqn:d2}) and (\ref{eqn:d3}), $|A \cap \sigma(J_N)|=|A|-2$ and
$|J_N \setminus A|=2N+3-|A|$ yields (\ref{eqn:d}) with {$N$ replaced by $N+1$}.

Finally, we examine case $(iii)$, $1$, $\overline{1} \notin A$.
Case $(iii)$ is further subdivided into $(iii.1)$, $|A|\leq N$, $1$, $\overline{1} \notin A$ and $(iii.2)$, $|A|=N+1$, $1$, $\overline{1} \notin A$.

For the condition $(iii.1)$, we obtain $|A \cap \sigma(J_N)|=|A|\leq N$.
Hence, (\ref{eqn:d2}) holds with $B=A$.
Using (\ref{eqn:d2}), (\ref{eqn:d3}), and $|A \cap \sigma(J_N)|=|A|$ yields
(\ref{eqn:d}) with {$N$ replaced by $N+1$}.

For condition $(iii.2)$, we obtain $|(J_N \setminus A) \cap \sigma(J_N)|=N$.
Hence, (\ref{eqn:d2}) holds with $B=J_N \setminus A$.
Using (\ref{eqn:d2}) and (\ref{eqn:d3}), $|A|=N+1$ and
$|(J_N \setminus A) \cap \sigma(J_N)|=N$ yields
(\ref{eqn:d}) with {$N$ replaced by $N+1$}.
\end{proof}

\begin{proof}
Here we will show Proposition \ref{prop:3-1}.
Using (\ref{eqn:fusion-Lambda4}), (\ref{eqn:d}), and $d_{\overline{J_N \setminus \Omega_i}}(x,r)=d_{J_N \setminus \Omega_i}(x,r)$ yields
\begin{gather}
d_{\overline{J_N \setminus \Omega_i}}(x,r)\overrightarrow{\Lambda}_{\overline{J_N \setminus \Omega_i}}(z)=\frac{\big[r-\frac{1}{2}\big]_x}{\big[\frac{1}{2}\big]_x}
\prod_{k=1}^{N-|{\Omega_i}|}\Delta\big(x^{2k}\big) d_{{\Omega_i}}(x,r)
\overrightarrow{\Lambda}_{{\Omega_i}}(z).
\label{eqn:d4}
\end{gather}
{Adding {relations} (\ref{eqn:d4}) over all $\Omega_i \subset J_N$
for each fixed $i$, $0\leq i \leq N$, yields~(\ref{prop:duality})}.
\end{proof}

\subsection{Proof of Theorem \ref{thm:3-2}}

\begin{Lemma}\label{lemma:4-6}
The $W$-currents $T_j(z)$, $1\leq j \leq N$,
satisfy the set of quadratic relations
\begin{gather}
f_{1,j}\bigg(\frac{z_2}{z_1}\bigg)T_1(z_1)T_j(z_2)-
f_{j,1}\bigg(\frac{z_1}{z_2}\bigg)T_j(z_2)T_1(z_1)\notag
\\ \qquad
{}=c(x,r)\bigg(\delta\bigg(\frac{x^{-j-1}z_2}{z_1}\bigg)T_{j+1}(x^{-1}z_2)
-\delta\bigg(\frac{x^{j+1}z_2}{z_1}\bigg)T_{j+1}(x z_2)\bigg)
+c(x,r)\Delta\big(x^{2N+2-2j}\big)\notag
\\ \qquad\phantom{=}
{}\times\!\bigg(\!\delta\bigg(\frac{x^{-2N+j-2}z_2}{z_1}\bigg)T_{j-1}\big(x^{-1}z_2\big)
-\delta\bigg(\frac{x^{2N-j+2}z_2}{z_1}\bigg)T_{j-1}(x z_2)\bigg),
\quad\!
1\leq j \leq N.\!\!
\label{eqn:quadratic2}
\end{gather}
Here, we use $f_{i,j}(z)$ {introduced} in $(\ref{def:fij})$.
\end{Lemma}
\begin{proof}
{In this proof, we frequently use exchange relations
(\ref{exchange:L1})--(\ref{exchange:L7}) in Appendix \ref{appendix:exchange}.}
We start from
\begin{gather*}
{{\rm LHS}_{1,j}=}f_{1,j}(z_2/z_1)T_1(z_1)T_j(z_2)-f_{j,1}(z_1/z_2)T_j(z_2)T_1(z_1),\qquad
1\leq j \leq N.
\end{gather*}
From the definition of $T_j(z)$ introduced in (\ref{def:Ti(z)}),
${\rm LHS}_{1,j}$ can be written {as} the sum of
\begin{gather*}
f_{1,j}(z_2/z_1)\Lambda_s(z_1)
\overrightarrow{\Lambda}_{\Omega_j}(z_2)-f_{j,1}(z_1/z_2)
\overrightarrow{\Lambda}_{\Omega_j}(z_2)\Lambda_s(z_1)\quad
{\rm over}\quad
s \in J_N,\  \Omega_j \subset J_N,\  |\Omega_j|=j,
\end{gather*}
summarized in Appendix \ref{appendix:exchange}.
Adding exchange relations (\ref{exchange:L1})--(\ref{exchange:L7})
over $s \in J_N$, $\Omega_j \subset J_N$, $|\Omega_j|=j$ yields
\begin{align}
{\rm LHS}_{1,j}={}c(x,r)
\Biggl\{&\sum_{m=0}^{[\frac{j}{2}]}
\biggl(\delta\biggl(x^{-j-1+2m}\frac{z_2}{z_1}\biggr){\overline{G}_{j+1-2m}(z_2)}
-\delta\biggl(x^{j+1-2m}\frac{z_2}{z_1}\biggr){{G}_{j+1-2m}(z_2)}\biggr)\notag
\\
&+\sum_{m=0}^{N-[\frac{j-1}{2}]}
\biggl(\delta\biggl(x^{-2N+j-2+2m}\frac{z_2}{z_1}\biggr)
{\overline{H}_{2N-j+2-2m}(z_2)}\notag
\\
&-\delta\biggl(x^{2N-j+2-2m}\frac{z_2}{z_1}\biggr)
{H_{2N-j+2-2m}(z_2)}
\biggr)\Biggr\}.
\label{eqn:L1j-2}
\end{align}
{\sloppy
Formulas for $\overline{G}_{j+1}(z)$, $G_{j+1}(z)$, $\overline{H}_{2N-j+2}(z)$, $H_{2N-j+2}(z)$,
$\overline{G}_{j+1-2m}(z)$, $G_{j+1-2m}(z)$, $\overline{H}_{2N-j+2-2m}(z)$, and $H_{2N-j+2-2m}(z)$
will be given below. In (\ref{def:H2}) we define $H_0(z)=0$ to avoid ambiguity of
$\overline{H}_0(z)$ and $H_0(z)$. In the case when $j$ {is} even, we have
${\rm LHS}_{1,j}=c(x,r)\overline{H}_0(z_2)-H_0(z_2)\delta(z_2/z_1)+
\overline{H}_2(z_2)\delta\big(x^{-2}z_2/z_1\big)-H_2(z_2)\delta\big(x^2z_2/z_1\big)+\cdots$.

}

First, we define $\overline{G}_{j+1}(z)$, $1\leq j \leq N$,
as the coefficient of
$\delta(x^{-j-1}z_2/z_1)$ in (\ref{eqn:L1j-2}).
In~what follows, for a subset $\Omega_j \subset J_N$ with $|\Omega_j|=j$,
we write its elements as $s_1, s_2, \dots, s_j$, $s_1\prec s_2 \prec \cdots \prec s_j$.
Adding the first term in (\ref{exchange:L1}) and the first term in (\ref{exchange:L6}) yields
\begin{align*}
\begin{split}
&\overline{G}_{j+1}(z)
=\sum_{\Omega_j \subset J_N}
\sum\limits_{\substack{s \in J_N\\s \prec s_1, \,\overline{s} \notin \Omega_j}}
d_{\Omega_{j}}(x,r){:}\Lambda_{s}\big(x^{-j-1}z\big)\overrightarrow{\Lambda}_{{\Omega_{j}}}(z){:}\notag
\\
&\hphantom{\overline{G}_{j+1}(z)=}
+\sum_{\Omega_j \subset J_N}\sum\limits_{\substack{n=1\\n \prec s_1}}^N
\sum\limits_{\substack{l=1\\s_l=\overline{n}}}^j\Delta\big(x^{2(N+1-l-n)}\big) d_{{\Omega_{j}}}(x,r)
{:}\Lambda_{n}\big(x^{-j-1}z\big)\overrightarrow{\Lambda}_{{\Omega_{j}}}(z){:}.
\end{split}
\end{align*}
Using
${:}\Lambda_{n}\big(x^{-j-1}z\big)\overrightarrow{\Lambda}_{{\Omega_{j}}}(z){:}
=\overrightarrow{\Lambda}_{{\Omega_{j}\cup \{n\}}}\big(x^{-1}z\big)$
and
\begin{gather*}
d_{{\Omega_{j}\cup \{n\}}}(x,r)=
d_{{\Omega_{j}}}(x,r)\times
\begin{cases}
1,& \overline{n} \notin \Omega_j,
\\
\Delta(x^{{2(N+1-l-n)}}),&\overline{n}=s_l
\end{cases}\quad
{\rm with}\ n \prec s_1,\ 1\leq n \leq N,
\end{gather*}
yields
\begin{gather*}
\overline{G}_{j+1}(z)=
\sum_{\Omega_j \subset J_N}\sum\limits_{\substack{s \in J_N
\\s\prec s_1}}d_{\Omega_j\cup \{s\}}(x,r)
\overrightarrow{\Lambda}_{\Omega_j \cup \{s\}}\big(x^{-1}z\big) .
\end{gather*}
Hence, we obtain ${\overline{G}_{j+1}(z)}=T_{j+1}\big(x^{-1}z\big)$, $1\leq j \leq N$.

Next, we define
${G}_{j+1}(z)$, $1\leq j \leq N$, as the coefficient of
$\delta\big(x^{j+1}z_2/z_1\big)$ in (\ref{eqn:L1j-2}).
Adding the second term in (\ref{exchange:L1}) and
the third term in (\ref{exchange:L7}) yields
\begin{align*}
{G}_{j+1}(z)=&\sum_{\Omega_j \subset J_N}
\sum\limits_{\substack{s \in J_N\\s_j \prec s, \overline{s} \notin \Omega_j}}
d_{\Omega_{j}}(x,r)
{:}\overrightarrow{\Lambda}_{\Omega_{j}}(z)\Lambda_s\big(x^{j+1}z\big){:}
\\
&+
\sum_{\Omega_j \subset J_N}
\sum\limits_{\substack{n=1\\s_j \prec \overline{n}}}^N
\sum\limits_{\substack{k=1\\s_k=n}}^j
\Delta\big(x^{2(N+k-j-n)}\big) d_{\Omega_{j}}(x,r)
{:}\overrightarrow{\Lambda}_{\Omega_{j}}(z)\Lambda_{\overline{n}}\big(x^{j+1}z\big){:}.
\end{align*}
Using
${:}\overrightarrow{\Lambda}_{\Omega_{j}}(z)\Lambda_{\overline{n}}\big(x^{j+1}z\big){:}
=\overrightarrow{\Lambda}_{\Omega_{j}\cup \{\overline{n}\}}(xz)$
and
\begin{gather*}
d_{{\Omega_{j}\cup \{\overline{n}\}}}(x,r)=
d_{{\Omega_{j}}}(x,r)\times
\begin{cases}
1,& n \notin \Omega_j,\\
\Delta\big(x^{{2(N+k-j-n)}}\big),&n=s_k
\end{cases}\quad
{\rm with}~s_j \prec \overline{n},\ 1\leq n \leq N,
\end{gather*}
yields
\begin{gather*}
{G}_{j+1}(z)=
\sum_{\Omega_j \subset J_N}
\sum\limits_{\substack{s \in J_N
\\s_j \prec s}}d_{\Omega_j\cup \{s\}}(x,r)
\overrightarrow{\Lambda}_{\Omega_j \cup \{s\}}(xz).
\end{gather*}
Hence, we obtain ${G_{j+1}(z)}=T_{j+1}(z)$, $1\leq j \leq N$.

We define $\overline{H}_{2N-j+2}(z), 1\leq j \leq N$, as the coefficient of
$\delta\big(x^{-2N+j-2}z_2/z_1\big)$ in (\ref{eqn:L1j-2}).
Adding the first term in (\ref{exchange:L3}),
the second term in (\ref{exchange:L3}), the second term in (\ref{exchange:L6}),
and the fourth term in (\ref{exchange:L6}) yields
\begin{align}
\overline{H}_{2N-j+2}(z)={}&\sum_{n=1}^{j-1}
\sum_{k=1}^{j-n}\sum\limits_{\substack{\Omega_j \subset J_N\\s_k=n
\\s_l=\overline{n},\, l=j-n+1}}
d_{\Omega_j}(x,r) {:}\Lambda_n\big(x^{-2N+j-2}z\big)\overrightarrow{\Lambda}_{\Omega_j}(z){:}\notag
\\
&-\sum_{n=1}^{j-2}\sum_{k=1}^{j-n-1}
\sum\limits_{\substack{\Omega_j \subset J_N\\s_k=n\\s_l=\overline{n},\, l=j-n}}
d_{\Omega_j}(x,r) {:}\Lambda_n\big(x^{-2N+j-2}z\big)
\overrightarrow{\Lambda}_{\Omega_j}(z){:}\notag
\\
&+\sum_{n=1}^{j}\sum_{k=1}^{j-n+1}\sum\limits_{\substack{\Omega_j \subset J_N
\\s_{k-1}\prec n \prec s_k\\s_l=\overline{n},\, l=j+1-n}}
\Delta\big(x^{2(N-j+k)}\big) d_{\Omega_j}(x,r) {:}\Lambda_n\big(x^{-2N+j-2}z\big)\overrightarrow{\Lambda}_{\Omega_j}(z){:}\notag
\\
&-\sum_{n=1}^{j-1}\sum_{k=1}^{j-n}\sum\limits_{\substack{\Omega_j \subset J_N
\\s_{k-1} \prec n \prec s_k\\s_l=\overline{n},\, l=j-n}}\!\!\!\!
\Delta\big(x^{2(N-j+k)}\big) d_{\Omega_j}(x,r){:}\Lambda_n\big(x^{-2N+j-2}z\big)\overrightarrow{\Lambda}_{\Omega_j}(z){:}.
\label{def:bH1}
\end{align}
The second term in (\ref{def:bH1}) vanishes, because there doesn't exist $s_j \in J_N$,
if $\Omega_j \subset J_N$, $s_k=n$, $1\leq k \leq j+1-n$, $s_l=\overline{n}$, $l=j-n$,
and $1\leq n \leq j-2$ are satisfied. The fourth term in (\ref{def:bH1}) vanishes,
because there doesn't exist $s_j \in J_N$,
if $\Omega_j \subset J_N$, $s_{k-1}\prec n \prec s_k$, $1\leq k \leq j-n$, $s_l=\overline{n}$, $l=j-n$, and $1\leq n \leq j-1$ are satisfied. Rewriting the sum of the first and the third terms yields\looseness=-1
\begin{align*}
\overline{H}_{2N-j+2}(z)={}&\sum_{n=1}^{j-1}
\sum_{k=1}^{\operatorname{Min}(j-n, n)}
\sum_{\substack{\Omega_j \subset J_N,\,s_k=n\\(s_{j-n+1}, \dots, s_{j-1}, s_j)\\=(\overline{n}, \dots,
\overline{2}, \overline{1})}}
d_{\Omega_j}(x,r) {:}\Lambda_n\big(x^{-2N+j-2}z\big)\overrightarrow{\Lambda}_{\Omega_j}(z){:}
\\[-.5ex]
&+\sum_{n=1}^{j}\!\!\sum_{k=1}^{\operatorname{Min}(j+1-n, n)}\!\!\!\!\!\!\!\!\!\!
\sum_{\substack{\Omega_j \subset J_N,\,s_{k-1}\prec n \prec s_k
\\(s_{j-n+1}, \dots, s_{j-1}, s_j)\\=(\overline{n},\dots, \overline{2}, \overline{1})}}
\!\!\!\!\!\!\!\!\!\!\Delta\big(x^{2(N-j+k)}\big) d_{\Omega_j}(x,r) {:}\Lambda_n\big(x^{-2N+j-2}z\big)\overrightarrow{\Lambda}_{\Omega_j}(z){:}.
\end{align*}
The relation $d_{\Omega_j}(x,r)=\Delta\big(x^{2(N+1-j)}\big) d_{\Omega_j \setminus \{\overline{n}\}}(x,r)$ holds, if $s_k=n$, $s_l=\overline{n}$, $1\leq n \leq j-1$, and $1\!\leq\! k\!<\!l\!=\!j+1-n$ are satisfied.
The relation $d_{\Omega_j}(x,r)\Delta\big(x^{2(N-j+k)}\big)
=\Delta\big(x^{2(N+1-j)}\big) d_{\Omega_j \setminus \{\overline{n}\}}(x,r)$ holds,
if $s_{k-1} \prec n \prec s_k$, $s_l=\overline{n}$, $1\leq n \leq j$, and
$1\leq k\leq l=j+1-n$ are satisfied. Using the above two relations and
\begin{gather*}
{:}\Lambda_n\big(x^{-2N+j-2}z\big)\overrightarrow{\Lambda}_{\Omega_j}(z){:}=
\overrightarrow{\Lambda}_{\Omega_j \setminus \{\overline{n}\}}\big(x^{-1}z\big),
\\
(s_{j+1-n}, \dots, s_{j-1}, s_j)=(\overline{n},\dots, \overline{2}, \overline{1}),\qquad
1\leq n \leq j,
\end{gather*}
obtained from (\ref{eqn:fusion-Lambda2}) and (\ref{eqn:fusion-Lambda3}),
yields
\begin{align*}
\overline{H}_{2N-j+2}(z)={}&
\Delta\big(x^{2(N-j+1)}\big)\Bigg(
\sum_{n=1}^{j-1}
\sum_{k=1}^{\operatorname{Max}(j-n, n)}\!\!\!
\sum_{\substack{\Omega_j \subset J_N
\\s_k=n,\,s_l=\overline{n},\,l=j+1-n}}\!
d_{\Omega_j \setminus \{s_l\}}(x,r)
\overrightarrow{\Lambda}_{\Omega_j \setminus \{s_l\}}\big(x^{-1}z\big)
\\[-1ex]
&+\sum_{n=1}^j \sum_{k=1}^{\operatorname{Max}(j+1-n,n)}
\sum_{\substack{\Omega_j \subset J_N
\\s_{k-1}\prec n \prec s_k\\s_l=\overline{n},\, l=j+1-n}}d_{\Omega_j \setminus \{s_l\}}(x,r)
\overrightarrow{\Lambda}_{\Omega_j \setminus \{s_l\}}\big(x^{-1}z\big)\Bigg).
\end{align*}
Hence, we obtain
${\overline{H}_{2N-j+2}(z)}=\Delta\big(x^{2(N-j+1)}\big)T_{j-1}\big(x^{-1}z\big)$,
$1\leq j \leq N$.

We define $H_{2N-j+2}(z)$, $1\leq j \leq N$, as the coefficient of
$\delta(x^{2N-j+2}z_2/z_1)$ in (\ref{eqn:L1j-2}). Adding the first term in (\ref{exchange:L4}),
the second term in (\ref{exchange:L4}), the second term in (\ref{exchange:L7}),
and the fourth term in (\ref{exchange:L7}) yields
\begin{align}
H_{2N-j+2}(z)={}&-\sum_{n=1}^{j-2}\sum_{l=n+2}^{j}
\sum_{\substack{\Omega_j \subset J_N\\s_k=n,\, k=n+1\\s_l=\overline{n}}}
d_{\Omega_j}(x,r){:}\overrightarrow{\Lambda}_{\Omega_j}(z)
\Lambda_{\overline{n}}\big(x^{2N-j+2}\big){:}\notag
\\[-1ex]
&+\sum_{n=1}^{j-1}\sum_{l=n+1}^{j}\sum_{\substack{\Omega_j \subset J_N\\
s_k=n,\, k=n\\s_l=\overline{n}}}d_{\Omega_j}(x,r)
{:}\overrightarrow{\Lambda}_{\Omega_j}(z)
\Lambda_{\overline{n}}\big(x^{2N-j+2}z\big){:}\notag
\\[-1ex]
&+\sum_{n=1}^{j}\sum_{l=n}^{j}\sum_{\substack{\Omega_j \subset J_N
\\s_k=n,\, k=n\\ s_l \prec \overline{n} \prec s_{l+1}}}
\Delta\big(x^{2(N+1-l)}\big) d_{\Omega_j}(x,r)
{:}\overrightarrow{\Lambda}_{\Omega_j}(z)
\Lambda_{\overline{n}}\big(x^{2N-j+2}z\big){:}\notag
\\[-1ex]
&-\sum_{n=1}^{j-1}\sum_{l=n+1}^{j}\!\!\!\sum_{\substack{\Omega_j \subset J_N
\\s_k=n,\, k=n+1\\s_l\prec \overline{n} \prec s_{l+1}}}
\!\!\!\Delta\big(x^{2(N+1-l)}\big) d_{\Omega_j}(x,r)
{:}\overrightarrow{\Lambda}_{\Omega_j}(z)
\Lambda_{\overline{n}}\big(x^{2N-j+2}z\big){:}.
\label{def:H1}
\end{align}
The first term in (\ref{def:H1}) vanishes,
because there doesn't exist $s_1 \in J_N$,
if $\Omega_j \subset J_N$, $s_k=n$, $k=n+1$, $s_l=\overline{n}$, $n+2\leq l \leq j$,
and $1\leq n \leq j-2$ are satisfied. The fourth term in (\ref{def:H1}) vanishes,
because there doesn't exist $s_1 \in J_N$,
if $\Omega_j \subset J_N$, $s_k=n$, $k=n+1$,
$s_l \prec \overline{n} \prec s_{l+1}$, $n+1\leq l \leq j$, and $1\leq n \leq j-1$ are satisfied.
Rewriting the sum of the second and the third terms yields
\begin{align*}
{H}_{2N-j+2}(z)={}&\sum_{n=1}^{j-1}\sum_{l=\operatorname{Max}(n+1,j+1-n)}^{j}
\sum_{\substack{\Omega_j \subset J_N\\(s_1, s_2, \dots, s_n)\\=(1, 2, \dots, n)
\\s_l=\overline{n}}}
d_{\Omega_j}(x,r) {:}\overrightarrow{\Lambda}_{\Omega_j}(z)
\Lambda_{\overline{n}}\big(x^{2N-j+2}z\big){:}
\\
&+\sum_{n=1}^{j}\sum_{l=\operatorname{Max}(n,j+1-n)}^{j}
\sum_{\substack{\Omega_j \subset J_N\\(s_1, s_2, \dots, s_n)\\=
(1,2, \dots, n)\\s_l \prec \overline{n} \prec s_{l+1}}}\!\!\!\!
\Delta\big(x^{2(N+1-l}\big) d_{\Omega_j}(x,r){:}
\overrightarrow{\Lambda}_{\Omega_j}(z)
\Lambda_{\overline{n}}\big(x^{2N-j+2}z\big){:}.
\end{align*}
The relation $d_{\Omega_j}(x,r)=\Delta\big(x^{2(N+1-j)}\big) d_{\Omega_j \setminus \{n\}}(x,r)$ holds,
if $s_k=n$, $s_l=\overline{n}$, $1\leq n \leq j-1$, and~$k=n$, and $n+1\leq l \leq j$ are satisfied.
The relation $d_{\Omega_j}(x,r)\Delta\big(x^{2(N+1-l)}\big)
=\Delta\big(x^{2(N+1-j)}\big)\times d_{\Omega_j \setminus \{n\}}(x,r)$ holds, if
$s_k=n$, $s_{l} \prec \overline{n} \prec s_{l+1}$, $1\leq n \leq j$, $k=n$, and
$n+1\leq l\leq j$ are satisfied. Using the above two relations and
\begin{align*}
{:}\overrightarrow{\Lambda}_{\Omega_j}(z)
\Lambda_{\overline{n}}\big(x^{2N-j+2}z\big){:}=
\overrightarrow{\Lambda}_{\Omega_j \setminus \{n\}}(xz),\qquad
(s_1, s_2, \dots, s_n)=(1, 2,\dots, n),\qquad
1\leq n \leq j,
\end{align*}
obtained from (\ref{eqn:fusion-Lambda2}) and (\ref{eqn:fusion-Lambda3}), yields
\begin{align*}
{H}_{2N-j+2}(z)={}&\Delta\big(x^{2(N-j+1)}\big)\Bigg(
\sum_{n=1}^{j-1}\sum_{l=\operatorname{Max}(n+1, j+1-n)}^{j}
\sum_{\substack{\Omega_j \subset J_N\\s_k=n, \, k=n\\s_l=\overline{n}}}
d_{\Omega_j \setminus \{s_k\}}(x,r)
\overrightarrow{\Lambda}_{\Omega_j \setminus \{s_k\}}(xz)
\\
&\hphantom{\Delta\big(x^{2(N-j+1)}\big)\Bigg(}+\sum_{n=1}^j\sum_{l=\operatorname{Max}(n, j+1-n)}^{j}
\sum_{\substack{\Omega_j \subset J_N\\s_k=n,\, k=n
\\s_l \prec \overline{n} \prec s_{l+1}}}d_{\Omega_j \setminus \{s_k\}}(x,r)
\overrightarrow{\Lambda}_{\Omega_j \setminus \{s_k\}}(xz)\Bigg).
\end{align*}
Hence, we obtain
${H_{2N-j+2}(z)}=\Delta\big(x^{2(N-j+1)}\big)T_{j-1}(xz)$, $1\leq j \leq N$.

We define $\overline{G}_{j+1-2m}(z)$, $1\leq m \leq [\frac{j}{2}]$, $1\leq j \leq N$,
as the coefficient of $\delta(x^{-j-1+2m}z_2/z_1)$ in~(\ref{eqn:L1j-2}).
Adding the first term in~(\ref{exchange:L1}), the second term in~(\ref{exchange:L1}),
the first term in~(\ref{exchange:L6}), the third term in~(\ref{exchange:L6}),
the first term in~(\ref{exchange:L7}),
and the third term in~(\ref{exchange:L7}) yields
\begin{align*}
\overline{G}_{j+1-2m}(z)={}&\sum_{\Omega_j \subset J_N}
\sum_{\substack{s\in J_N\\s_m \prec s \prec s_{m+1}
\\ \overline{s} \notin \Omega_j}}d_{\Omega_j}(x,r)
{:}\Lambda_s\big(x^{-j-1+2m}z\big)\overrightarrow{\Lambda}_{\Omega_j}(z){:}
\\
&-\sum_{\Omega_j \subset J_N}\sum_{\substack{s\in J_N
\\s_{m-1} \prec s \prec s_{m}\\ \overline{s} \notin \Omega_j}}d_{\Omega_j}(x,r)
{:}\Lambda_s\big(x^{-j-1+2m}z\big)\overrightarrow{\Lambda}_{\Omega_j}(z){:}
\\
&+\sum_{\Omega_j \subset J_N}\sum_{n=1}^N\sum_{\substack{l=m+1
\\s_m \prec s \prec s_{m+1}\\s_l=\overline{s},\, s=n}}^j\!\!\!
\Delta\big(x^{2(l-m+n-N-1)}\big) d_{\Omega_j}(x,r)
{:}\Lambda_s\big(x^{-j-1+2m}z\big)\overrightarrow{\Lambda}_{\Omega_j}(z){:}
\\
&-\sum_{\Omega_j \subset J_N}\sum_{n=1}^N\sum_{\substack{l=m
\\s_{m-1} \prec s \prec s_{m}\\s_l=\overline{s},\, s=n}}^j\!\!\!
\Delta\big(x^{2(l-m+n-N-1)}\big) d_{\Omega_j}(x,r)
{:}\Lambda_s\big(x^{-j-1+2m}z\big)\overrightarrow{\Lambda}_{\Omega_j}(z){:}
\\
&+\sum_{\Omega_j \subset J_N}\sum_{n=1}^N\sum_{\substack{k=1
\\s_k=\overline{s},\, s=\overline{n}\\s_m \prec s \prec s_{m+1}}}^m\!\!\!
\Delta\big(x^{2(m-k+n-N-1)}\big) d_{\Omega_j}(x,r)
{:}\Lambda_s\big(x^{-j-1+2m}z\big)\overrightarrow{\Lambda}_{\Omega_j}(z){:}
\\
&-\sum_{\Omega_j \subset J_N}\sum_{n=1}^N\!\sum_{\substack{k=1
\\s_k=\overline{s},\, s=\overline{n}
\\s_{m-1} \prec s \prec s_{m}}}^{m-1}\!\!\!
\Delta\big(x^{2(m-k+n-N-1)}\big) d_{\Omega_j}(x,r)
{:}\Lambda_s\big(x^{-j-1+2m}z\big)\overrightarrow{\Lambda}_{\Omega_j}(z){:}.
\end{align*}
For a subset $\Omega_j=\{s_1, s_2, \dots, s_j\}\subset J_N$ and an element $s \notin \Omega_j$,
$s \in J_N$, we write elements of $\Omega_j \cup \{s\}$ as
$t_1, t_2, \dots, t_{j+1}, t_1 \prec t_2 \prec \cdots \prec t_{j+1}$.
In what follows, we use the abbreviation
$\Omega_{j+1}=\{t_1, t_2, \dots, t_{j+1}\}$.
Rewriting the sum yields
\begin{align*}
\overline{G}_{j+1-2m}(z)={}&
\sum\limits_{\substack{{\Omega}_{j+1} \subset J_N
\\ \overline{t}_{m+1} \notin {\Omega}_{j+1}
\setminus \{t_{m+1}\}}}
d_{\{t_1, \dots, t_m\}}(x,r)\,
d_{\{t_{m+2}, \dots, t_{j+1}\}}(x,r)
\\ \qquad
&\phantom{-}\times \prod_{p=1}^m\prod\limits_{\substack{q=m+2 \\t_q=\overline{t}_p}}^{j+1}
\Delta\big(x^{2(q-p+t_p-N-2)}\big){:}\Lambda_{t_{m+1}}\big(x^{-j-1+2m}z\big)
\overrightarrow{\Lambda}_{{\Omega}_{j+1}
\setminus \{t_{m+1}\}}(z){:}
\\
&-\sum\limits_{\substack{\Omega_{j+1}\subset J_N
\\ \overline{t}_m \notin \Omega_{j+1} \setminus \{t_m\}}}
d_{\{t_1,\dots, t_{m-1}\}}(x,r) d_{\{t_{m+1}, \dots, t_{j+1}\}}(x,r)
\\
&\phantom{-}\times \prod_{p=1}^{m-1}\prod\limits_{\substack{q=m+1\\t_q=\overline{t}_p}}^{j+1}
\Delta\big(x^{2(q-p+t_p-N-2)}\big){:}\Lambda_{t_{m}}\big(x^{-j-1+2m}z\big)
\overrightarrow{\Lambda}_{\Omega_{j+1}\setminus \{t_{m}\}}(z){:}
\\
&+\sum_{\Omega_{j+1} \subset J_N}
\sum\limits_{\substack{n=1\\ t_{m+1}=n}}^N
\sum\limits_{\substack{l=m+2\\ t_l=\overline{t}_{m+1}}}^{j+1}
d_{\{t_1,\dots, t_m\}}(x,r) d_{\{t_{m+2}, \dots, t_{j+1}\}}(x,r)
\\
&\phantom{-}\times \prod_{p=1}^m\prod\limits_{\substack{q=m+2\\t_q=\overline{t}_p}}^{j+1}
\Delta\big(x^{2(q-p+t_p-N-2)}\big)\Delta\big(x^{2(l-m+t_{m+1}-N-2)}\big)
\\
&\phantom{-}\times {:}\Lambda_{t_{m+1}}\big(x^{-j-1+2m}z\big)
\overrightarrow{\Lambda}_{\Omega_{j+1}\setminus \{t_{m+1}\}}(z){:}
\\
&-\sum_{\Omega_{j+1}\subset J_N}
\sum\limits_{\substack{n=1\\t_m=n}}^N
\sum\limits_{\substack{l=m+1\\t_l=\overline{t}_m}}^{j+1}
d_{\{t_1,\dots, t_{m-1}\}}(x,r) d_{\{t_{m+1},\dots, t_{j+1}\}}(x,r)
\\
&\phantom{-}\times \prod_{p=1}^{m-1}\prod\limits_{\substack{q=m+1\\t_q=\overline{t}_p}}^{j+1}
\Delta\big(x^{2(q-p+t_p-N-2)}\big)\Delta\big(x^{2(l-m+t_m-N-2)}\big)
\\
&\phantom{-}\times {:}\Lambda_{t_{m}}\big(x^{-j-1+2m}z\big)
\overrightarrow{\Lambda}_{\Omega_{j+1}\setminus \{t_m\}}(z){:}
\\
&+\sum_{\Omega_{j+1} \subset J_N}
\sum\limits_{\substack{n=1\\t_{m+1}=\overline{n}}}^N
\sum\limits_{\substack{k=1\\t_{m+1}=\overline{t}_k}}^m
d_{\{t_1,\dots, t_m\}}(x,r) d_{\{t_{m+2}, \dots, t_{j+1}\}}(x,r)
\\
&\phantom{-}\times
\prod_{p=1}^m\prod\limits_{\substack{q=m+2\\t_q=\overline{t}_p}}^{j+1}
\Delta\big(x^{2(q-p+t_p-N-2)}\big)\Delta\big(x^{2(m-k+t_k-N-1)}\big)
\\
&\phantom{-}\times{:}\Lambda_{t_{m+1}}\big(x^{-j-1+2m}z\big)
\overrightarrow{\Lambda}_{\Omega_{j+1}\setminus \{t_{m+1}\}}(z){:}
\\
&-\sum_{\Omega_{j+1}\subset J_N}
\sum\limits_{\substack{n=1\\t_m=\overline{n}}}^N
\sum\limits_{\substack{k=1\\t_m=\overline{t}_k}}^{m-1}
d_{\{t_1, \dots, t_{m-1}\}}(x,r) d_{\{t_{m+1}, \dots, t_{j+1}\}}(x,r)
\\
&\phantom{-}\times\prod_{p=1}^{m-1}\prod\limits_{\substack{q=m+1\\t_q=\overline{t}_p}}^{j+1}
\Delta\big(x^{2(q-p+t_p-N-2)}\big)
\Delta\big(x^{2(m-k+t_k-N-1)}\big)
\\
&\phantom{-}\times{:}\Lambda_{t_{m}}\big(x^{-j-1+2m}z\big)
\overrightarrow{\Lambda}_{\Omega_{j+1}\setminus \{t_m\}}(z){:}.
\end{align*}
Rewriting the sum yields
\begin{align*}
\overline{G}_{j+1-2m}(z)={}&
\sum_{\Omega_{j+1}\subset J_N}
d_{\{t_1,\dots,t_m\}}(x,r) d_{\{t_{m+2},\dots,t_{j+1}\}}(x,r)
\prod_{p=1}^m \prod\limits_{\substack{q=m+2
\\t_q=\overline{t}_p}}^{j+1}\Delta\big(x^{2(q-p+t_p-N-2)}\big)
\\
&\phantom{-}\times\prod\limits_{\substack{p=1\\t_{m+1}=\overline{t}_p}}^m
\Delta\big(x^{2(m-p+t_p-N-1)}\big)
\prod\limits_{\substack{q=m+2\\t_q=\overline{t}_{m+1}}}^{j+1}
\Delta\big(x^{2(q-m+t_{m+1}-N-2)}\big)
\\
&\phantom{-}\times{:}\Lambda_{t_{m+1}}\big(x^{-j-1+2m}z\big)
\overrightarrow{\Lambda}_{\Omega_{j+1}\setminus \{t_{m+1}\}}(z){:}
\\
&-\sum_{\Omega_{j+1}\subset J_N}\!\!\!
d_{\{t_1,\dots,t_{m-1}\}}(x,r) d_{\{t_{m+1},\dots,t_{j+1}\}}(x,r)
\prod_{p=1}^{m-1}\! \prod\limits_{\substack{q=m+1
\\t_q=\overline{t}_p}}^{j+1}\!\!\!\Delta\big(x^{2(q-p+t_p-N-2)}\big)
\\
&\phantom{-}\times\prod\limits_{\substack{p=1\\t_{m}=\overline{t}_p}}^m
\Delta\big(x^{2(m-p+t_p-N-1)}\big)\prod\limits_{\substack{q=m+1\\t_q=\overline{t}_{m}}}^{j+1}
\Delta\big(x^{2(q-m+t_{m}-N-2)}\big)
\\
&\phantom{-}\times{:}\Lambda_{t_{m}}\big(x^{-j-1+2m}z\big)
\overrightarrow{\Lambda}_{\Omega_{j+1}\setminus \{t_{m}\}}(z){:}.
\end{align*}
Using
\begin{gather*}
d_{\{t_1,\dots,t_m\}}(x,r) d_{\{t_{m+2},\dots,t_{j+1}\}}(x,r)
\prod_{p=1}^m \prod\limits_{\substack{q=m+2\\t_q=\overline{t}_p}}^{j+1}
\Delta\big(x^{2(q-p+t_p-N-2)}\big)
\prod\limits_{\substack{p=1\\t_{m+1}=\overline{t}_p}}^m
\Delta\big(x^{2(m-p+t_p-N-1)}\big)
\\ \qquad
{}\times\prod\limits_{\substack{q=m+2\\t_q=\overline{t}_{m+1}}}^{j+1}
\Delta\big(x^{2(q-m+t_{m+1}-N-2)}\big)
=d_{\{t_1,\dots,t_{m-1}\}}(x,r) d_{\{t_{m+1},\dots,t_{j+1}\}}(x,r)
\\ \qquad
{}\times
\prod_{p=1}^{m-1} \prod\limits_{\substack{q=m+1
\\t_q=\overline{t}_p}}^{j+1}\!
\Delta\big(x^{2(q-p+t_p-N-2)}\big)\!
\prod\limits_{\substack{p=1
\\t_{m}=\overline{t}_p}}^m\Delta\big(x^{2(m-p+t_p-N-1)}\big)
\prod\limits_{\substack{q=m+1\\t_q=\overline{t}_{m}}}^{j+1}
\Delta\big(x^{2(q-m+t_{m}-N-2)}\big)
\end{gather*}
and
\begin{gather*}
{:}\Lambda_{t_{m+1}}\big(x^{-j-1+2m}z\big)
\overrightarrow{\Lambda}_{\Omega_{j+1}\setminus \{t_{m+1}\}}(z){:}
={:}\Lambda_{t_{m}}\big(x^{-j-1+2m}z\big)
\overrightarrow{\Lambda}_{\Omega_{j+1}\setminus \{t_{m}\}}(z){:}
\end{gather*}
yields
$\overline{G}_{j+1-2m}(z)=0$, $1\leq m \leq [\frac{j}{2}]$, $1\leq j \leq N$.

We define ${G}_{j+1-2m}(z)$, $1\leq m \leq \big[\frac{j}{2}\big]$, $1\leq j \leq N$,
as the coefficient of $\delta\big(x^{j+1-2m}z_2/z_1\big)$ in~(\ref{eqn:L1j-2}).
Adding the first term in~(\ref{exchange:L1}),
the second term in~(\ref{exchange:L1}), the first term in~(\ref{exchange:L6}),
the third term in~(\ref{exchange:L6}), the first term in~(\ref{exchange:L7}),
and the third term in~(\ref{exchange:L7}) yields
\begin{align*}
{G}_{j+1-2m}(z)={}&\sum_{\Omega_j \subset J_N}
\sum\limits_{\substack{s \in J_N\\s_{j+1-m} \prec s \prec s_{j+2-m}
\\ \overline{s} \notin \Omega_j}}
d_{\Omega_j}(x,r){:}\overrightarrow{\Lambda}_{\Omega_j}(z)
\Lambda_s\big(x^{j+1-2m}z\big) {:}
\\
&-\sum_{\Omega_j \subset J_N}\sum\limits_{\substack{s\in J_N\\s_{j-m} \prec s \prec s_{j-m+1}
\\ \overline{s} \notin \Omega_j}}d_{\Omega_j}(x,r)
{:}\overrightarrow{\Lambda}_{\Omega_j}(z)
\Lambda_s\big(x^{j+1-2m}z\big){:}
\\
&+\sum_{\Omega_j \subset J_N}\sum_{n=1}^N\sum\limits_{\substack{l=j+2-m
\\s_{j-m+1} \prec s\\ \prec s_{j-m+2}\\ s_l=\overline{s},\, s=n}}^j
\Delta\big(x^{2(l+m+n-j-N-2)}\big) d_{\Omega_j}(x,r)
{:}\overrightarrow{\Lambda}_{\Omega_j}(z)\Lambda_s\big(x^{j+1-2m}z\big){:}
\\
&-\sum_{\Omega_j \subset J_N}\sum_{n=1}^N\sum\limits_{\substack{l=j+1-m
\\s_{j-m} \prec s\\ \prec s_{j-m+1}\\ s_l=\overline{s},\, s=n}}^j
\Delta\big(x^{2(l+m+n-j-N-2)}\big) d_{\Omega_j}(x,r)
{:}\overrightarrow{\Lambda}_{\Omega_j}(z)\Lambda_s\big(x^{j+1-2m}z\big){:}
\\
&+\sum_{\Omega_j \subset J_N}\sum_{n=1}^N
\sum\limits_{\substack{k=1
\\ s_k=\overline{s},\, s=\overline{n}\\ s_{j+1-m} \prec s\\ \prec s_{j+2-m}}}^{j+1-m}
\Delta\big(x^{2(j-m-k+n-N)}\big) d_{\Omega_j}(x,r)
{:}\overrightarrow{\Lambda}_{\Omega_j}(z)\Lambda_s\big(x^{j+1-2m}z\big){:}
\\
&-\sum_{\Omega_j \subset J_N}\sum_{n=1}^N
\sum\limits_{\substack{k=1\\ s_k=\overline{s},\, s=\overline{n}
\\ s_{j-m} \prec s\\ \prec s_{j-m+1}}}^{j-m}
\Delta\big(x^{2(j-m-k+n-N)}\big) d_{\Omega_j}(x,r)
{:}\overrightarrow{\Lambda}_{\Omega_j}(z)\Lambda_s\big(x^{j+1-2m}z\big){:}.
\end{align*}
Rewriting the sum, in the same way as the case of $\overline{G}_{j+1-2m}(z)$, yields
\begin{align*}
{G}_{j+1-2m}(z)={}&
\sum\limits_{\substack{{\Omega}_{j+1} \subset J_N
\\ \overline{t}_{j+2-m} \notin {\Omega}_{j+1}\setminus \{t_{j+2-m}\}}}
d_{\{t_1, \dots, t_{j+1-m}\}}(x,r) d_{\{t_{j+3-m}, \dots, t_{j+1}\}}(x,r)
\\
&\times\!\!\prod_{p=1}^{j+1-m}\!\!\!\prod\limits_{\substack{q=m+j+3-m \\ t_q=\overline{t}_p}}^{j+1}
\!\!\!\!\!\!\Delta\big(x^{2(q-p+t_p-N-2)}\big)
{:}\overrightarrow{\Lambda}_{{\Omega}_{j+1}\setminus \{t_{j+2-m}\}}(z)
\Lambda_{t_{j+2-m}}\big(x^{j+1-2m}z\big){:}
\\
&-\sum\limits_{\substack{\Omega_{j+1}\subset J_N
\\ \overline{t}_{j+1-m} \notin \Omega_{j+1} \setminus\{t_{j+1-m}\}}}
d_{\{t_1,\dots, t_{j-m}\}}(x,r) d_{\{t_{j+2-m}, \dots, t_{j+1}\}}(x,r)
\\
&\times\prod_{p=1}^{j-m}\prod\limits_{\substack{q=j+2-m
\\ t_q=\overline{t}_p}}^{j+1}\Delta\big(x^{2(q-p+t_p-N-2)}\big)
{:}\overrightarrow{\Lambda}_{\Omega_{j+1}\setminus \{t_{j+1-m}\}}(z)
\Lambda_{t_{j+1-m}}\big(x^{j+1-2m}z\big){:}
\\
&+\sum_{\Omega_{j+1} \subset J_N}\sum\limits_{\substack{n=1
\\t_{j-m+2}=n}}^N\sum\limits_{\substack{l=j+2-m \\ t_l=\overline{n}}}^{j+1}
d_{\{t_1,\dots, t_{j+1-m}\}}(x,r) d_{\{t_{j+3-m}, \dots, t_{j+1}\}}(x,r)
\\
&\times\prod_{p=1}^{j+1-m}\prod\limits_{\substack{q=j+3-m\\t_q=\overline{t}_p}}^{j+1}
\Delta\big(x^{2(q-p+t_p-N-2)}\big)\Delta\big(x^{2(l+m+n-j-N-2)}\big)
\\
&\times{:}\overrightarrow{\Lambda}_{\Omega_{j+1}\setminus \{t_{j+2-m}\}}(z)
\Lambda_{t_{j+2-m}}\big(x^{j+1-2m}z\big){:}
\\
&-\sum_{\Omega_{j+1}\subset J_N}\sum\limits_{\substack{n=1\\t_{j+1-m}=n}}^N
\sum\limits_{\substack{l=j+1-m \\ t_l=\overline{n}}}^{j}
d_{\{t_1,\dots, t_{j-m}\}}(x,r) d_{\{t_{j+2-m},\dots, t_{j+1}\}}(x,r)
\\
&\times\prod_{p=1}^{j-m}\prod\limits_{\substack{q=j+2-m\\t_q=\overline{t}_p}}^{j+1}
\Delta\big(x^{2(q-p+t_p-N-2)}\big)\Delta\big(x^{2(l+m+n-j-N-2)}\big)
\\
&\times{:}\overrightarrow{\Lambda}_{\Omega_{j+1}\setminus \{t_{j+1-m}\}}(z)
\Lambda_{t_{j+1-m}}\big(x^{j+1-2m}z\big){:}
\\
&+\sum_{\Omega_{j+1} \subset J_N}\sum\limits_{\substack{n=1\\t_{j+2-m}=\overline{n}}}^N
\sum\limits_{\substack{k=1\\t_{k}=n}}^{j+1-m}
d_{\{t_1,\dots, t_{j+1-m}\}}(x,r) d_{\{t_{j+3-m}, \dots, t_{j+1}\}}(x,r)
\\
&\times\prod_{p=1}^{j+1-m}\prod\limits_{\substack{q=j+3-m\\t_q=\overline{t}_p}}^{j+1}
\Delta\big(x^{2(q-p+t_p-N-2)}\big)\Delta\big(x^{2(j+n-m-k-N)}\big)
\\
&\times{:}\overrightarrow{\Lambda}_{\Omega_{j+1}\setminus \{t_{j+2-m}\}}(z)
\Lambda_{t_{j+2-m}}\big(x^{j+1-2m}z\big){:}
\\
&-\sum_{\Omega_{j+1}\subset J_N}\sum\limits_{\substack{n=1\\t_{j+1-m}=\overline{n}}}^N
\sum\limits_{\substack{k=1\\{t_k=n}}}^{j-m}
d_{\{t_1, \dots, t_{j-m}\}}(x,r) d_{\{t_{j-m+1}, \dots, t_{j+1}\}}(x,r)
\\
&\times\prod_{p=1}^{j-m}\prod\limits_{\substack{q=j+2-m\\t_q=\overline{t}_p}}^{j+1}
\Delta\big(x^{2(q-p+t_p-N-2)}\big)\Delta\big(x^{2(j+n-m-k-N)}\big)
\\
&\times{:}\overrightarrow{\Lambda}_{\Omega_{j+1}\setminus \{t_{j+1-m}\}}(z)\Lambda_{t_{j+1-m}}\big(x^{j+1-2m}z\big){:}.
\end{align*}
Rewriting the sum yields
\begin{align*}
{G}_{j+1-2m}(z)={}&\sum_{\Omega_{j+1}\subset J_N}
d_{\{t_1,\dots,t_{j+1-m}\}}(x,r) d_{\{t_{j+3-m},\dots,t_{j+1}\}}(x,r)
\\
&\times\prod_{p=1}^{j+1-m} \prod\limits_{\substack{q=j+3-m\\t_q=\overline{t}_p}}^{j+1}
\Delta\big(x^{2(q-p+t_p-N-2)}\big)
\prod\limits_{\substack{p=1\\t_{j+2-m}=\overline{t}_p}}^{j+1-m}
\Delta\big(x^{2(j+t_p-m-p-N)}\big)
\\
&\times\!\!\!
\prod\limits_{\substack{q=j+3-m\\t_q=\overline{t}_{j+2-m}}}^{j+1}\!\!\!\!
\Delta\big(x^{2(q+m+t_{j+2-m}-j-N-2)}\big)
{:}\overrightarrow{\Lambda}_{\Omega_{j+1}\setminus \{t_{j+2-m}\}}(z)
\Lambda_{t_{j+2-m}}\big(x^{j+1-2m}z\big){:}
\\
&-\sum_{\Omega_{j+1}\subset J_N}
d_{\{t_1,\dots,t_{j-m}\}}(x,r) d_{\{t_{j+2-m},\dots,t_{j+1}\}}(x,r)
\\
&\times\prod_{p=1}^{j-m} \prod\limits_{\substack{q=j+2-m\\t_q=\overline{t}_p}}^{j+1}
\Delta\big(x^{2(q-p+t_p-N-2)}\big)
\prod\limits_{\substack{p=1\\t_{j+1-m}=\overline{t}_p}}^{j-m}
\Delta\big(x^{2(j+t_p-m-N-p)}\big)
\\
&\times\!\!\!\!\prod\limits_{\substack{q=j+2-m\\t_q=\overline{t}_{j+1-m}}}^{j+1}\!\!\!\!
\Delta\big(x^{2(q+m+t_{j+1-m}-j-N-3)}\big)
{:}\overrightarrow{\Lambda}_{\Omega_{j+1}\setminus \{t_{j+1-m}\}}(z)
\Lambda_{t_{j+1-m}}\big(x^{j+1-2m}z\big){:}.
\end{align*}
Using
\begin{gather*}
d_{\{t_1,\dots,t_{j+1-m}\}}(x,r)
d_{\{t_{j+3-m},\dots,t_{j+1}\}}(x,r)
\prod_{p=1}^{j+1-m} \prod\limits_{\substack{q=j+3-m\\t_q=\overline{t}_p}}^{j+1}
\Delta\big(x^{2(q-p+t_p-N-2)}\big)
\\[-.5ex] \qquad\phantom{=}
{}\times
\prod\limits_{\substack{p=1\\t_{j+2-m}=\overline{t}_p}}^{j+1-m}
\Delta\big(x^{2(j+t_p-m-p-N)}\big)
\prod\limits_{\substack{q=j+3-m\\t_q=\overline{t}_{j+2-m}}}^{j+1}
\Delta\big(x^{2(q+m+t_{j+2-m}-j-N-3)}\big)
\\[-.5ex] \qquad
{}=d_{\{t_1,\dots,t_{j-m}\}}(x,r) d_{\{t_{j-m+2},\dots,t_{j+1}\}}(x,r)
\prod_{p=1}^{j-m} \prod\limits_{\substack{q=j+2-m\\t_q=\overline{t}_p}}^{j+1}
\Delta\big(x^{2(q-p+t_p-N-2)}\big)
\\[-.5ex] \qquad\phantom{=}
{}\times\prod\limits_{\substack{p=1\\t_{j+1-m}=\overline{t}_p}}^{j-m}
\Delta\big(x^{2(j+t_p-m-p-N)}\big)
\prod\limits_{\substack{q=j+2-m\\t_q=\overline{t}_{j+1-m}}}^{j+1}
\Delta\big(x^{2(q+m+t_{j+1-m}-j-N-3)}\big)
\end{gather*}
and
\begin{gather*}
{:}\overrightarrow{\Lambda}_{\Omega_{j+1}\setminus \{t_{j+2-m}\}}(z)
\Lambda_{t_{j+2-m}}\big(x^{j+1-2m}z\big){:}
={:}\overrightarrow{\Lambda}_{\Omega_{j+1}\setminus \{t_{j+1-m}\}}(z)
\Lambda_{t_{j+1-m}}\big(x^{j+1-2m}z\big){:}
\end{gather*}
yields ${{G}_{j+1-2m}(z)}=0$, $1\leq m \leq \big[\frac{j}{2}\big]$, $1\leq j \leq N$.

{\sloppy We define $\overline{H}_{2N-j+2-2m}(z)$,
$1\leq j \leq N$, $1\leq m \leq N-\big[\frac{j-1}{2}\big]$,
as the coefficient of $\delta\big(x^{-2N+j-2+2m}z_2/z_1\big)$ in (\ref{eqn:L1j-2}).
We set
\begin{gather*}
\overline{H}_{2N-j+2-2m}(z)=\sum_{\varepsilon=\pm}\varepsilon
\big(\overline{\beta}_\varepsilon(z)+\overline{\gamma}_\varepsilon(z)
+\overline{\delta}_\varepsilon(z)\big),
\end{gather*}}%
where we give
$\overline{\beta}_+(z)$, $\overline{\beta}_-(z)$, $\overline{\gamma}_+(z)$,
$\overline{\gamma}_-(z)$, $\overline{\delta}_+(z)$, and $\overline{\delta}_-(z)$
in~(\ref{def:bbeta+}), (\ref{def:bbeta-}), (\ref{def:bgamma+}), (\ref{def:bgamma-}),
(\ref{def:bdelta+}), and~(\ref{def:bdelta-}), respectively.
Adding the first term in (\ref{exchange:L3}) and the fourth term
in (\ref{exchange:L6}) yields
\begin{align}
\overline{\beta}_+(z)={}&\sum_{\Omega_j \subset J_N}
\sum_{n=m+1}^{\operatorname{Min}(N, j+m-1)}
\sum\limits_{\substack{k=1\\s_k=n
\\ s_l=\overline{n},\, l=m+j+1-n}}^{j+m-n}d_{\Omega_j}(x,r)
{:}\Lambda_{n}\big(x^{-2N+j+2m-2}z\big)\overrightarrow{\Lambda}_{\Omega_j}(z){:}\notag
\\[-.5ex]
&+\sum_{\Omega_j \subset J_N}\sum_{n=m+1}^{\operatorname{Min}(N, j+m)}
\sum\limits_{\substack{k=1\\s_{k-1}\prec n \prec s_k
\\ s_l=\overline{n},\, l=j+m+1-n}}^{j+m+1-n}
d_{\Omega_j}(x,r)\Delta\big(x^{2(m+j-N-k)}\big)\notag
\\[-.5ex]
&\phantom{+}\times{:}\Lambda_{n}\big(x^{-2N+j+2m-2}z\big)\overrightarrow{\Lambda}_{\Omega_j}(z){:}.
\label{def:bbeta+}
\end{align}
Adding the second term in (\ref{exchange:L3}) and the second term in~(\ref{exchange:L6}) yields
\begin{align}
\overline{\beta}_-(z)={}&
\sum_{\Omega_j \subset J_N}\sum_{n=m}^{\operatorname{Min}(N, j+m-2)}
\sum\limits_{\substack{k=1\\s_k=n
\\{_l=\overline{n},\, l=j+m-n}}}^{j+m-n-1}
d_{\Omega_j}(x,r)
{:}\Lambda_{n}\big(x^{-2N+j+2m-2}z\big)\overrightarrow{\Lambda}_{\Omega_j}(z){:}\notag
\\[-.5ex]
&+\sum_{\Omega_j \subset J_N}\sum_{n=m}^{\operatorname{Min}(N, j+m-1)}
\sum\limits_{\substack{k=1\\s_{k-1}\prec n \prec s_k
\\s_l=\overline{n},\, l=j+m-n}}^{j+m-n}d_{\Omega_j}(x,r)\Delta\big(x^{2(m+j-N-k)}\big)\notag
\\[-.5ex]
&\phantom{+}\times{:}\Lambda_{n}\big(x^{-2N+j+2m-2}z\big)\overrightarrow{\Lambda}_{\Omega_j}(z){:}.
\label{def:bbeta-}
\end{align}
Adding the first term in (\ref{exchange:L5}) yields
\begin{gather}
\overline{\gamma}_+(z)=
\sum\limits_{\substack{\Omega_j \subset J_N\\s_k=0,\, k=j+m-N}}
d_{\Omega_j}(x,r){:}\Lambda_0\big(x^{-2N+j-2+2m}z\big)
\overrightarrow{\Lambda}_{\Omega_j}(z){:}.
\label{def:bgamma+}
\end{gather}
Adding the second term in (\ref{exchange:L5}) yields
\begin{gather}
\overline{\gamma}_-(z)=\sum\limits_{\substack{\Omega_j \subset J_N
\\s_k=0,\, k=j+m-N-1}}d_{\Omega_j}(x,r)
{:}\Lambda_0\big(x^{-2N+j-2+2m}z\big)\overrightarrow{\Lambda}_{\Omega_j}(z){:}.
\label{def:bgamma-}
\end{gather}
Adding the first term in (\ref{exchange:L4}) and
the fourth term in (\ref{exchange:L7}) yields
\begin{align}
\overline{{\delta}}_+(z)={}&\sum_{\Omega_j \subset J_N}\sum_{n=2N+2-j-m}^{N}
\sum\limits_{\substack{l=m+j+n-2N\\
s_k=n, \, k=j+m+n-2N-1\\ s_l=\overline{n}}}^{j}d_{\Omega_j}(x,r)
{:}\overrightarrow{\Lambda}_{\Omega_j}(z)\Lambda_{\overline{n}}\big(x^{-2N+j+2m-2}z\big){:}\notag
\\
&+\sum_{\Omega_j \subset J_N}\sum_{n=2N+2-j-m}^{N}
\sum\limits_{\substack{l=m+j+n-2N-1\\s_k=n,\, k=j+m+n-2N-1
\\s_l\prec \overline{n} \prec s_{l+1}}}^{j}
d_{\Omega_j}(x,r)\Delta\big(x^{2(N+l+1-m-j)}\big)\notag
\\
&\phantom{+}\times{:}\overrightarrow{\Lambda}_{\Omega_j}(z)
\Lambda_{\overline{n}}\big(x^{-2N+j+2m-2}z\big){:}.
\label{def:bdelta+}
\end{align}
Adding the second term in (\ref{exchange:L4}) and
the second term in (\ref{exchange:L7}) yields
\begin{align}
\overline{\delta}_-(z)={}&
\sum_{\Omega_j \subset J_N}\sum_{n=2N+3-j-m}^{N}
\sum\limits_{\substack{l=j+m+n-2N-1\\s_k=n,\, k=j+m+n-2N-2
\\s_l=\overline{n}}}^{j}d_{\Omega_j}(x,r)
{:}\overrightarrow{\Lambda}_{\Omega_j}(z)\Lambda_{\overline{n}}\big(x^{-2N+j+2m-2}z\big){:}\notag
\\
&+\sum_{\Omega_j \subset J_N}\sum_{n=2N+3-j-m}^{N}
\sum\limits_{\substack{l=j+m+n-2N-2\\s_k=n,\, k=j+m+n-2N-2
\\s_l\prec \overline{n} \prec s_{l+1}}}^{j}
d_{\Omega_j}(x,r)\Delta\big(x^{2(N+l+1-m-j)}\big)\notag
\\
&\phantom{+}\times{:}\overrightarrow{\Lambda}_{\Omega_j}(z)
\Lambda_{\overline{n}}\big(x^{-2N+j+2m-2}z\big){:}.
\label{def:bdelta-}
\end{align}

We show $\overline{H}_{2N-j+2-2m}(z)=0$, $1\leq j \leq N$, $1\leq m \leq N-\big[\frac{j-1}{2}\big]$.
In this proof we frequently use relation (\ref{eqn:fusion-Lambda3}).
The proof is divided into three cases: $(i)$, $1\leq m \leq N-j$,
$(ii)$, $m=N+1-j$, and $(iii)$, $N+2-j \leq m \leq N-\big[\frac{j-1}{2}\big]$.

First, we study the case $(i)$, $j+m \leq N$.
In the case $(i)$, $\overline{\gamma}_\pm(z)$ and
$\overline{\delta}_\pm(z)$ vanish. Hence, we have
$\overline{H}_{2N-j+2-2m}(z)=\overline{\beta}_+(z)-\overline{\beta}_-(z)$.
We start from $\overline{\beta}_+(z)$. Rewriting the sum yields
\begin{align}
\overline{\beta}_+(z)={}&
\Bigg(\sum_{n=m+1}^{\operatorname{Min}(N, j+m-1)}
\sum_{k=1}^{j+m-n}\sum\limits_{\substack{\Omega_j \subset J_N\\s_k=n}}
\sum\limits_{\substack{r=0\\(s_l,s_{l+1},\dots,s_{l+r})=
(\overline{n}, \overline{n-1},\dots,\overline{n-r})
\\ \overline{n-r-1}\prec s_{l+r+1},\, l=m+j+1-n}}^{n-m-1}\Delta\big(x^{2(m+j-N-k)}\big)\notag
\\
&+\sum_{n=m+1}^{\operatorname{Min}(N, j+m)}\sum_{k=1}^{j+m-n-1}
\sum\limits_{\substack{\Omega_j \subset J_N\\s_{k-1}\prec n \prec s_k}}
\sum\limits_{\substack{r=0\\(s_l,s_{l+1},\dots,s_{l+r})=
(\overline{n}, \overline{n-1},\dots,\overline{n-r})
\\ \overline{n-r-1}\prec s_{l+r+1},\, l=m+j+1-n}}^{n-m-1}
\Delta\big(x^{2(m+j-N-k)}\big)\Bigg)\notag
\\
&\phantom{+}\times\prod_{a=1}^{k-1}\prod\limits_{\substack{b=1
\\s_a=n-b}}^r \Delta\big(x^{2(m+j-N-a)}\big)
\prod_{p=1}^{k-1}\prod\limits_{\substack{q=l+r+1
\\s_q=\overline{s}_p}}^j\Delta\big(x^{2(q-p+s_p-N-1)}\big)
\notag
\\
&\phantom{+}\times\prod\limits_{\substack{k+1 \leq p<q \leq l-1\\s_q=\overline{s}_p}}
\Delta\big(x^{2(q-p+s_p-N-1)}\big)
{:}\Lambda_n\big(x^{-2N+j+2m-2}z\big)\overrightarrow{\Lambda}_{\Omega_j}(z){:}.
\label{eqn:bbeta+}
\end{align}
Using
\begin{gather}
{:}\Lambda_n\big(x^{-2N+j+2m-2}z\big)\overrightarrow{\Lambda}_{\{s_1,s_2, \dots, s_{l-1}, \overline{n}, \overline{n-1}, \dots, \overline{n-r}, s_{l+r+1}, s_{l+r+2},\dots, s_j\}}(z){:}\notag
\\ \qquad
{}={:}\Lambda_{n-r-1}\big(x^{-2N+j+2m-2}z\big)\overrightarrow{\Lambda}_{\{s_1,s_2, \dots, s_{l-1}, \overline{n-1}, \overline{n-2}, \dots, \overline{n-r-1}, s_{l+r+1}, s_{l+r+2},
\dots, s_j\}}(z){:},\notag
\\
0\leq r \leq n-m-1,\qquad l=m+j+1-n,
\label{eqn:subscript1}
\end{gather}
obtained from (\ref{eqn:fusion-Lambda3}), and replacing $n$ by $n+1$
yields $\overline{\beta}_+(z)=\overline{\beta}_-(z)$. Hence, we obtain
$\overline{H}_{2N-j+2-2m}(z)=\overline{\beta}_+(z)-\overline{\beta}_-(z)=0$
for $1\leq m \leq N-j$.

Next, we examine the case $(ii)$, $m=N+1-j$.
In the case $(ii)$, $\overline{\delta}_\pm(z)$ and
$\overline{\gamma}_-(z)$ vanish. Hence, we have
$\overline{H}_{2N-j+2-2m}(z)=\overline{\beta}_+(z)+\overline{\gamma}_+(z)
-\overline{\beta}_-(z)$.
We start from $\overline{\beta}_+(z)+\overline{\gamma}_+(z)$.
Using~(\ref{eqn:fusion-Lambda3}) yields
\begin{gather}
\overline{\gamma}_+(z)=\Delta(1)\sum_{k=1}^{j}
\sum\limits_{\substack{\Omega_j \subset J_N
\\ \overline{N+1-k}\prec s_{k+1}\prec \cdots \prec s_j}}\!\!\!\!\!\!\!\!\!\!
{:}\Lambda_{N+1-k}\big(x^{-j}z\big)\overrightarrow{\Lambda}_{\{\overline{N},\overline{N-1},
\dots,\overline{N+1-k}, s_{k+1}, s_{k+2}, \dots, s_j\}}(z){:}.
\label{eqn:bgamma+0}
\end{gather}
Using (\ref{eqn:bbeta+}), (\ref{eqn:subscript1}), and (\ref{eqn:bgamma+0})
yields $\overline{\beta}_+(z)+\overline{\gamma}_+(z)=\overline{\beta}_-(z)$.
Hence, we obtain
$\overline{H}_{2N-j+2-2m}(z)\allowbreak=\overline{\beta}_+(z)
+\overline{\gamma}_+(z)-\overline{\beta}_-(z)=0$
for $m=N-j+1$.

Next, we examine the case $(iii)$, $N+2-j \leq m \leq N-[\frac{j-1}{2}]$.
Rewriting the sum yields
\begin{align}
\overline{\delta}_+(z)={}&
\Bigg(\sum_{n=2N+2-j-m}^{N}\sum_{l=j+m+n-2N}^{j}
\sum\limits_{\substack{\Omega_j \subset J_N\\s_l=\overline{n}}}
\sum\limits_{\substack{r=0\\(s_{k-r}, s_{k-r+1},\dots,s_{k})\\=
(n-r, n-r+1,\dots, n)\\s_{k-r-1} \prec n-r-1\\ k=m+j+n-2N-1}}^{j+m+n-2N-2}
\Delta\big(x^{2(l-j-m+N)}\big)\notag
\\
&+\sum_{n=2N+2-j-m}^{N}\sum_{l=m+j+n-2N-1}^j
\sum\limits_{\substack{\Omega_j \subset J_N \\s_l=\overline{n}}}
\sum\limits_{\substack{r=0\\(s_{k-r}, s_{k-r+1},\dots,s_{k})\\=(n-r, n-r+1,\dots, n)
\\s_{k-r-1} \prec n-r-1\\ k=m+j+n-2N-1}}^{j+m+n-2N-2}
\Delta\big(x^{2(l-j-m+N+1)}\big)\Bigg)\notag
\\
&\phantom{+}\times\prod_{a=1}^r \prod\limits_{\substack{b=l+1\\\overline{s}_b=n-a}}^j
\Delta\big(x^{2(N+b-j-m)}\big)\prod_{p=1}^{n-r-1}
\prod\limits_{\substack{q=l+1\\ s_q=\overline{s}_p}}^j
\Delta\big(x^{2(q-p+s_p-N-1)}\big)\notag
\\
&\phantom{+}\times\prod\limits_{\substack{k+1 \leq p<q \leq l-1\\s_q=\overline{s}_p}}
\Delta\big(x^{2(q-p+s_p-N-1)}\big)
{:}\overrightarrow{\Lambda}_{\Omega_j}(z)\Lambda_{\overline{n}}\big(x^{-2N+j+2m-2}z\big){:}
\label{eqn:bdelta+}
\end{align}
and
\begin{align}
\overline{\gamma}_+(z)={}&\sum\limits_{\substack{\Omega_j \subset J_N
\\s_k=0,\, k=j+m-N\\ \overline{N}\prec s_{k+1}}}
\prod_{p=1}^{k-1}\prod\limits_{\substack{q=k+1\\ s_q=\overline{s}_p}}^j
\Delta\big(x^{2(q-p+s_p-N-1)}\big)
{:}\overrightarrow{\Omega}_{\Lambda_j}(z)\Lambda_0\big(x^{-2N+j-2+2m}\big){:}\notag
\\
&+\sum\limits_{\substack{\Omega_j \subset J_N\\s_k=0,\, k=j+m-N
\\ \overline{N}\prec s_{k+1}}}
\sum\limits_{\substack{r=0\\(s_{k+1},s_{k+2},\dots, s_{k+r+1})\\
=(\overline{N}, \overline{N-1},\dots, \overline{N-r})
\\ \overline{N-r-1}\prec s_{k+r+2}}}^{j-1}
\prod_{a=1}^{k-1}\prod_{b=0}^r \Delta\big(x^{2(j+m-N-a)}\big)\notag
\\
&\phantom{+}\times\prod_{p=1}^{k-1}\prod\limits_{\substack{q=k+r+2\\s_q=\overline{s}_p}}^j
\Delta\big(x^{2(q-p+s_p-N-1)}\big){:}
\overrightarrow{\Omega}_{\Lambda_j}(z)\Lambda_0\big(x^{-2N+j-2+2m}\big){:}.
\label{eqn:bgamma+}
\end{align}
Using (\ref{eqn:fusion-Lambda1}) and (\ref{eqn:fusion-Lambda3}) yields
\begin{gather}
{:}\overrightarrow{\Lambda}_{\{
s_1,s_2, \dots, s_{k-r-1}, n-r, n-r+1,\dots, n, s_{k+1}, s_{k+2}, \dots, s_{j}
\}}(z) \Lambda_{\overline{n}}\big(x^{-2N+j+2m-2}z\big){:}\notag
\\ \qquad
{}={:}\overrightarrow{\Lambda}_{\{
s_1,s_2, \dots, s_{k-r-1}, n-r-1, n-r,\dots, n-1, s_{k+1}, s_{k+2}, \dots, s_{j}
\}}(z) \Lambda_{\overline{n-r-1}}\big(x^{-2N+j+2m-2}z\big){:},\notag
\\
0\leq r \leq k-1,\quad k=j+m+n-2N-1,
\label{eqn:subscript2}
\end{gather}
and
\begin{gather}
{:}\overrightarrow{\Lambda}_{\{
s_1,s_2, \dots, s_{k-1}, 0, \overline{N}, \overline{N-1}, \dots,
\overline{N-r}, s_{k+r+2}, s_{k+r+3}, \dots, s_{j}
\}}(z) \Lambda_{0}\big(x^{-2N+j+2m-2}z\big){:}\notag
\\ \qquad
{}=\Delta(1)
{:}\overrightarrow{\Lambda}_{\{
s_1,s_2, \dots, s_{k-1}, \overline{N}, \overline{N-1}, \dots,
\overline{N-r-1}, s_{k+r+2}, s_{k+r+3}, \dots, s_{j}\}}(z)
\Lambda_{N-r-1}\big(x^{-2N+j+2m-2}z\big){:},\notag
\\
0\leq r \leq 2N+l+1-j-m,\qquad k=j+m+n-2N-1.
\label{eqn:subscript3}
\end{gather}
Using (\ref{eqn:bbeta+}), (\ref{eqn:subscript1}), (\ref{eqn:bdelta+}),
(\ref{eqn:bgamma+}), (\ref{eqn:subscript2}), and (\ref{eqn:subscript3})
and replacing $n$ by $n+1$
yield $\overline{\beta}_+(z)+\overline{\gamma}_+(z)+\overline{\delta}_+(z)
=\overline{\beta}_-(z)+\overline{\gamma}_-(z)+\overline{\delta}_-(z)$.
Hence, we obtain
$\overline{H}_{2N-j+2-2m}(z)=\sum_{\varepsilon=\pm}\varepsilon
\big(\overline{\beta}_\varepsilon(z)+\overline{\gamma}_\varepsilon(z)
+\overline{\delta}_\varepsilon(z)\big)=0$
for $N+2-j \leq m \leq N-[\frac{j-1}{2}]$.
Finally, we obtain
$\overline{H}_{2N-j+2-2m}(z)=0$ for $1 \leq m \leq N-\big[\frac{j-1}{2}\big]$.

{\sloppy We define ${H}_{2N-j+2-2m}(z)$, $1\leq j \leq N$, $1\leq m \leq N-[\frac{j-1}{2}]$,
as the coefficient of $\delta(x^{-2N+j-2+2m}z_2/z_1)$ in (\ref{eqn:L1j-2}). We set
\begin{gather}
{H}_{2N-j+2-2m}(z)=\begin{cases}
\displaystyle \sum_{\varepsilon=\pm}
\varepsilon ({\beta}_\varepsilon(z)+{\gamma}_\varepsilon(z)+{\delta}_\varepsilon(z))
& {\rm otherwise},
\\
0&{\rm {if}}\quad j~{\rm {is}~even},\quad m=N-\big[\frac{j-1}{2}\big],
\end{cases}\!\!\!\label{def:H2}
\end{gather}}%
where we give
${\beta}_+(z)$, ${\beta}_-(z)$, ${\gamma}_+(z)$,
${\gamma}_-(z)$, ${\delta}_+(z)$, and ${\delta}_-(z)$
in (\ref{def:beta+}), (\ref{def:beta-}), (\ref{def:gamma+}), (\ref{def:gamma-}),
(\ref{def:delta+}), and (\ref{def:delta-}), respectively.
In (\ref{def:H2}) we define $H_0(z)=0$ to avoid ambiguity of
$\overline{H}_0(z)$ and $H_0(z)$. In the case when $j$ is even, we have
${\rm LHS}_{1,j}=c(x,r)\big(\overline{H}_0(z_2)-H_0(z_2))\delta(z_2/z_1)
+\overline{H}_2(z_2)\delta(x^{-2}z_2/z_1)-H_2(z_2)\delta(x^2z_2/z_1)+\cdots\big)$.
Adding the first term in (\ref{exchange:L4}) and the fourth term
in (\ref{exchange:L7}) yields
\begin{align}
\beta_+(z)={}&\sum_{\Omega_j \subset J_N}
\sum_{n=m}^{\operatorname{Min}(N, j+m-2)}
\sum\limits_{\substack{l=n+2-m\\s_k=n,\, k=n+1-m
\\{_l=\overline{n}}}}^{j}d_{\Omega_j}(x,r)
{:}\overrightarrow{\Lambda}_{\Omega_j}(z)
\Lambda_{\overline{n}}\big(x^{2N+2-j-2m}z\big){:}\notag
\\
&+\sum_{\Omega_j \subset J_N}\sum_{n=m}^{\operatorname{Min}(N, j+m-1)}
\sum\limits_{\substack{l=n+1-m\\s_k=n, \, k=n+1-m\\
s_l \prec \overline{n} \prec s_{l+1}}}^{j}d_{\Omega_j}(x,r)
\Delta\big(x^{2(l+m-N-1)}\big)\notag
\\
&\phantom{+}\times{:}\overrightarrow{\Lambda}_{\Omega_j}(z)
\Lambda_{\overline{n}}\big(x^{2N+2-j-2m}z\big){:}.
\label{def:beta+}
\end{align}
Adding the second term in (\ref{exchange:L4}) and the second term in
(\ref{exchange:L7}) yields
\begin{align}
{\beta}_-(z)={}&\sum_{\Omega_j \subset J_N}\sum_{n=m+1}^{\operatorname{Min}(N, j+m-1)}
\sum\limits_{\substack{l=n+1-m\\s_k=n,\, k=n-m
\\s_l=\overline{n}}}^{j}d_{\Omega_j}(x,r)
{:}\overrightarrow{\Lambda}_{\Omega_j}(z)\Lambda_{\overline{n}}\big(x^{2N+2-j-2m}z\big){:}\notag
\\
&+\sum_{\Omega_j \subset J_N}\sum_{n=m+1}^{\operatorname{Min}(N, j+m)}
\sum\limits_{\substack{l=n-m\\s_k=n, \, k=n-m
\\s_l \prec \overline{n} \prec s_{l+1}}}^{j}
d_{\Omega_j}(x,r)\Delta\big(x^{2(l+m-N-1)}\big)\notag
\\
&\phantom{+}\times{:}\overrightarrow{\Lambda}_{\Omega_j}(z)\Lambda_{\overline{n}}\big(x^{2N+2-j-2m}z\big){:}.
\label{def:beta-}
\end{align}
Adding the first term in (\ref{exchange:L5}) yields
\begin{gather}
{\gamma}_+(z)=\sum\limits_{\substack{\Omega_j \subset J_N
\\ s_k=0,\, k=N+2-m}}d_{\Omega_j}(x,r)
{:}\Lambda_0\big(x^{2N+2-j-2m}z\big)\overrightarrow{\Lambda}_{\Omega_j}(z){:}.
\label{def:gamma+}
\end{gather}
Adding the second term in (\ref{exchange:L5}) yields
\begin{gather}
{\gamma}_-(z)=\sum\limits_{\substack{\Omega_j \subset J_N\\s_k=0,\, k=N+1-m}}
d_{\Omega_j}(x,r){:}\Lambda_0\big(x^{2N+2-j-2m}z\big)\overrightarrow{\Lambda}_{\Omega_j}(z){:}.
\label{def:gamma-}
\end{gather}
Adding the first term in (\ref{exchange:L3}) and the fourth term in (\ref{exchange:L6}) yields
\begin{align}
\delta_+(z)={}&\sum_{\Omega_j \subset J_N}\sum_{n=2N+3-j-m}^{N}
\sum\limits_{\substack{k=1\\s_k=n\\s_l=\overline{n},\, l=2N+3-m-n}}^{2N+2-m-n}
d_{\Omega_j}(x,r){:}\Lambda_{n}\big(x^{2N+2-j-2m}z\big)
\overrightarrow{\Lambda}_{\Omega_j}(z){:}\notag
\\
&+\sum_{\Omega_j \subset J_N}\sum_{n=2N+3-j-m}^{N}
\sum\limits_{\substack{k=1\\s_{k-1} \prec n \prec s_k
\\s_l=\overline{n},\, l=2N+3-m-n}}^{2N+3-m-n}
d_{\Omega_j}(x,r)\Delta\big(x^{2(N+2-m-k)}\big)\notag
\\
&\phantom{+}\times{:}\Lambda_{n}\big(x^{2N+2-j-2m}z\big)\overrightarrow{\Lambda}_{\Omega_j}(z){:}.
\label{def:delta+}
\end{align}
Adding the second term in (\ref{exchange:L3}) and the second term in (\ref{exchange:L6}) yields
\begin{align}
{\delta}_-(z)={}&\sum_{\Omega_j \subset J_N}\sum_{n=2N+2-j-m}^{N}
\sum\limits_{\substack{k=1\\s_k=n
\\s_l=\overline{n},\, l=2N+2-m-n}}^{2N+1-m-n}d_{\Omega_j}(x,r)
{:}\Lambda_{n}\big(x^{2N+2-j-2m}z\big)\overrightarrow{\Lambda}_{\Omega_j}(z){:}\notag
\\
&+\sum_{\Omega_j \subset J_N}\sum_{n=2N+2-j-m}^{N}
\sum\limits_{\substack{k=1\\s_{k-1}\prec n \prec s_k
\\s_l=\overline{n}, \, l=2N+2-m-n}}^{2N+2-m-n}d_{\Omega_j}(x,r)\Delta\big(x^{2(N+2-m-k)}\big) \notag
\\
&\phantom{+}\times{:}\Lambda_{n}\big(x^{2N+2-j-2m}z\big)\overrightarrow{\Lambda}_{\Omega_j}(z){:}.
\label{def:delta-}
\end{align}
The relation $H_{2N+2-j-2m}(z)=0$ is shown in the same way as
$\overline{H}_{2N+2-j-2m}(z)=0$.
\end{proof}

{\samepage\begin{Lemma}
The {currents} $T_i(z)$ satisfy the following fusion relation
\begin{gather}
\lim_{z_1 \to x^{\pm (i+j)}z_2}
\bigg(1-\frac{x^{\pm (i+j)}z_2}{z_1}\bigg)f_{i,j}\bigg(\frac{z_2}{z_1}\bigg)T_i(z_1)T_j(z_2)\notag
\\ \qquad
{}=\mp c(x,r)\prod_{l=1}^{\operatorname{Min}(i,j)-1}\Delta\big(x^{2l+1}\big)T_{i+j}\big(x^{\pm i}z_2\big),\qquad
1\leq i, j \leq N.
\label{eqn:fusion-T1}
\end{gather}
\end{Lemma}}

\begin{proof}
For 
$\Omega_i^{(1)}=\{s_1, s_2, \dots, s_i\} \subset J_N$
with $s_1\prec s_2 \prec \cdots \prec s_i$ and
$\Omega_j^{(2)}=\{t_1, t_2, \dots, t_j\}\allowbreak \subset J_N$
with $t_1\prec t_2 \prec \cdots \prec t_j$,
we set $\Omega_{i+j}=\Omega_i^{(1)} \cup \Omega_j^{(2)}$.
From (\ref{eqn:Lambda}), the necessary and sufficient condition that
$f_{i,j}(z_2/z_1)\overrightarrow{\Lambda}_{\Omega_i^{(1)}}(z_1)
\overrightarrow{\Lambda}_{\Omega_j^{(2)}}(z_2)$
has a pole at $z_1=x^{-(i+j)}z_2$ (respectively $z_1=x^{i+j}z_2$)
is $s_i \prec t_1$ (respectively $t_j \prec s_1$).
In the case when $s_i \prec t_1$ or $t_j \prec s_1$, we obtain\vspace{-1ex}
\begin{gather*}
f_{i,j}\bigg(\frac{z_2}{z_1}\bigg)
\overrightarrow{\Lambda}_{\Omega_i^{(1)}}(z_1)
\overrightarrow{\Lambda}_{\Omega_j^{(2)}}(z_2)
=\prod_{k=0}^{\operatorname{Min}(i,j)-1}
\Delta\bigg(\frac{x^{{\pm(2k+1-i-j)}}z_2}{z_1}\bigg)
\\ \qquad
{}\times\prod_{p=1}^i\prod\limits_{\substack{q=1\\t_q=\overline{s}_p}}^j
\Delta\bigg(\frac{x^{{\pm \{2(q-p+s_p-N-1+i)-i-j\}}}z_2}{z_1}\bigg){:}
\overrightarrow{\Lambda}_{\Omega_i^{(1)}}(z_1)
\overrightarrow{\Lambda}_{\Omega_j^{(2)}}(z_2){:}.
\end{gather*}
The signs $\pm$
in the products in the above expression
of ${f_{i,j}(z_2/z_1)\overrightarrow{\Lambda}_{\Omega_i^{(1)}}(z_1)
\overrightarrow{\Lambda}_{\Omega_j^{(2)}}(z_2)}$ are in the same order.
The upper sign is for $s_i \prec t_1$, and the lower sign is for $t_j \prec s_1$.
Taking the limit yields\looseness=-1
\begin{gather}
\lim_{z_1 \to x^{\pm (i+j)}z_2}
\bigg(1-\frac{x^{\pm (i+j)}z_2}{z_1}\bigg)f_{i,j}\bigg(\frac{z_2}{z_1}\bigg)
\overrightarrow{\Lambda}_{\Omega_i^{(1)}}(z_1)
\overrightarrow{\Lambda}_{\Omega_j^{(2)}}(z_2)\notag
=\mp c(x,r)\prod_{l=1}^{\operatorname{Min}(i,j)-1}\Delta\big(x^{2l+1}\big)
\\ \qquad
{}\times\prod_{p=1}^i\prod\limits_{\substack{q=1\\t_q=\overline{s}_p}}^j
\Delta\bigg(x^{2(q-p+s_p-N-1+i)}\bigg)
\overrightarrow{\Lambda}_{\Omega_{i+j}}\big(x^{\pm i}z_2\big),\qquad
1\leq i, j \leq N.
\label{eqn:fusion-T4}
\end{gather}
Here, we use ${:}\overrightarrow{\Lambda}_{\Omega_i^{(1)}}\big(x^{\pm (i+j)}z\big)
\overrightarrow{\Lambda}_{\Omega_j^{(2)}}(z){:}=
\overrightarrow{\Lambda}_{\Omega_{i+j}}\big(x^{\pm i}z\big)$.
Adding (\ref{eqn:fusion-T4}) over all $\Omega_i^{(1)}$ and
$\Omega_j^{(2)}$ yields (\ref{eqn:fusion-T1}).
\end{proof}

\begin{Lemma}
The currents $T_i(z)$ satisfy the following fusion relations:\vspace{-1ex}
\begin{gather}
\lim_{z_1 \to x^{\pm (2N+1+i-j)}z_2}
\bigg(1-\frac{x^{\pm (2N+1+i-j)}z_2}{z_1}\bigg)f_{i,j} \bigg(\frac{z_2}{z_1}\bigg)T_i(z_1)T_j(z_2)\notag
\\ \qquad
{}=\mp c(x,r)\prod_{l=1}^{i-1}\Delta\big(x^{2l+1}\big)
\prod_{l=N+1-j}^{N+i-j}\Delta(x^{2l})T_{j-i}\big(x^{\pm i}z_2\big),\qquad
1\leq i \leq j \leq N,
\label{eqn:fusion-T2}
\\
\lim_{z_1 \to x^{\pm (2N+1-i+j)}z_2}
\bigg(1-\frac{x^{\pm (2N+1-i+j)}z_2}{z_1}\bigg)
f_{i,j}\bigg(\frac{z_2}{z_1}\bigg)T_i(z_1)T_j(z_2)\notag
\\
{}=\mp c(x,r)\prod_{l=1}^{j-1}\Delta\big(x^{2l+1}\big)\prod_{l=N+1-i}^{N+j-i}
\Delta\big(x^{2l}\big)T_{i-j}\big(x^{\pm (2N+1-i)}z_2\big),\qquad
1\leq j \leq i \leq N.
\label{eqn:fusion-T3}
\end{gather}
\end{Lemma}

\begin{proof}
Using (\ref{prop:duality}), (\ref{eqn:fusion-f3(2)}), (\ref{eqn:fusion-f3(3)}),
and (\ref{eqn:fusion-T1}) yields (\ref{eqn:fusion-T2}) and (\ref{eqn:fusion-T3}).
\end{proof}

\begin{proof}
Here we will give a proof of Theorem \ref{thm:3-2}.
We prove Theorem~\ref{thm:3-2} by induction.
Lem\-ma~\ref{lemma:4-6} is the base for induction.
We define ${\rm LHS}_{i,j}$, ${\rm RHS1}_{i,j}$ and ${\rm RHS2}_{i,j}(k)$
with $1\leq k \leq i \leq j \leq N$ as
\begin{gather*}
{\rm LHS}_{i,j}=
f_{i,j}\bigg(\frac{z_2}{z_1}\bigg)T_i(z_1)T_j(z_2)
-f_{j,i}\bigg(\frac{z_1}{z_2}\bigg)T_j(z_2)T_i(z_1),
\\
{\rm RHS1}_{i,j}=c(x,r)\prod_{l=1}^{i-1}
\Delta\big(x^{2l+1}\big)\prod_{l=N+1-j}^{N+i-j}\Delta\big(x^{2l}\big)
\\ \hphantom{{\rm RHS1}_{i,j}=}
{}\times\bigg(\delta\bigg(\frac{x^{-2N+j-i-1}z_2}{z_1}\bigg)
T_{j-i}\big(x^{-i}z_2\big)-\delta
\bigg(\frac{x^{2N-j+i+1}z_2}{z_1}\bigg)T_{j-i}\big(x^{i}z_2\big)\bigg),
\\
{\rm RHS2}_{i,j}(k)=c(x,r)\prod_{l=1}^{k-1}\Delta\big(x^{2l+1}\big)
\bigg(\delta\bigg(\frac{x^{-j+i-2k}z_2}{z_1}\bigg)
f_{i-k,j+k}\big(x^{j-i}\big)T_{i-k}\big(x^kz_1\big)T_{j+k}\big(x^{-k}z_2\big)
\\ \hphantom{{\rm RHS2}_{i,j}(k)=}
{}-\delta\bigg(\frac{x^{j-i+2k}z_2}{z_1}\bigg)f_{i-k,j+k}\big(x^{-j+i}\big)
T_{i-k}(x^{-k}z_1)T_{j+k}\big(x^{k}z_2\big)\bigg),\quad
1\leq k \leq i-1,
\\
{\rm RHS2}_{i,j}(i)=c(x,r)\prod_{l=1}^{i-1}\Delta\big(x^{2l+1}\big)
\bigg(\delta\bigg(\frac{x^{-j-i}z_2}{z_1}\bigg)T_{j+i}\big(x^{-i}z_2\big)-
\delta\bigg(\frac{x^{j+i}z_2}{z_1}\bigg)T_{j+i}\big(x^{i}z_2\big)\bigg).
\end{gather*}
We prove the following relation by induction on
$i$, $1\leq i \leq j \leq N$.
\begin{gather}
{\rm LHS}_{i,j}={\rm RHS1}_{i,j}+\sum_{k=1}^i {\rm RHS2}_{i,j}(k).
\label{induction1}
\end{gather}
The base, $i=1 \leq j \leq N$ was proved previously in
Lemma \ref{lemma:4-6}.

We assume that the relation (\ref{induction1}) holds for
some $i$, $1\leq i<j\leq N$, and show that
${\rm LHS}_{i+1,j}={\rm RHS1}_{i+1,j}+\sum_{k=1}^{i+1} {\rm RHS2}_{i+1,j}(k)$ from this assumption.
First, we summarize some relations.
The assumption (\ref{induction1}) yields
\begin{gather}
\lim_{w_1 \to x^{2N-i-j}w_2}
\bigg(1-x^{2N-i-j}\frac{w_2}{w_1}\bigg)f_{j-1,i}\bigg(\frac{w_2}{w_1}\bigg)
T_{j-1}(w_1)T_i(w_2)=0,\label{induction2}
\\
\lim_{w_1 \to x^{2N-i-j}w_2}
\bigg(1-x^{2N-i-j}\frac{w_2}{w_1}\bigg)f_{1,j-i}\bigg(\frac{w_2}{w_1}\bigg)
T_{1}(w_1)T_{j-i}(w_2)=0,\label{induction3}
\\
f_{i,j}\bigg(\frac{w_2}{w_1}\bigg)T_i(w_1)T_j(w_2)=
f_{j,i}\bigg(\frac{w_1}{w_2}\bigg)T_j(w_2)T_i(w_1)
\label{induction4}
\end{gather}
for $\frac{w_2}{w_1}\neq x^{\pm (2N-j+i+1)}$, $x^{\pm(j-i+2k)}$, $1\leq k \leq i$.
Direct calculation yields
\begin{gather}
\lim_{w_2 \to x^{-i-1}w_1}
\bigg(1-x^{-i-1}\frac{w_1}{w_2}\bigg)\Delta\bigg(x^{-i}\frac{w_1}{w_2}\bigg)=c(x,r).
\label{induction5}
\end{gather}

Multiplying ${\rm LHS}_{i,j}$ by
$f_{1,i}(z_1/z_3)f_{1,j}(z_2/z_3)T_1(z_3)$
on the left and using the quadratic rela\-tion~(\ref{induction1}) with $i=1$,
along with the fusion relation (\ref{eqn:fusion-f4}) yields
\begin{gather}
f_{1,i}\bigg(\frac{z_1}{z_3}\bigg)
f_{1,j}\bigg(\frac{z_2}{z_3}\bigg)T_1(z_3)
\times {\rm LHS}_{i, j}\notag
\\ \qquad
{}=f_{1,j}\bigg(\frac{z_2}{z_3}\bigg)
f_{i,j}\bigg(\frac{z_2}{z_1}\bigg)
f_{1,i}\bigg(\frac{z_1}{z_3}\bigg)T_1(z_3)T_i(z_1)T_j(z_2)\notag
\\ \qquad\hphantom{=}
{}-f_{j,1}\bigg(\frac{z_3}{z_2}\bigg)
f_{j,i}\bigg(\frac{z_1}{z_2}\bigg)T_j(z_2)
f_{1,i}\bigg(\frac{z_1}{z_3}\bigg)T_1(z_3)T_i(z_1)\notag
\\ \qquad\hphantom{=}
{}-c(x,r)\Delta\big(x^{2(N+1-j)}\big)
\delta\bigg(\frac{x^{-2N+j-2}z_2}{z_3}\bigg)
\Delta\bigg(\frac{x^{-i}z_1}{z_3}\bigg)
f_{j-1,i}\bigg(\frac{x^{-2N+j-1}z_1}{z_3}\bigg)\notag
\\ \qquad\hphantom{=-}
{}\times T_{j-1}\big(x^{2N-j+1}z_3\big)T_i(z_1)
+c(x,r)\Delta\big(x^{2(N+1-j)}\big)\notag
\\ \qquad\hphantom{=-}
{}\times\delta\bigg(\frac{x^{2N-j+2}z_2}{z_3}\bigg)
\Delta\bigg(\frac{x^{i}z_1}{z_3}\bigg)
f_{j-1,i}\bigg(\frac{x^{2N-j+1}z_1}{z_3}\bigg)
T_{j-1}\big(x^{-2N+j-1}z_3\big)T_i(z_1)\notag
\\ \qquad\hphantom{=}
{}-c(x,r)\delta\bigg(\frac{x^{-j-1}z_2}{z_3}\bigg)
\Delta\bigg(\frac{x^{-i}z_1}{z_3}\bigg)f_{j+1,i} \bigg(\frac{x^{-j}z_1}{z_3}\bigg)T_{j+1}\big(x^jz_3\big)T_i(z_1)\notag
\\ \qquad\hphantom{=}
{}+c(x,r)\delta\bigg(\frac{x^{j+1}z_2}{z_3}\bigg)
\Delta\bigg(\frac{x^{i}z_1}{z_3}\bigg)
f_{j+1,i}\bigg(\frac{x^{j}z_1}{z_3}\bigg)T_{j+1}\big(x^{-j}z_3\big)T_i(z_1).
\label{induction6}
\end{gather}
Taking the limit $z_3\to x^{-i-1}z_1$ of (\ref{induction6}) multiplied by
$c(x,r)^{-1} \big(1-x^{-i-1}z_1/z_3\big)$ and using the relations
(\ref{eqn:fusion-T1}) and (\ref{eqn:fusion-T3}),
(\ref{induction2}), (\ref{induction4}), and (\ref{induction5}) yields
\begin{gather}
\lim_{z_3 \to x^{-i-1}z_1}\frac{1}{c(x,r)}\bigg(1-x^{-i-1}\frac{z_1}{z_3}\bigg)
f_{1,i}\bigg(\frac{z_1}{z_3}\bigg)f_{1,j}\bigg(\frac{z_2}{z_3}\bigg)T_1(z_3)
\times {\rm LHS}_{i, j}\notag
\\ \qquad
{}=f_{i+1,j}\bigg(\frac{xz_2}{z_1}\bigg)T_{i+1}\big(x^{-1}z_1\big)T_j(z_2)-
f_{j,i+1}\bigg(\frac{x^{-1}z_1}{z_2}\bigg)T_j(z_2)T_{i+1}\big(x^{-1}z_1\big)\notag
\\ \qquad\hphantom{=}
{}+c(x,r) \prod_{l=1}^i \Delta\big(x^{2l+1}\big)\prod_{l=N-j+1}^{N+i+1-j}
\Delta\big(x^{2l}\big)\delta\bigg(\frac{x^{2N-j+i+3}z_2}{z_1}\bigg)T_{j-i-1}\big(x^{i+1}z_2\big)\notag
\\ \qquad\hphantom{=}
{}-c(x,r)\delta\bigg(\frac{x^{i-j}z_2}{z_1}\bigg)
f_{i,j+1}\big(x^{-i+j+1}\big)T_i(z_1)T_{j+1}\big(x^{-1}z_2\big)\notag
\\ \qquad\hphantom{=}
{}+c(x,r)\delta\bigg(\frac{x^{i+j+2}z_2}{z_1}\bigg)\prod_{l=1}^{i}
\Delta\big(x^{2l+1}\big)T_{i+j+1}\big(x^{i+1}z_2\big).
\label{induction7}
\end{gather}

{\sloppy Multiplying ${\rm RHS1}_{i,j}$ by $f_{1,i}(z_1/z_3)f_{1,j}(z_2/z_3)T_1(z_3)$ from the left and using fusion re\-la\-tions~(\ref{eqn:fusion-f3(3)}) and (\ref{eqn:fusion-f4}) yields
\begin{gather}
f_{1,i}\bigg(\frac{z_1}{z_3}\bigg)f_{1,j}\bigg(\frac{z_2}{z_3}\bigg)T_1(z_3)
\times {\rm RHS1}_{i, j}=
c(x,r)\prod_{l=1}^{i-1}\Delta(x^{2l+1})\prod_{l=N+1-j}^{N+i-j}\Delta\big(x^{2l}\big)\notag
\\ \qquad
{}\times
\bigg\{\delta\bigg(\frac{x^{-2N+j-i-1}z_2}{z_1}\bigg)
\Delta\bigg(\frac{x^i z_1}{z_3}\bigg)f_{1,j-i}\bigg(\frac{x^{-i}z_2}{z_3}\bigg) T_1(z_3)T_{j-i}\big(x^{-i}z_2\big)\notag
\\ \qquad
{}-\delta\bigg(\frac{x^{2N-j+i+1}z_2}{z_1}\bigg)
\Delta\bigg(\frac{x^{-i} z_1}{z_3}\bigg)
f_{1,j-i}\bigg(\frac{x^{i}z_2}{z_3}\bigg)T_1(z_3)T_{j-i}\big(x^{i}z_2\big)\bigg\}.
\label{induction8}
\end{gather}}%
Taking the limit $z_3\to x^{-i-1}z_1$ of (\ref{induction8}) multiplied by $c(x,r)^{-1} (1-x^{-i-1}z_1/z_3)$ and using the relations
(\ref{eqn:fusion-T2}) and (\ref{induction3}) yields
\begin{gather}
\lim_{z_3 \to x^{-i-1}z_1}\frac{1}{c(x,r)}\bigg(1-x^{-i-1}\frac{z_1}{z_3}\bigg)
f_{1,i}\bigg(\frac{z_1}{z_3}\bigg)
f_{1,j}\bigg(\frac{z_2}{z_3}\bigg)T_1(z_3)\times {\rm RHS1}_{i, j}\notag
\\ \qquad
{}=c(x,r)\prod_{l=1}^i \Delta\big(x^{2l+1}\big)
\prod_{l=N+1-j}^{N+i+1-j}\Delta\big(x^{2l}\big)
\delta\bigg(\frac{x^{-2N+j-i-1}z_2}{z_1}\bigg)T_{j-i-1}\big(x^{-i-1}z_2\big).
\label{induction9}
\end{gather}

Multiplying ${\rm RHS2}_{i,j}(i)$ by $f_{1,i}(z_1/z_3)f_{1,j}(z_2/z_3)T_1(z_3)$ from the left and using the fusion relation (\ref{eqn:fusion-f5}) yields
\begin{gather}
f_{1,i}\bigg(\frac{z_1}{z_3}\bigg)f_{1,j}\bigg(\frac{z_2}{z_3}\bigg)T_1(z_3)
\times {\rm RHS2}_{i, j}(i)\notag
\\ \qquad
{}=c(x,r)\prod_{l=1}^{i-1}\Delta_1(x^{2l+1})\bigg(\delta\bigg(\frac{x^{-i-j}z_2}{z_1}\bigg)
f_{1,i+1}\bigg(\frac{x^jz_1}{z_3}\bigg)\Delta\bigg(\frac{x^i z_1}{z_3}\bigg)T_1(z_3)T_{i+j}\big(x^jz_1\big)
\notag
\\ \qquad\hphantom{=}
{}-\delta\bigg(\frac{x^{i+j}z_2}{z_1}\bigg)f_{1,i+1}\bigg(\frac{x^{-j}z_1}{z_3}\bigg)
\Delta\bigg(\frac{x^{-i} z_1}{z_3}\bigg)T_1(z_3)T_{i+j}\big(x^{-j}z_1\big)\bigg).
\label{induction10}
\end{gather}
Taking the limit $z_3\to x^{-i-1}z_1$ of (\ref{induction10})
multiplied by $c(x,r)^{-1}\big(1-x^{-i-1}z_1/z_3\big)$
and using the relations (\ref{eqn:fusion-T1}) and (\ref{induction5}) yields
\begin{gather}
\lim_{z_3 \to x^{-i-1}z_1}\frac{1}{c(x,r)}\bigg(1-x^{-i-1}\frac{z_1}{z_3}\bigg)
f_{1,i}\bigg(\frac{z_1}{z_3}\bigg)f_{1,j}\bigg(\frac{z_2}{z_3}\bigg)T_1(z_3)
\times {\rm RHS2}_{i, j}(i)\notag
\\ \qquad
{}=c(x,r)\delta\bigg(\frac{x^{-i-j}z_2}{z_1}\bigg)\prod_{l=1}^i \Delta\big(x^{2l+1}\big)T_{i+j+1}\big(x^{-i-1}z_2\big)\notag
\\ \qquad \hphantom{=}
{}-c(x,r)\delta\bigg(\frac{x^{i+j}z_2}{z_1}\bigg)
\prod_{l=1}^{i-1} \Delta\big(x^{2l+1}\big)f_{1,i+j}\big(x^{i-j+1}\big)T_1\big(x^{-i-1}z_1\big)T_{i+j}\big(x^{i}z_2\big).
\label{induction11}
\end{gather}
Multiplying ${\rm RHS2}_{i,j}(k)$, $1\leq k \leq i-1$, by
$f_{1,i}(z_1/z_3)f_{1,j}(z_2/z_3)T_1(z_3)$
from the left and using relations (\ref{eqn:fusion-f6})
and (\ref{induction4}) yields
\begin{gather}
f_{1,i}\bigg(\frac{z_1}{z_3}\bigg)
f_{1,j}\bigg(\frac{z_2}{z_3}\bigg)T_1(z_3)
\times {\rm RHS2}_{i, j}(k)\label{induction12}
\\ \qquad
{}=c(x,r)\prod_{l=1}^{k-1}\Delta\big(x^{2l+1}\big)
\bigg(\delta\bigg(\frac{x^{-j+i-2k}z_2}{z_1}\bigg)
f_{1,i-k}\bigg(\frac{x^kz_1}{z_3}\bigg)f_{j+k,i-k}\big(x^{i-j}\big)f_{1,j+k} \notag
\\ \qquad \hphantom{=}
{}\times \bigg(\frac{x^{-i+j+k}z_1}{z_3}\bigg)
T_1(z_3)T_{j+k}\big(x^{j-i+k}z_1\big)T_{i-k}\big(x^kz_1\big)
-\delta\bigg(\frac{x^{j-i+2k}z_2}{z_1}\bigg)
f_{1,i-k}\bigg(\frac{x^{-k}z_1}{z_3}\bigg)\notag
\\ \qquad \hphantom{=}
{}\times f_{i-k,j+k}\big(x^{i-j}\big)
f_{1,j+k}\bigg(\frac{x^{i-j-k}z_1}{z_3}\bigg)
T_1(z_3)T_{i-k}\big(x^{-k}z_1\big)T_{j+k}\big(x^kz_2\big)\bigg), \!\!\quad
1\leq k \leq i-1.\notag
\end{gather}
Taking the limit $z_3\to x^{-i-1}z_1$ of (\ref{induction12}) multiplied by $c(x,r)^{-1}\big(1-x^{-i-1}z_1/z_3\big)$, and using the fusion relations (\ref{eqn:fusion-f3}), (\ref{eqn:fusion-T1}), and (\ref{induction4}) yields
\begin{gather}
\lim_{z_3 \to x^{-i-1}z_1}\frac{1}{c(x,r)}\bigg(1-x^{-i-1}\frac{z_1}{z_3}\bigg)
f_{1,i}\bigg(\frac{z_1}{z_3}\bigg)
f_{1,j}\bigg(\frac{z_2}{z_3}\bigg)T_1(z_3)
\times {\rm RHS2}_{i, j}(k)\notag
\\ \qquad
{}=c(x,r)\prod_{l=1}^k \Delta\big(x^{2l+1}\big)
\delta\bigg(\frac{x^{-j+i-2k}z_2}{z_1}\bigg)
f_{j+k-1,i-k}\big(x^{i-j+1}\big)T_{i-k}(x^kz_1)T_{j+k+1}\big(x^{-k-1}z_2\big)\nonumber
\\ \qquad\hphantom{=}
{}-c(x,r)\prod_{l=1}^{k-1}\Delta\big(x^{2l+1}\big)
\delta\bigg(\frac{x^{j-i+2k}z_2}{z_1}\bigg)
f_{i-k+1,j+k}\big(x^{i-j+1}\big)T_{i-k+1}\big(x^{-k-1}z_1\big)\notag
\\ \qquad\hphantom{=+}
{}\times T_{j+k}\big(x^kz_2\big),\qquad 1\leq k \leq i-1.
\label{induction13}
\end{gather}
Adding (\ref{induction7}), (\ref{induction9}), (\ref{induction11}),
and (\ref{induction13}) for $1\leq k \leq i-1$, and replacing
$z_1$ by $xz_1$ yields ${\rm LHS}_{i+1,j}={\rm RHS1}_{i+1,j}+
\sum_{k=1}^{i+1}{\rm RHS2}_{i+1,j}(k)$.
By induction on $i$, we proved quadratic relation~(\ref{thm:quadratic}).
\end{proof}

\subsection{Proof of Lemma \ref{lem:3-4}}

\begin{Lemma}
\label{lemma:4-9}
The current $T_1(z)$ commutes with
the screening currents $S_k(w)$ as follows.
\begin{gather}
[T_1(z), S_k(w)]=C_k(z)
(D_{x^{r}}\delta)\bigg(\frac{x^{k}w}{z}\bigg)+\overline{C}_k(z)
(D_{x^{r}}\delta)\bigg(\frac{x^{2N+1-k}w}{z}\bigg),\quad\
1\leq k \leq N.
\label{eqn:commute-TS}
\end{gather}
Here we set $q$-difference
\begin{gather*}
(D_q\delta)(z)=\delta(qz)-\delta\big(q^{-1}z\big),
\end{gather*}
the currents $C_k(z)$ and $\overline{C}_k(z)$,
$1\leq k \leq N$, are given by
\begin{gather*}
C_k(z)=x^{-r+1}\big(x^{r-1}-x^{-r+1}\big){:}\Lambda_k(z)S_k\big(x^{r-k}z\big){:},
\\
\overline{C}_k(z)=x^{r-1}\big(x^{r-1}-x^{-r+1}\big)
{:}\Lambda_{\overline{k}}(z)S_k\big(x^{-2N{-}1+k-r}z\big){:}.
\end{gather*}
\end{Lemma}

\begin{proof}
Adding (\ref{exchange:Lambda-S}) yields
\begin{align*}
[T_1(z),S_k(w)]={}&\big(x^{r-1}-x^{-r+1}\big)
\bigg({-}x^{-r+1}{:}\Lambda_k(z)S_k(w){:}\delta\bigg(x^{k-r}\frac{w}{z}\bigg)
\\
&+x^{r-1}{:}\Lambda_{k+1}(z)S_k(w){}\delta\bigg(x^{k+r}\frac{w}{z}\bigg)
+x^{r-1}{:}\Lambda_{\overline{k}}(z)S_k(w){}\delta\bigg(x^{2N+1-k+r}\frac{w}{z}\bigg)
\\
&-x^{-r+1}{:}\Lambda_{\overline{k+1}}(z)S_k(w){:}
\delta\bigg(x^{2N+1-k-r}\frac{w}{z}\bigg)\bigg),
\qquad
1\leq k \leq N-1,
\end{align*}
and
\begin{align*}
[T_1(z),S_N(w)]={}&\big(x^{-r+1}-x^{r-1}\big)
\bigg(x^{-r+1}{:}\Lambda_N(z)S_N(w){:}\delta\bigg(x^{N-r}\frac{w}{z}\bigg)
\\
&-x^{r-1}{:}\Lambda_N(z)S_{\overline{N}}(w){:}
\delta\bigg(x^{N+1+r}\frac{w}{z}\bigg)\bigg)
\\
&+\frac{[r-1]_x[\frac{1}{2}]_x}{[r-\frac{1}{2}]_x}\big(x-x^{-1}\big)
\bigg(\delta\bigg(x^{N+r}\frac{w}{z}\bigg)-\delta\bigg(x^{N+1-r}\frac{w}{z}
\bigg)\bigg){:}\Lambda_0(z)S_N(w){:}.
\end{align*}
Using the relations
\begin{gather*}
x^{-r+1}{:}\Lambda_k(z)S_k\big(x^{r-k}z\big){:}
=x^{r-1}{:}\Lambda_{k+1}(z)S_k\big(x^{-r-k}z\big){:},\qquad
1\leq k \leq N-1,
\\
x^{-r+1}{:}\Lambda_N(z)S_N\big(x^{r-N}z\big){:}
=\frac{\big[\frac{1}{2}\big]_x}{\big[r-\frac{1}{2}\big]_x}
{:}\Lambda_{0}(z)S_N\big(x^{-r-N}z\big){:},
\\
x^{r-1}{:}\Lambda_{\overline{k}}(z)S_k\big(x^{-2N{-}1+k-r}z\big){:}
=x^{-r+1}{:}\Lambda_{\overline{k+1}}(z)
S_k\big(x^{-2N{-}1+k+r}z\big){:},\qquad 1\leq k \leq N-1,
\\
x^{r-1}{:}\Lambda_{\overline{N}}(z)S_N\big(x^{-r-N-1}z\big){:}
=\frac{\big[\frac{1}{2}\big]_x}{\big[r-\frac{1}{2}\big]_x}
{:}\Lambda_{0}(z)S_N\big(x^{r-N-1}z\big){:}.
\end{gather*}
yields (\ref{eqn:commute-TS}).
\end{proof}

\begin{Corollary}
The current $T_1(z)$ commutes with the screening operators $S_k$
\begin{gather}
[T_1(z), S_k]=0,\qquad 1\leq k \leq N.
\label{eqn:screening2}
\end{gather}
\end{Corollary}

\begin{proof}
From (\ref{eqn:commute-TS}), we obtain
\begin{gather*}
[T_1(z), S_k]=\oint \frac{{\rm d}w}{2\pi \sqrt{-1}w}
\bigg(C_k(z)(D_{x^{r}}\delta)\bigg(\frac{x^{k}w}{z}\bigg)+\overline{C}_k(z)
(D_{x^{r}}\delta)\bigg(\frac{x^{2N+1-k}w}{z}\bigg)\bigg).
\end{gather*}
Using
$\oint \frac{{\rm d}w}{2\pi \sqrt{-1}w}
(D_{x^{r}}\delta)\big(\frac{x^{s}w}{z}\big)=0$ with $s=k$, $2N+1-k$
yields $[T_1(z), S_k]=0$.
\end{proof}

\begin{proof}
Here we will give a proof of Lemma \ref{lem:3-4}.
Set $T_j(z)=\sum_{m \in \mathbb{Z}}T_j[m]z^{-m}$,
$1\leq j \leq 2N$ and $f_{i,j}(z)=\sum_{l=0}^\infty f_{i,j}^lz^l$.
From (\ref{eqn:quadratic2}), we obtain
\begin{gather*}
\big(x^{-(j+1)k+m}-x^{(j+1)k-m}\big)T_{j+1}[m]
\\ \qquad
{}=\Delta\big(x^{2N+2-2j}\big)\big(x^{(2N-j+2)k-m}-x^{(-2N+j-2)k+m}\big)T_{j-1}[m]
\\ \qquad\phantom{=}
{}+c(x,r)^{-1}\sum_{l=0}^\infty
\big(f_{1,j}^l T_1[k-l]T_j[l-k+m]-f_{j,1}^l
T_j[k-l-m]T_1[l-k]\big),
\\
m, k\in \mathbb{Z},\qquad
1\leq j \leq N.
\end{gather*}
Hence, $T_{j+1}[m]$, $m \in \mathbb{Z}, 1\leq j \leq N$,
are expressed in terms of $T_{j}[n], T_{j-1}[n]$,
and $T_{1}[n]$, $n \in \mathbb{Z}$, $1\leq j \leq N$.
From duality (\ref{prop:duality}),
$T_j[m]$, $m \in \mathbb{Z}$, $N+2 \leq j \leq 2N$
are expressed in terms of $T_{2N+1-j}[n]$, $n \in \mathbb{Z}$,
$N+2\leq j \leq 2N$.
Finally, $T_{j}[m]$, $m \in \mathbb{Z}$, $1\leq j \leq 2N$
are expressed in terms of $T_1[n]$, $n \in \mathbb{Z}$.
Hence, we obtain (\ref{eqn:screening}) from (\ref{eqn:screening2}).
\end{proof}

\section{Conclusion and discussion}
\label{section:5}

In this paper, we obtained the free field construction of
higher $W$-currents $T_i(z)$, $i \geq 2$,
of the deformed $W$-algebra ${\mathcal W}_{x,r}\big(A_{2N}^{(2)}\big)$.
We obtained a closed set of quadratic relations for the $W$-currents $T_i(z)$,
which are completely different from those in types $A_N^{(1)}$ and $A(M,N)^{(1)}$.
The quadratic relations of ${\mathcal W}_{x,r}\big(A_{2N}^{(2)}\big)$
do not preserve ``parity'', though those of ${\mathcal W}_{x,r}\big(A_N^{(1)}\big)$
and ${\mathcal W}_{x,r}\big(A(M,N)^{(1)}\big)$ do. Here we define
``parity'' of $T_i(z)T_j(w)$ as $i+j$.
We obtained the duality $T_{2N+1-i}(z)=c_i T_{i}(z)$, $1 \leq i \leq N$, which
is a new structure that does not occur in
types~$A_2^{(2)}$,~$A_N^{(1)}$, and~$A(M,N)^{(1)}$.
This allowed us to define the deformed $W$-algebra
${\mathcal W}_{x,r}\big(A_{2N}^{(2)}\big)$ using generators and relations
similarly to the definition of the twisted affine {Lie} algebra
of type $A_{2N}^{(2)}$ given in Section \ref{section:2}.

We also justified our definition of
the deformed $W$-algebra of type $A_{2N}^{(2)}$.
We compare Definition \ref{def:3-3} with other definitions.
In~\cite{Frenkel-Reshetikhin1},
the deformed $W$-algebras of type{s}
$A_N^{(1)}$, $B_N^{(1)}$, $C_N^{(1)}$, $D_N^{(1)}$, and $A_{2N}^{(2)}$
were proposed as the intersection of the kernels of the screening operators.
We recall the definition based on the screening operators for
$A_{2N}^{(2)}$.
Let {{${\mathbf H}_{x,r}$}}
be the vector space spanned by
the formal power series currents of the form
\begin{gather*}
{:}\partial_z^{n_1} Y_{i_1}\big(x^{r j_1+k_1}{z}\big)^{\varepsilon_1}\cdots
\partial_z^{n_l} Y_{i_l}\big(x^{r j_l+k_l}{z}\big)^{\varepsilon_l}{:},
\end{gather*}
where $\varepsilon_i=\pm 1$
\footnote[2]{We define $Y_i(z)^{-1}$ as the inverse element of $Y_i(z)$, that is,
$Y_i(z) Y_i(z)^{-1}=Y_i(z)^{-1}Y_i(z)=1$. Specifically, we obtain
$Y_i(z)^{-1}=x^{-ry_i(0)}\langle Y_i(z)Y_i(z)\rangle
{:}\exp\bigl({-}\sum_{m \neq 0}y_i(m)z^{-m}\bigr){:}$,
where we used the symbol $\langle~\rangle$ defined in~(\ref{def:coupling}).}.
We define ${\mathbf W}_{x,r}$ as the vector subspace of~${\bf H}_{x,r}$
consisting of all currents that commute with the screening operators
$S_i$, $1\!\leq \!i \leq\! N$, in (\ref{def:screening}).
Let $\bigl\{F_a(z)\!=\!\sum_{m \in \mathbb{Z}}F_a[m]z^{-m}\bigr\}_{a \in A}$ be a basis of
the vector space ${\bf W}_{x,r}$. Let ${\mathcal W}^{\rm FR}$ be the associative algebra generated by
elements~$F_a[m]$, $m \in \mathbb{Z}$, $a \in A$.
Let~$J_K$ be the left ideal of ${\mathcal W}^{\rm FR}$
generated by elements $F_a[m]$, $m \geq K \in \mathbb{N}$, $a\in A$.
We define the deformed $W$-algebra
\begin{gather*}
{\mathcal W}_{x,r}^{\rm FR}\big(A_{2N}^{(2)}\big)=\lim_{\leftarrow}{\mathcal W}^{\rm FR}/J_K.
\end{gather*}

We propose another definition of the deformed $W$-algebra.
From (\ref{eqn:screening}), the $W$-currents
$T_i(z)=\sum_{m \in \mathbb{Z}}T_i[m]z^{-m}$, $1 \leq i \leq 2N$,
commute with the screening operators.
Let ${\mathcal W}^{\rm AKOS}$ be the associative algebra generated by
elements $T_i[m]$, $m \in \mathbb{Z}$, $1\leq i \leq 2N$.
Let $L_K$ be the left ideal of ${\mathcal W}^{\rm AKOS}$ generated by elements
$T_i[m]$, $m \geq K \in \mathbb{N}$, $1\leq i \leq 2N$.
We define the deformed $W$-algebra
\begin{gather*}
{\mathcal W}_{x,r}^{\rm AKOS}\big(A_{2N}^{(2)}\big)=\lim_{\leftarrow}{\mathcal W}^{\rm AKOS}/L_K.
\end{gather*}

In this study, our definitions ${\mathcal W}_{x,r}\big(A_{2N}^{(2)}\big)$
were based on generators and relations.
We have introduced three definitions of the deformed $W$-algebra
for the twisted algebra of the type $A_{2N}^{(2)}$.

\begin{Conjecture}
\label{conj:5-1}
${\mathcal W}_{x,r}\big(A_{2N}^{(2)}\big)$,
${\mathcal W}_{x,r}^{\rm AKOS}\big(A_{2N}^{(2)}\big)$, and
${\mathcal W}_{x,r}^{\rm FR}\big(A_{2N}^{(2)}\big)$
are isomorphic as associative algebras
\begin{gather}
{\mathcal W}_{x,r}\big(A_{2N}^{(2)}\big)\cong
{\mathcal W}_{x,r}^{\rm AKOS}\big(A_{2N}^{(2)}\big)\cong
{\mathcal W}_{x,r}^{\rm FR}\big(A_{2N}^{(2)}\big).
\label{eqn:conjecture}
\end{gather}
\end{Conjecture}

The author believes that this conjecture can be extended to
arbitrary affine Lie algebras.
Some necessary conditions of isomorphism (\ref{eqn:conjecture})
in Conjecture \ref{conj:5-1}
can be indicated immediately.
From (\ref{eqn:screening}), we obtain the following inclusion:
\begin{gather*}
{\mathcal W}_{x,r}^{\rm AKOS}\big(A_{2N}^{(2)}\big) \subseteq
{\mathcal W}_{x,r}^{\rm FR}\big(A_{2N}^{(2)}\big).
\end{gather*}
We establish a homomorphism of associative algebras
$\varphi \in {\rm Hom}_{{\bf C}}\big({\mathcal W}_{x,r}\big(A_{2N}^{(2)}\big),
{\mathcal W}_{x,r}^{\rm AKOS}\big(A_{2N}^{(2)}\big)\big)$
using $\varphi(\overline{T}_i[m])=T_i[m]$. $\varphi$ is surjective,
\begin{gather*}
\varphi\big({\mathcal W}_{x,r}\big(A_{2N}^{(2)}\big)\big)
={\mathcal W}_{x,r}^{\rm AKOS}\big(A_{2N}^{(2)}\big).
\end{gather*}
If we assume that $\varphi$ is injective,
the isomorphism on the left side in (\ref{eqn:conjecture}) is obtained.
In other words, no independent relations other than (\ref{prop:duality})
and (\ref{thm:quadratic}) exist
in ${\mathcal W}_{x,r}\big(A_{2N}^{(2)}\big)$.
We propose two results to support this claim.
In the classical limit the second Hamiltonian structure $\{\cdot, \cdot\}$
of the $q$-Poisson algebra
\cite{Frenkel-Reshetikhin2,Frenkel-Reshetikhin1, Frenkel-Reshetikhin-Semenov,
Semenov-Sevostyanov}
was obtained from the quadratic relations
(see (\ref{def:q-Poisson}) and~(\ref{def:q-Poisson2})).
In the conformal limit all defining relations of the $W$-algebra
${\mathcal W}_\beta\big(A_N^{(1)}\big)$, $N=1, 2$,
are obtained from the quadratic relations of
${\mathcal W}_{x, r}\big(A_N^{(1)}\big)$ {upon the assumption
that the currents $T_i(z)$ have the form of expansion
for small parameter $\hbar$}
(see~\cite[Appendix]{Awata-Kubo-Odake-Shiraishi2}).

The definition of the deformed $W$-algebra ${\mathcal W}_{x,r}(\mathfrak{g})$
for non-twisted affine Lie algebra ${\mathfrak g}$
was formulated in terms of~the quantum Drinfeld--Sokolov reduction in~\cite{Sevostyanov}.
Formulating the definition of
the deformed $W$-algebras ${\mathcal W}_{x,r}({\mathfrak g})$
in terms of the quantum Drinfeld--Sokolov reduction
for~twisted affine Lie algebra or affine Lie superalgebra
\cite{Ding-Feigin,Feigin-Jimbo-Mukhin-Vilkoviskiy,
Harada-Matsuo-Noshita-Watanabe, Kojima2, Kojima1} is still a problem that needs to be solved.

It remains an open challenge
to identify quadratic relations of the deformed $W$-algebras
${\mathcal W}_{x,r}({\mathfrak g})$ for the affine Lie algebras ${\mathfrak g}$
except for types $A_N^{(1)}$ and $A_{2N}^{(2)}$. We believe
that this paper presents a key step towards extending our construction for general affine Lie algebras
${\mathfrak g}$. In \cite{Frenkel-Reshetikhin1} and
\cite{Feigin-Jimbo-Mukhin-Vilkoviskiy} the free field construction of the basic $W$-current $T_1(z)$
of ${\mathcal W}_{x,r}\big({\mathfrak g}\big)$
was suggested in the case when
the underlying simple finite-dimensional Lie algebra
$\stackrel{\circ}{\mathfrak g}$ is of classical type,
\begin{gather*}
T_1(z)=\begin{cases}
\Lambda_1(z)+\cdots \Lambda_{N+1}(z)&{{\rm for}~{\mathfrak g}~{\rm of~type}}~A_N^{(1)},
\\
\Lambda_1(z)+\cdots+\Lambda_N(z)+\Lambda_0(z)
\\[-.5ex]
\hphantom{\Lambda_1(z)}+\Lambda_{\overline{N}}(z)+\cdots+\Lambda_{\overline{1}}(z)&
{{\rm for}~{\mathfrak g}~{\rm of~types}}~B_N^{(1)}, A_{2N}^{(2)}, D_{N+1}^{(2)},
\\
\Lambda_1(z)+\cdots+\Lambda_N(z)+
\Lambda_{\overline{N}}(z)+\cdots+\Lambda_{\overline{1}}(z)
&{{\rm for}~{\mathfrak g}~{\rm of~types}}~C_N^{(1)}, D_N^{(1)}, A_{2N-1}^{(2)}.
\end{cases}
\end{gather*}
Here we omit details of free field constructions of $\Lambda_i(z)$.
The free field construction of $T_1(z)$ has similar form to that
for ${\mathfrak g}$ of type $A_{2N}^{(2)}$
except for the case of $A_N^{(1)}$. Therefore, we expect that
a similar duality as (\ref{prop:duality}) and similar quadratic relation (\ref{thm:quadratic})
hold in all cases in types $B_N^{(1)}$, $C_N^{(1)}$, $D_N^{(1)}$, $A_{2N-1}^{(2)}$, and
$D_{N+1}^{(2)}$. We would like to draw your attention to the following analogy.
Let ${\mathfrak g}$ be an affine Lie algebras of one of the types
$B_N^{(1)}$, $C_N^{(1)}$, $D_N^{(1)}$, $A_{2N-1}^{(2)}$, or~$D_{N+1}^{(2)}$.
Let $\stackrel{\circ}{\mathfrak g}$ be
the underlying simple finite-dimensional Lie algebra.
Let $\stackrel{\circ}{\mathfrak h}$ be a Cartan subalgebra of~$\stackrel{\circ}{\mathfrak g}$.
Let $\overline{\Lambda}_1, \overline{\Lambda}_2, \dots, \overline{\Lambda}_l$
be the fundamental weights of $\stackrel{\circ}{\mathfrak g}$, where~$l$ is the dimension of $\stackrel{\circ}{\mathfrak h}$.
Let~$V_{\overline{\Lambda}_1}$ be the integrable highest weight representation of
$U_q(\stackrel{\circ}{\mathfrak g})$ with the highest weight $\overline{\Lambda}_1$.
Let~$V$ be the evaluation representation corresponding to $V_{\overline{\Lambda}_1}$
of the quantum affine algebra $U_q({\mathfrak g})$
with a~spectral parameter $z \in {\mathbf C}^\times$.
Let $n$ be the dimension of $V_{\overline{\Lambda}_1}$.
We have $\stackrel{n-i}{\wedge} V \simeq\bigl(\stackrel{i}{\wedge} V \bigr)^* \simeq
\stackrel{i}{\wedge} V^*$, because $\stackrel{n}{\wedge} V \simeq {\mathbf C}$.
The evaluation representation $V$ of $U_q\big({\mathfrak g}\big)$ is self-dual except for
${\mathfrak g}$ of type~$A_N^{(1)}$. Hence, we obtain the duality of the representations of
$U_q\bigl({\mathfrak g}\bigr)$,
\begin{gather*}
\stackrel{n-i}{\wedge} V \simeq
\stackrel{i}{\wedge} V\qquad {\rm if}\quad {\mathfrak g}~{\rm is~not~of~type}~A_N^{(1)},
\end{gather*}
which is similar as {that} in (\ref{prop:duality}).
As an analogy, we expect the duality of the $W$-currents,
\begin{gather*}
T_{n-i}(z)=c_i T_i(z)\qquad {\rm if}\quad
{\mathfrak g}~{\rm is~not~of~type}~A_N^{(1)},
\end{gather*}
for the deformed $W$ algebras ${\mathcal W}_{x,r}({\mathfrak g})$.
Here $c_i$, $0 \leq i \leq n$, are constants.

It remains an open challenge
to identify quadratic relations of the deformed $W$-algebras
${\mathcal W}_{x,r}({\mathfrak g})$ for affine
superalgebra ${\mathfrak g}$
except for those of type $A(M,N)^{(1)}$.
Recently the deformed $W$-superalgebra
${\mathcal W}_{x,r}({\mathfrak g})$
has appeared in the study of D-branes and physical interest is growing to this subject, see, e.g.,
\cite{Harada-Matsuo-Noshita-Watanabe}.
As revealed in~\cite{Harada-Matsuo-Noshita-Watanabe, Kojima2, Kojima1},
it is expected that, in cases of superalgebras ${\mathfrak g}$,
infinite number of higher $W$-currents $T_i(z)$, $i=1,2,3,\dots$,
satisfy a closed set of infinite number of quadratic relations.
It is interesting to understand how duality will be extended to the case of superalgebras.
We expect to report on quadratic relations and duality
for more general deformed $W$-algebras
${\mathcal W}_{x,r}({\mathfrak g})$ associated with
affine Lie algebras and affine Lie superalgebras in the near future.

\appendix

\section{Normal ordering rules}
\label{appendix:normal}

We list the normal ordering rules.
For operators $V(z)$ and $W(w)$ we use the notation
\begin{gather}
V(z)W(w)=\langle V(z)W(w) \rangle {:}V(z)W(w){:}
\label{def:coupling}
\end{gather}
and write down only the part $\langle V(z)W(w) \rangle$
in the formulas below. Using the standard formula
\begin{gather*}
{\rm e}^A{\rm e}^B={\rm e}^{[A,B]}{\rm e}^B{\rm e}^A,\qquad [[A,B],A]=0\quad \text{and}\quad [[A,B],B]=0,
\end{gather*}
we obtain the normal ordering rules.

\subsection[\protect{$A\_\{i\}(z)$} and \protect{$S\_\{i\}(z)$}]
{$\boldsymbol{A_i(z)}$ and $\boldsymbol{S_i(z)}$}

\begin{gather}
\langle A_i(z_1)A_i(z_2)\rangle
=\bigg(\Delta\bigg(\frac{xz_2}{z_1}\bigg)
\Delta\bigg(\frac{x^{-1}z_2}{z_1}\bigg)\bigg)^{-1},\qquad
1\leq i \leq N-1,\nonumber
\\
\langle A_N(z_1)A_N(z_2)\rangle
=\Delta\bigg(\frac{z_2}{z_1}\bigg)
\bigg(\Delta\bigg(\frac{x z_2}{z_1}\bigg)
\Delta\bigg(\frac{x^{-1}z_2}{z_1}\bigg)\bigg)^{-1},\nonumber
\\
\langle A_i(z_1)A_j(z_2)\rangle
=\Delta\bigg(\frac{z_2}{z_1}\bigg),\qquad
|i-j|=1,\quad 1\leq i, j \leq N,\nonumber
\\
\langle A_i(z_1)A_j(z_2)\rangle=1,\qquad
|i-j|\geq 2,\quad 1\leq i, j \leq N,\label{appendix:A}
\\
\langle S_i(z_1)S_i(z_2)\rangle
=z_1^{\frac{2(r-1)}{r}}\bigg(1-\frac{z_2}{z_1}\bigg)
\frac{\big(x^2z_2/z_1;x^{2r}\big)_\infty}{\big(x^{2r-2}z_2/z_1;x^{2r}\big)_\infty},\qquad
1\leq i \leq N-1,\nonumber
\\
\langle S_N(z_1)S_N(z_2)\rangle
=z_1^{\frac{r-1}{r}}\bigg(1-\frac{z_2}{z_1}\bigg)
\frac{(x^2z_2/z_1;x^{2r})_\infty \big(x^{2r-2}z_2/z_1;x^{2r}\big)_\infty}{
\big(xz_2/z_1;x^{2r}\big)_\infty \big(x^{2r-1}z_2/z_1;x^{2r}\big)_\infty},\nonumber
\\
\langle S_i(z_1)S_j(z_2)\rangle
=z_1^{-\frac{r-1}{r}}\frac{\big(x^{2r-1}z_2/z_1;x^{2r}\big)_\infty}{\big(xz_2/z_1;x^{2r}\big)_\infty},\qquad
|i-j|=1,\quad 1\leq i,j \leq N,\nonumber
\\
\langle S_i(z_1)S_j(z_2)\rangle=1,\qquad
|i-j|\geq 2,\quad 1\leq i, j \leq N,
\label{appendix:S}
\\
\langle A_i(z_1)S_i(z_2)\rangle=x^{-4(r-1)}
\frac{\big(1-x^r\frac{z_2}{z_1}\big)\big(1-x^{r-2}\frac{z_2}{z_1}\big)}{
\big(1-x^{-r}\frac{z_2}{z_1}\big)\big(1-x^{2-r}\frac{z_2}{z_1}\big)},\qquad
1\leq i \leq N-1,\nonumber
\\
\langle A_N(z_1)S_N(z_2)\rangle=x^{-2(r-1)}\frac{
\big(1-x^r\frac{z_2}{z_1}\big)\big(1-x^{r-2}\frac{z_2}{z_1}\big)
\big(1-x^{1-r}\frac{z_2}{z_1}\big)}{\big(1-x^{-r}\frac{z_2}{z_1}\big)
\big(1-x^{2-r}\frac{z_2}{z_1}\big)\big(1-x^{r-1}\frac{z_2}{z_1}\big)},
\\
\langle A_i(z_1)S_j(z_2)\rangle
=x^{2(r-1)}\frac{\big(1-x^{1-r}\frac{z_2}{z_1}\big)}{\big(1-x^{r-1}\frac{z_2}{z_1}\big)},\qquad
|i-j|=1,\quad 1\leq i,j \leq N,\nonumber
\\
\langle A_i(z_1)S_j(z_2)\rangle=1,\qquad
|i-j|\geq 2,\quad 1\leq i,j \leq N,
\label{appendix:AS}
\\
\langle S_i(z_1)A_i(z_2)\rangle=\frac{\big(1-x^{2-r}\frac{z_2}{z_1}\big)}{
\big(1-x^{r}\frac{z_2}{z_1}\big)\big(1-x^{r-2}\frac{z_2}{z_1}\big)},\qquad
1\leq i \leq N-1,\nonumber
\\
\langle S_N(z_1)A_N(z_2)\rangle=\frac{
\big(1-x^{-r}\frac{z_2}{z_1}\big)\big(1-x^{2-r}\frac{z_2}{z_1}\big)
\big(1-x^{r-1}\frac{z_2}{z_1}\big)}{
\big(1-x^{r}\frac{z_2}{z_1}\big)\big(1-x^{r-2}\frac{z_2}{z_1}\big)
\big(1-x^{1-r}\frac{z_2}{z_1}\big)},\nonumber
\\
\langle S_j(z_1)A_i(z_2)\rangle=\frac{\big(1-x^{r-1}\frac{z_2}{z_1}\big)}{
\big(1-x^{1-r}\frac{z_2}{z_1}\big)},\qquad |i-j|=1,\quad 1\leq i,j \leq N,\nonumber
\\
\langle S_j(z_1)A_i(z_2)\rangle=1,\qquad |i-j|\geq 2,\quad 1\leq i,j \leq N.
\label{appendix:SA}
\end{gather}

\subsection[\protect{$Y\_\{i\}(z)$}, \protect{$A\_\{i\}(z)$} and \protect{$S\_\{i\}(z)$}]
{$\boldsymbol{Y_i(z)}$, $\boldsymbol{A_i(z)}$ and $\boldsymbol{S_i(z)}$}
The symmetric matrix $I(m)=(I_{i,j}(m))_{i,j=1}^N$ is the inverse matrix of $B(m)$.
The elements $I_{i,j}(m)=I_{j,i}(m)$, $1\leq i \leq j \leq N$, are written as
\begin{gather*}
I_{i, j}(m)=\frac{1}{[(N\!+\!1)m]_x-[Nm]_x}
\times \begin{cases}
[(N+1-j)m]_x-[(N-j)m]_x, &i=1,\ 1\leq j \leq N,
\\[1ex]
(-1)^{N-j+i}\displaystyle\sum_{k=i-1}^{N-j+i}(-1)^k [km]_x,
&2\leq i \leq j \leq N-1,
\\[1ex]
[i m]_x,&1\leq i \leq N,\ j=N.
\end{cases}
\label{appendix:I(m)}
\end{gather*}
The generators $y_i(m)$, $1\leq i \leq N$, are written as
\begin{gather}
y_i(m)=\sum_{j=1}^N I_{i,j}(m)a_j(m),\qquad
Q_i^y=\sum_{j=1}^N I_{i,j}(0)Q_j.
\label{eqn:y(m)}
\end{gather}
From (\ref{def:A}), (\ref{def:Y}) and (\ref{appendix:I(m)}) we obtain
\begin{gather}
\langle Y_1(z_1)Y_1(z_2)\rangle=f_{1,1}\bigg(\frac{z_2}{z_1}\bigg)^{-1},\nonumber
\\
\langle Y_1(z_1)A_1(z_2)\rangle =\Delta\bigg(\frac{z_2}{z_1}\bigg)^{-1},\qquad
\langle Y_1(z_1)A_i(z_2)\rangle=1,\qquad 2\leq i \leq N,\nonumber
\\
\langle A_1(z_1)Y_1(z_2)\rangle=\Delta\bigg(\frac{z_2}{z_1}\bigg)^{-1},\qquad
\langle A_i(z_1)Y_1(z_2)\rangle=1,\qquad 2\leq i \leq N,\nonumber
\\
\langle Y_1(z_1)S_1(z_2)\rangle =x^{-2(r-1)}\frac{\big(1-x^{r-1}\frac{z_2}{z_1}\big)}{\big(1-x^{1-r}\frac{z_2}{z_1}\big)},\qquad
\langle Y_1(z_1)S_i(z_2)\rangle=1,\qquad 2\leq i \leq N,
\nonumber
\\
\langle S_1(z_1)Y_1(z_2)\rangle
=\frac{\big(1-x^{1-r}\frac{z_2}{z_1}\big)}{\big(1-x^{r-1}\frac{z_2}{z_1}\big)},\qquad
\langle S_i(z_1)Y_1(z_2)\rangle =1,\qquad 2\leq i \leq N.
\label{appendix:Y}
\end{gather}

\section{Exchange relations}\label{appendix:exchange}
In this appendix we list the exchange relations.

\subsection[\protect{$Lambda\_\{i\}(z)$}]
{$\boldsymbol{\Lambda_i(z)}$}
We give the exchange relations of $\Lambda_j(z)$ and $\overrightarrow{\Lambda}_{\Omega_i}(z)$,
which are obtained from (\ref{eqn:Lambda}).
We set $s \in J_N=\{1, 2, \dots, N, 0, \overline{N}, \dots, \overline{2}, \overline{1}\}$.
For an element $s \in J_N$ and a subset
$\Omega_i=\{s_1, s_2, \dots, s_i\} \subset J_N$ with $s_1 \prec s_2 \prec \cdots \prec s_i$, we calculate
\begin{gather*}
X_{\Omega_i,s}(z_1,z_2)=f_{1,i}\bigg(\frac{z_2}{z_1}\bigg)\Lambda_s(z_1) \overrightarrow{\Lambda}_{\Omega_i}(z_2)-
f_{i,1}\bigg(\frac{z_1}{z_2}\bigg)\overrightarrow{\Lambda}_{\Omega_i}(z_2)\Lambda_s(z_1).
\end{gather*}
$\bullet$~In the case of $s, \overline{s} \notin \Omega_i$, we obtain
\begin{gather}
{X_{\Omega_i,s}(z_1,z_2)}=
c(x,r){:}\Lambda_s(z_1)\overrightarrow{\Lambda}_{\Omega_i}(z_2){:}
\bigg(\delta\bigg(\frac{x^{-i-3+2k} z_2}{z_1}\bigg)-
\delta\bigg(\frac{x^{-i-1+2k} z_2}{z_1}\bigg)\bigg).
\label{exchange:L1}
\end{gather}
Here we set $k$, $1\leq k \leq i+1$, by
$k=\begin{cases}
1&{\rm if}\quad s\prec s_1,\\
q&{\rm if}\quad s_{q-1}\prec s \prec s_q,\quad 2 \leq q \leq i,\\
i+1&{\rm if}\quad s_i \prec s.
\end{cases}$

\noindent
$\bullet$
In the case of $s \in \Omega_i$ and $\overline{s} \notin \Omega_i$,
we obtain
\begin{gather}
f_{1,i}\bigg(\frac{z_2}{z_1}\bigg)\Lambda_s(z_1)\overrightarrow{\Lambda}_{\Omega_i}(z_2)-
f_{i,1}\bigg(\frac{z_1}{z_2}\bigg)\overrightarrow{\Lambda}_{\Omega_i}(z_2)\Lambda_s(z_1)=0.
\label{exchange:L2}
\end{gather}
$\bullet$
In the case of $s$, $\overline{s} \in \Omega_i$ and $s=n$,
$1\leq n \leq N$, we obtain
\begin{align}
X_{\Omega_i,s}(z_1,z_2)={}&
c(x,r){:}\Lambda_n(z_1)\overrightarrow{\Lambda}_{\Omega_i}(z_2){:}\nonumber
\\
&\times\bigg(\delta\bigg(\frac{x^{-2N-i+2n+2l-4}z_2}{z_1}\bigg)-
\delta\bigg(\frac{x^{-2N-i+2n+2l-2} z_2}{z_1}\bigg)\bigg).
\label{exchange:L3}
\end{align}
Here we set $k$, $l$, $1\leq k<l \leq i$, by $s=n=s_k$ and
$\overline{s}=\overline{n}=s_l$.

\noindent
$\bullet$
In the case of $s$, $\overline{s} \in \Omega_i$ and $s=\overline{n}$,
$1\leq n \leq N$, we obtain
\begin{align}
X_{\Omega_i,s}(z_1,z_2)={}&
c(x,r){:}\Lambda_{\overline{n}}(z_1)\overrightarrow{\Lambda}_{\Omega_i}(z_2){:}\nonumber
\\
&\times\bigg(\delta\bigg(\frac{x^{2N-i-2n+2k}z_2}{z_1}\bigg)-
\delta\bigg(\frac{x^{2N-i-2n+2k+2} z_2}{z_1}\bigg)\bigg).
\label{exchange:L4}
\end{align}
Here we set $k$, $l$, $1\leq k<l \leq i$,
by $\overline{s}=n=s_k$ and $s=\overline{n}=s_l$.

\noindent
$\bullet$
In the case of $s=0 \in \Omega_i$, we obtain
\begin{gather}
{X_{\Omega_i,s}(z_1,z_2)}=
c(x,r){:}\Lambda_{0}(z_1)\overrightarrow{\Lambda}_{\Omega_i}(z_2){:}
\bigg(\delta\bigg(\frac{x^{-i-2+2k}z_2}{z_1}\bigg)-
\delta\bigg(\frac{x^{-i+2k} z_2}{z_1}\bigg)\bigg).
\label{exchange:L5}
\end{gather}
Here we set $k$, $1\leq k \leq i$, by $s_k=0$.

\noindent
$\bullet$
In the case of $s \notin \Omega_i$ and $ \overline{s} \in \Omega_i$
and $s=n$, $1\leq n \leq N$, we obtain
\begin{align}
X_{\Omega_i,s}(z_1,z_2)={}&c(x,r)\Delta\big(x^{2(l-k+n-N)}\big)
{:}\Lambda_{n}(z_1)\overrightarrow{\Lambda}_{\Omega_i}(z_2){:}\notag
\\
&\phantom{+}\times\bigg(\delta\bigg(\frac{x^{-i+2k-3}z_2}{z_1}\bigg)-
\delta\bigg(\frac{x^{-2N+2n+2l-i-2} z_2}{z_1}\bigg)\bigg)
\notag
\\
&+c(x,r)\Delta\big(x^{2(l-k+n-N-1)}\big)
{:}\Lambda_{n}(z_1)\overrightarrow{\Lambda}_{\Omega_i}(z_2){:}\notag
\\
&\phantom{+}\times\bigg({-}\delta\bigg(\frac{x^{-i+2k-1}z_2}{z_1}\bigg)+
\delta\bigg(\frac{x^{-2N+2n+2l-i-4} z_2}{z_1}\bigg)\bigg).
\label{exchange:L6}
\end{align}
Here we set $k$, $l$, $1 \leq k \leq l \leq i$, by $s_l=\overline{s}=\overline{n}$ and
$k=\begin{cases}
1&{\rm if}\quad s=n\prec s_1,\\
q&{\rm if}\quad s_{q-1}\prec s=n \prec s_q,\quad 2 \leq q \leq i.
\end{cases}$

\noindent
$\bullet$
In the case of $s \notin \Omega_i$ and $\overline{s} \in \Omega_i$ and
$s=\overline{n}$, $1\leq n \leq N$, we obtain
\begin{align}
X_{\Omega_i,s}(z_1,z_2)={}&
c(x,r)\Delta\big(x^{2(l-k+n-N-1)}\big)
{:}\Lambda_{\overline{n}}(z_1)\overrightarrow{\Lambda}_{\Omega_i}(z_2){:}\notag
\\
&\phantom{+}\times\bigg(\delta\bigg(\frac{x^{-i+2l-1}z_2}{z_1}\bigg)-
\delta\bigg(\frac{x^{2N-2n-i+2k+2} z_2}{z_1}\bigg)\bigg)\notag
\\
&+c(x,r)\Delta\big(x^{2(l-k+n-N)}\big)
{:}\Lambda_{\overline{n}}(z_1)\overrightarrow{\Lambda}_{\Omega_i}(z_2){:}\notag
\\
&\phantom{+}\times\bigg({-}\delta\bigg(\frac{x^{-i+2l+1}z_2}{z_1}\bigg)+
\delta\bigg(\frac{x^{2N-2n-i+2k} z_2}{z_1}\bigg)\bigg).
\label{exchange:L7}
\end{align}
Here we set $k$, $l$, $1\leq k \leq l \leq i$, by
$s_k=\overline{s}=n$ and $l=\begin{cases}
q&{\rm if}\ \ s_{q} \prec s=\overline{n} \prec s_{q+1},\ \ 1 \leq q \leq i-1,\\
i&{\rm if}\ \ s_i \prec s=\overline{n}.
\end{cases}$

\subsection[\protect{$S\_\{i\}(z)$}]{$\boldsymbol{S_i(z)}$}
From (\ref{appendix:S}) we obtain
\begin{gather}
S_i(z_1)S_i(z_2)=-\frac{[u_2-u_1+1]}{[u_1-u_2+1]}S_i(z_2)S_i(z_1),\qquad
1\leq i \leq N-1,\nonumber
\\
S_N(z_1)S_N(z_2)=-\frac{[u_1-u_2+\frac{1}{2}][u_2-u_1+1]}{
[u_2-u_1+\frac{1}{2}][u_1-u_2+1]}S_N(z_2)S_N(z_1),\nonumber
\\
S_i(z_1)S_j(z_2)=\frac{[u_1-u_2+\frac{1}{2}]}{[u_2-u_1+\frac{1}{2}]}S_j(z_2)S_i(z_1),\qquad
|i-j|=1,\quad 1\leq i,j \leq N,\nonumber
\\
S_i(z_1)S_j(z_2)=S_j(z_2)S_i(z_1),\qquad
|i-j|\geq 2,\quad 1\leq i,j \leq N.\label{exchange:S}
\end{gather}
Here we set $z_i=x^{2u_i}$, $i=1,2$ and
$[u]=x^{\frac{u^2}{r}-2u}\Theta_{x^{2r}}(z)$.

\subsection[\protect{$Lambda\_\{i\}(z)$} and \protect{$S\_\{i\}(z)$}]
{$\boldsymbol{\Lambda_i(z)}$ and $\boldsymbol {S_i(z)}$}
From (\ref{appendix:AS}), (\ref{appendix:SA}) and (\ref{appendix:Y}) we obtain
\begin{gather}
[\Lambda_k(z_1),S_k(z_2)]=\big(x^{-2r+2}-1\big){:}\Lambda_k(z_1)S_k(z_2){:}
\delta\bigg(\frac{x^{k-r}z_2}{z_1}\bigg),\qquad
1\leq k \leq N,\nonumber
\\
[\Lambda_{k+1}(z_1),S_k(z_2)]=\big(x^{2r-2}-1\big){:}\Lambda_{k+1}(z_1)S_k(z_2){:}
\delta\bigg(\frac{x^{k+r}z_2}{z_1}\bigg),\qquad
1\leq k \leq N-1,\nonumber
\\
[\Lambda_{\overline{k}}(z_1),S_k(z_2)]=\big(x^{2r-2}-1\big)
{:}\Lambda_{\overline{k}}(z_1)S_k(z_2){:}
\delta\bigg(\frac{x^{2N+1-k+r}z_2}{z_1}\bigg),\qquad 1\leq k \leq N,
\nonumber
\\
[\Lambda_{\overline{k+1}}(z_1),S_k(z_2)]=\big(x^{-2r+2}-1\big)
{:}\Lambda_{\overline{k+1}}(z_1)S_k(z_2){:}
\delta\bigg(\frac{x^{{2N+1-k-r}}z_2}{z_1}\bigg),\qquad
1\leq k \leq N-1,\nonumber
\\
[\Lambda_0(z_1),S_N(z_2)]={\big(x-x^{-1}\big)}
\frac{[r-1]_x\big[\frac{1}{2}\big]_x}{\big[r-\frac{1}{2}\big]_x}
\bigg(\delta\bigg(\frac{x^{r+N}z_2}{z_1}\bigg)
-\delta\bigg(\frac{x^{-r+N+1}z_2}{z_1}\bigg)\bigg)\nonumber
\\ \hphantom{[\Lambda_0(z_1),S_N(z_2)]=}
{}\times{:}\Lambda_0(z_1)S_N(z_2){:}.
\label{exchange:Lambda-S}
\end{gather}
Other commutators on the type $[\Lambda_i(z_1), S_k(z_2)]$
that are used in the proof of Lemma \ref{lemma:4-9} are zeroes.

\subsection*{Acknowledgements}

The author would like to thank Professor Michio Jimbo
for his kind and courteous advice.
The author would like to thank the referees
for their careful reading of the paper and their helpful comments.
The author would like to thank Editage ({\it www.editage.com}) for English language editing.
This work was supported by
JSPS KAKENHI (Grant Number JP19K03509).


\pdfbookmark[1]{References}{ref}

\LastPageEnding

\end{document}